 \documentclass[12pt]{amsart}
 \usepackage{amsmath,amsfonts,mathrsfs,amsthm,amssymb}
 \usepackage{multirow}
 \usepackage{framed}
 \usepackage{youngtab}  \Yvcentermath1
 \usepackage{enumerate}
 \usepackage{lscape}
 \usepackage[usenames,dvipsnames,svgnames,table]{xcolor}

 \usepackage{hyperref}
 \usepackage{tikz}
 \usepackage{ifthen}
 \usepackage{xspace}
 \input{xstring}

 \newcommand\id{\mathrm{id}}
 \newcommand\tr{\mathrm{tr}}
 
 \newcommand\tGr{\mathrm{Gr}}

 \newcommand\wt{\widetilde}
 \newcommand\tMat{\mathrm{Mat}}

 \renewcommand\l{\left}
 \renewcommand\r{\right}

 \newcommand\qRa{\quad\Rightarrow\quad}
 \newcommand\mat[2]{\left( \begin{array}{#1} #2 \end{array} \right)}

 \newcommand\bop{\bigoplus}
 \newcommand\op{\oplus}
 \newcommand\ot{\otimes}

 \newcommand\fann{\mathfrak{ann}}

 \newcommand\bu{{\mathbf u}}

 \newcommand\0{{\bf 0}}

\newcommand\bA{{\bf A}}
\newcommand\bB{{\bf B}}
\newcommand\bC{{\bf C}}
\newcommand\bD{{\bf D}}

 \newcommand\bS{{\mathbf S}} 
 \newcommand\bT{{\mathbf T}}
 \newcommand\bU{{\mathbf U}}
 \newcommand\bV{{\mathbf V}}
 \newcommand\bW{{\mathbf W}}  
 \newcommand\bX{{\mathbf X}} 
 \newcommand\bY{{\mathbf Y}}
 \newcommand\bZ{{\mathbf Z}}
 
 \newcommand\rnk{\mathrm{rank}}

 \newcommand\im{\mathrm{im}}

 \newcommand\fa{{\mathfrak a}} 
 \newcommand\fb{{\mathfrak b}}

 \newcommand\ff{{\mathfrak f}}
 \newcommand\fg{{\mathfrak g}}
 \newcommand\fgl{\mathfrak{gl}}
 \newcommand\fh{{\mathfrak h}}

 \newcommand\fk{{\mathfrak k}}
 
 \newcommand\fm{{\mathfrak m}}

 \newcommand\fp{{\mathfrak p}}
 \newcommand\fq{{\mathfrak q}}
 
 \newcommand\fs{{\mathfrak s}}
 \newcommand\fsl{\mathfrak{sl}}
 \newcommand\fso{\mathfrak{so}}
 \newcommand\fsp{\mathfrak{sp}}

 \newcommand\fv{{\mathfrak v}}

 \newcommand\fz{{\mathfrak z}}

 \newcommand\fS{{\mathfrak S}}
 
 \newcommand\fU{{\mathfrak U}}

 \newcommand\fX{{\mathfrak X}}

 \newcommand\cA{{\mathcal A}}
 \newcommand\cB{{\mathcal B}}
 \newcommand\cC{{\mathcal C}}
 \newcommand\cD{{\mathcal D}}
 \newcommand\cE{{\mathcal E}}
 \newcommand\cF{{\mathcal F}}
 \newcommand\cG{{\mathcal G}}
 \newcommand\cH{{\mathcal H}}

 \newcommand\cL{{\mathcal L}}
 
 \newcommand\cN{{\mathcal N}}
 \newcommand\cO{{\mathcal O}}
 \newcommand\cP{{\mathcal P}}
 
 \newcommand\cR{{\mathcal R}}
 \newcommand\cS{{\mathcal S}}
 \newcommand\cT{{\mathcal T}}
 
 \newcommand\cV{{\mathcal V}}

 \newcommand\bbC{{\mathbb C}}

 \newcommand\bbP{{\mathbb P}}
 
 \newcommand\bbR{{\mathbb R}}
 \newcommand\bbS{{\mathbb S}}
 
 \newcommand\bbU{{\mathbb U}}
 \newcommand\bbV{{\mathbb V}}
 \newcommand\bbW{{\mathbb W}}

 \newcommand\bbZ{{\mathbb Z}}

 \newcommand\half{\frac{1}{2}}

 \renewcommand\a{\alpha}
 \renewcommand\b{\beta}
 \newcommand\tspan{\mathrm{span}}

 \newcommand\Ben{\begin{enumerate}}
 \newcommand\Een{\end{enumerate}}
 \newcommand\Bex{\begin{example}}
 \newcommand\Eex{\end{example}}

 \newcommand\ra{\rightarrow}

 \def\inj{\hookrightarrow}

 \newcommand\fder{\mathfrak{der}}
 \newcommand\GL{\operatorname{GL}}

\def\tO{\text{O}}
 \newcommand\SL{\mathrm{SL}}

 \newcommand\CO{\mathrm{CO}}
 \newcommand\SO{\mathrm{SO}}
 \newcommand\SU{\text{SU}}
 \newcommand\Sp{\mathrm{Sp}}
 \newcommand\fsu{\mathfrak{su}}
 \newcommand\ad{{\rm ad}}
 
 \newcommand\Ad{{\rm Ad}}
 \newcommand\Aut{\text{Aut}}
 \def\aut{\mathfrak{aut}}

 \def\assoc/{associative}
 \def\arb/{arbitrary}
 \def\btw/{between}
 \def\coeff/{coefficient}
 \def\cohom/{cohomology}
 \def\coord/{coordinate}
 \def\coordsys/{coordinate system}
 \def\cpt/{compact}
 \def\cred/{completely reducible}
 \def\cts/{continuous}
 \def\dga/{differential-graded algebra}
 \def\dR/{de Rham}
 \def\Euc/{Euclidean} 
 \def\grp/{group}
 \def\hom/{homomorphism}
 \def\inv/{invariant}
 \def\iso/{isomorphism}
 \def\La/{Lie algebra}
 \def\Lag/{Lagrangian Grassmannian}
 \def\LG/{Lie group}
 \def\MA/{Monge--Amp\`ere}
 \def\MC/{Maurer--Cartan}
 \def\lintr/{linear transformation} 
 \def\mfld/{manifold}
 \def\nb/{normal bundle}
 \def\nbd/{neighbourhood}
 \def\nondeg/{non-degenerate}
 \def\posdef/{positive definite}
 \def\pu/{partition of unity}
 \def\rep/{representation}
 \def\Riem/{Riemannian}
 \def\sg/{subgroup}
 \def\ss/{semi-simple}
 \def\inv/{invariant}
 \def\irr/{irreducible}
 \def\Jacid/{Jacobi identity}
 \def\li/{linearly independent}
 \def\nd/{nowhere dependent}
 \def\nz/{nowhere zero}
 \def\on/{orthonormal}
 \def\onb/{\on/ basis}
 \def\orc/{\orth/ complement}
 \def\orth/{orthogonal}
 \def\orp/{\orth/ projection}
 \def\pde/{partial differential equation}
 \def\resp/{respectively}
 \def\seq/{sequence}
 \def\std/{standard}
 \def\SW/{Stiefel-Whitney}
 \def\uc/{universal cover}
 \def\vb/{vector bundle}
 \def\vf/{vector field}
 \def\vs/{vector space}
 \def\wrt/{with respect to}
 
 \renewcommand\mod{\,{\rm mod}\ }
 \newcommand\qbox[1]{\quad\mbox{#1}\quad}
 \renewcommand\dim{{\rm dim}}
 
\newcommand\myscale{0.8}

\newcommand\tcirc[3]{
	\ifthenelse{\equal{#1}{w}}{\filldraw[fill=white,draw=black] (#2) circle (0.075);}{}%
	\ifthenelse{\equal{#1}{b}}{\filldraw[black] (#2) circle (0.075);}{}%
	\draw (#2) ++(0,0.35) node {$#3$};
	}

\newcommand\tdots[1]{\draw (#1) ++(0.55,0) node {$\cdots$}}

\newcommand\bond[1]{\draw (#1) -- +(1,0)}

\newcommand\vbond[1]{\draw (#1) -- +(0,-1)}

\newcommand\diagbond[2]{
	\ifthenelse{\equal{#1}{u}}{
		\draw[semithick] (#2) -- +(0.5,0.865);
	}{}
	\ifthenelse{\equal{#1}{d}}{
		\draw[semithick] (#2) -- +(0.5,-0.865);
	}{}
	}

\newcommand\dbond[2]{
	\draw (#2) ++(0.03,0.03) -- +(0.94,0);
	\draw (#2) ++(0.03,-0.03) -- +(0.94,0);
	\ifthenelse{\equal{#1}{r}}{
		\draw[semithick] (#2) ++(0.6,0) ++(-0.15,0.2) -- ++(0.15,-0.2) -- +(-0.15,-0.2);
	}{}
	\ifthenelse{\equal{#1}{l}}{
		\draw[semithick] (#2) ++(0.45,0) ++(0.15,0.2) -- ++(-0.15,-0.2) -- +(0.15,-0.2);
	}{}
	}

\newcommand\tbond[2]{
	\draw (#2)  -- +(1,0);
	\draw (#2) ++(0.05,0.06) -- +(0.9,0);
	\draw (#2) ++(0.05,-0.06) -- +(0.9,0);
	\ifthenelse{\equal{#1}{r}}{
		\draw[semithick] (#2) ++(0.6,0) ++(-0.15,0.2) -- ++(0.15,-0.2) -- +(-0.15,-0.2);
	}{}
	\ifthenelse{\equal{#1}{l}}{
		\draw[semithick] (#2) ++(0.45,0) ++(0.15,0.2) -- ++(-0.15,-0.2) -- +(0.15,-0.2);
	}{}
	}

\newcommand\tcross[2]{
	\draw (#1) ++(0,0.35) node {$#2$};
	\draw[semithick] (#1) ++(-0.15,-0.15)-- +(0.3,0.3);
	\draw[semithick] (#1) ++(-0.15,0.15)-- +(0.3,-0.3);
	}

\newcommand\tsquare[2]{
		\draw[semithick,color=blue] (#1) ++(-0.15,-0.15) rectangle ++(0.3,0.3);
		\tcross{#1}{#2};
		}

\newcommand\tstar[2]{
	\draw[color=red] (#1) node {\Large$*$};
	\draw (#1) ++(0,0.35) node {$#2$};
	}

\newcommand\DDnode[3]{
\ifthenelse{\equal{#1}{w}}{\tcirc{w}{#2}{#3}}{}		
\ifthenelse{\equal{#1}{b}}{\tcirc{b}{#2}{#3}}{}		
\ifthenelse{\equal{#1}{x}}{\tcross{#2}{#3}}{}		
\ifthenelse{\equal{#1}{s}}{\tstar{#2}{#3}}{}		
\ifthenelse{\equal{#1}{q}}{\tsquare{#2}{#3}}{}		
}

\newcommand\Aone[2]{
 \begin{tiny}
 \begin{tikzpicture}[scale=\myscale,baseline=-3pt]
 
 \StrChar{#1}{1}[\nodetype];
 \DDnode{\nodetype}{0,0}{#2};
 \useasboundingbox (-.4,-.2) rectangle (0.4,0.55); 
 \end{tikzpicture}
 \end{tiny}
 }

\newcommand\Atwo[2]{
 \begin{tiny}
 \begin{tikzpicture}[scale=\myscale,baseline=-3pt]
 
 \bond{0,0};	

 \StrBefore{#2}{,}[\labelone]
 \StrBehind{#2}{,}[\labeltwo]
 
 \StrChar{#1}{1}[\nodetype];
 \DDnode{\nodetype}{0,0}{\labelone};
 \StrChar{#1}{2}[\nodetype];
 \DDnode{\nodetype}{1,0}{\labeltwo};
 \useasboundingbox (-.4,-.2) rectangle (1.4,0.55); 
 \end{tikzpicture}
 \end{tiny}
 }
 
 \newcommand\Athree[2]{
 \begin{tiny}
 \begin{tikzpicture}[scale=\myscale,baseline=-3pt]

 \bond{0,0};		
 \bond{1,0};		

 \StrBefore{#2}{,}[\labelone]
 \StrBetween[1,2]{#2}{,}{,}[\labeltwo]
 \StrBehind[2]{#2}{,}[\labelthree]

 \StrChar{#1}{1}[\nodetype];
 \DDnode{\nodetype}{0,0}{\labelone};
 \StrChar{#1}{2}[\nodetype];
 \DDnode{\nodetype}{1,0}{\labeltwo};
 \StrChar{#1}{3}[\nodetype];
 \DDnode{\nodetype}{2,0}{\labelthree};
 \useasboundingbox (-.4,-.2) rectangle (2.4,0.55); 
 \end{tikzpicture}
 \end{tiny}
 }

\newcommand\Btwo[2]{
 \begin{tiny}
 \begin{tikzpicture}[scale=\myscale,baseline=-3pt]
 
 \dbond{r}{0,0};	
 
 \StrBefore{#2}{,}[\labelone]
 \StrBehind{#2}{,}[\labeltwo]
 
 \StrChar{#1}{1}[\nodetype];
 \DDnode{\nodetype}{0,0}{\labelone};
 \StrChar{#1}{2}[\nodetype];
 \DDnode{\nodetype}{1,0}{\labeltwo};
 \useasboundingbox (-.4,-.2) rectangle (1.4,0.55); 
 \end{tikzpicture}
 \end{tiny}
 }
 
\newcommand\Bthree[2]{
 \begin{tiny}
 \begin{tikzpicture}[scale=\myscale,baseline=-3pt]
 \bond{0,0};		
 \dbond{r}{1,0};		

 \StrChar{#1}{1}[\nodetype];
 \DDnode{\nodetype}{0,0}{\StrBefore{#2}{,}};
 \StrChar{#1}{2}[\nodetype];
 \DDnode{\nodetype}{1,0}{\StrBetween[1,2]{#2}{,}{,}};
 \StrChar{#1}{3}[\nodetype];
 \DDnode{\nodetype}{2,0}{\StrBehind[2]{#2}{,}};
 \end{tikzpicture}
 \end{tiny}
 }

\newcommand\Cthree[2]{
 \begin{tiny}
 \begin{tikzpicture}[scale=\myscale,baseline=-3pt]
 \bond{0,0};		
 \dbond{l}{1,0};		

 \StrChar{#1}{1}[\nodetype];
 \DDnode{\nodetype}{0,0}{\StrBefore{#2}{,}};
 \StrChar{#1}{2}[\nodetype];
 \DDnode{\nodetype}{1,0}{\StrBetween[1,2]{#2}{,}{,}};
 \StrChar{#1}{3}[\nodetype];
 \DDnode{\nodetype}{2,0}{\StrBehind[2]{#2}{,}};
 \end{tikzpicture}
 \end{tiny}
 }

\newcommand\Gdd[2]{
 \begin{tiny}
 \begin{tikzpicture}[scale=\myscale,baseline=-3pt]
 \tbond{l}{0,0};	

 \StrBefore{#2}{,}[\labelone]
 \StrBehind{#2}{,}[\labeltwo]
 \StrChar{#1}{1}[\nodetype];

 \DDnode{\nodetype}{0,0}{\labelone};
 \StrChar{#1}{2}[\nodetype];
 \DDnode{\nodetype}{1,0}{\labeltwo};
 \useasboundingbox (-.4,-.2) rectangle (1.4,0.55); 
 \end{tikzpicture}
 \end{tiny}
 }
 
\newcommand\Fdd[2]{
 \begin{tiny}
 \begin{tikzpicture}[scale=\myscale,baseline=-3pt]
 \bond{0,0};
 \dbond{r}{1,0};
 \bond{2,0};

 \StrChar{#1}{1}[\nodetype];
 \DDnode{\nodetype}{0,0}{\StrBefore{#2}{,}};
 \StrChar{#1}{2}[\nodetype];
 \DDnode{\nodetype}{1,0}{\StrBetween[1,2]{#2}{,}{,}};
 \StrChar{#1}{3}[\nodetype];
 \DDnode{\nodetype}{2,0}{\StrBetween[2,3]{#2}{,}{,}};
 \StrChar{#1}{4}[\nodetype];
 \DDnode{\nodetype}{3,0}{\StrBehind[3]{#2}{,}};
 \useasboundingbox (-.4,-.2) rectangle (3.4,0.55); 
 \end{tikzpicture}
 \end{tiny}
 }
 
\newcommand\Edd[2]{
 \begin{tiny}
 \begin{tikzpicture}[scale=\myscale,baseline=-3pt]
 \foreach \x in {0,1,2,3} {
	\bond{\x,0};
 }
 \vbond{2,0};
 
 \StrLen{#1}[\Ernk]
 
 \StrChar{#1}{1}[\nodetype];
 \DDnode{\nodetype}{0,0}{\StrBefore{#2}{,}};
 \StrChar{#1}{2}[\nodetype];
 \DDnode{\nodetype}{2,-1}{\StrBetween[1,2]{#2}{,}{,}};
 \StrChar{#1}{3}[\nodetype];
 \DDnode{\nodetype}{1,0}{\StrBetween[2,3]{#2}{,}{,}};
 \StrChar{#1}{4}[\nodetype];
 \DDnode{\nodetype}{2,0}{\StrBetween[3,4]{#2}{,}{,}};
 \StrChar{#1}{5}[\nodetype];
 \DDnode{\nodetype}{3,0}{\StrBetween[4,5]{#2}{,}{,}};
 \StrChar{#1}{6}[\nodetype];

 \ifthenelse{\equal{\Ernk}{6}}{
 		\DDnode{\nodetype}{4,0}{\StrBehind[5]{#2}{,}};
 		\useasboundingbox (-.4,-1.2) rectangle (4.4,0.55);
	}{}%
 
 \ifthenelse{\equal{\Ernk}{7}}{
 		\bond{4,0};
 		\DDnode{\nodetype}{4,0}{\StrBetween[5,6]{#2}{,}{,}};
		\StrChar{#1}{7}[\nodetype];
		\DDnode{\nodetype}{5,0}{\StrBehind[6]{#2}{,}};
 		\useasboundingbox (-.4,-1.2) rectangle (5.4,0.55);
	}{}%

 \ifthenelse{\equal{\Ernk}{8}}{
 		\bond{4,0};
 		\bond{5,0};
 		\DDnode{\nodetype}{4,0}{\StrBetween[5,6]{#2}{,}{,}};
		\StrChar{#1}{7}[\nodetype];
		\DDnode{\nodetype}{5,0}{\StrBetween[6,7]{#2}{,}{,}};
		\StrChar{#1}{8}[\nodetype];
		\DDnode{\nodetype}{6,0}{\StrBehind[7]{#2}{,}};
		\useasboundingbox (-.4,-1.2) rectangle (6.4,0.55);
	}{}%

 \end{tikzpicture}
 \end{tiny}
 }

 \newtheorem{appthm}{Theorem}[section]
 \newtheorem{thm}{Theorem}[subsection]
 \newtheorem{lemma}[thm]{Lemma}
 \newtheorem{cor}[thm]{Corollary}
 \newtheorem{prop}[thm]{Proposition}
 \newtheorem{defn}[thm]{Definition}
 \newtheorem{example}[thm]{Example}
 \newtheorem{ex}[subsection]{Example}
 \newtheorem{recipe}{Recipe}
 
 \theoremstyle{remark}
 \newtheorem{remark}[thm]{Remark}
 \numberwithin{equation}{section}

 \newcommand\euc{{\rm euc}}
 \newcommand\pp{{\rm\bf pp}}
 \newcommand\metric{{\rm g}}
 \newcommand\ppmetric[1]{\metric_\pp^{(#1)}}
 \newcommand\eucmetric[1]{\metric_\euc^{(#1)}}
 \newcommand\wW{\widetilde{W}}
 \newcommand\wC{\widetilde{C}}
 \newcommand\Kh{\kappa_H}
 
 \newcommand\p{\partial}
 \newcommand\rkg{\ell} 	
 \newcommand\sr{\sigma}		
 \newcommand\finf{\mathfrak{inf}}
 \renewcommand\ss{semisimple\xspace}
 \newcommand\prn{\operatorname{pr}}
 \newcommand\opn{\operatorname{op}}
 \newcommand\oE{\operatorname{E}} 
 \newcommand\diag{\operatorname{diag}}
 \newcommand\Jac{\operatorname{Jac}}
 \newcommand\ev{\operatorname{ev}}
 \newcommand\gr{\operatorname{gr}}
 \newcommand\Hom{\operatorname{Hom}}
 
 \newcommand\Lie{\operatorname{Lie}}
 \newcommand\Spec{\operatorname{Spec}}
 \newcommand\pmat[1]{{\tiny \begin{pmatrix} #1 \end{pmatrix}}}

 \newcommand\mycase[1]{\l\{ \begin{array}{@{}c@{\,\,}l@{}} #1 \end{array} \r.}
 \renewcommand\myscale{0.8}
 \newcommand\NPRC[1]{
 \begin{tiny}
 \begin{tikzpicture}[scale=\myscale,baseline=-3pt]

 \bond{0,0};		
 \bond{1,0};		
 \bond{2,0};		
 \bond{3,0};
 \dbond{l}{4,0};		
 
 \DDnode{q}{0,0}{0};
 \DDnode{x}{1,0}{-5};
 \DDnode{s}{2,0}{4};
 \DDnode{w}{3,0}{0};
 \DDnode{#1}{4,0}{0};
 \DDnode{w}{5,0}{0};

 \useasboundingbox (-.4,-.2) rectangle (5.2,0.55); 
 \end{tikzpicture}
 \end{tiny}
 }

\newcommand\projwts[2]{
 \begin{tiny}
 \begin{tikzpicture}[scale=\myscale,baseline=-3pt]
 \bond{0,0}; 
 \bond{1,0};
 \bond{2,0};
 \tdots{3,0};
 \bond{4,0};
 \StrChar{#1}{1}[\nodetype];
 \DDnode{\nodetype}{0,0}{\StrBefore{#2}{,}};
 \StrChar{#1}{2}[\nodetype];
 \DDnode{\nodetype}{1,0}{\StrBetween[1,2]{#2}{,}{,}};
 \StrChar{#1}{3}[\nodetype];
 \DDnode{\nodetype}{2,0}{\StrBetween[2,3]{#2}{,}{,}};
 \StrChar{#1}{4}[\nodetype];
 \DDnode{\nodetype}{3,0}{\StrBetween[3,4]{#2}{,}{,}};
 \StrChar{#1}{5}[\nodetype];
 \DDnode{\nodetype}{4,0}{\StrBetween[4,5]{#2}{,}{,}};
 \StrChar{#1}{6}[\nodetype];
 \DDnode{\nodetype}{5,0}{\StrBehind[5]{#2}{,}};
 \useasboundingbox (-.4,-.2) rectangle (5.4,0.55);
 \end{tikzpicture}
 \end{tiny}
 } 

 \newcommand\pathwts[2]{
 \begin{tiny}
 \begin{tikzpicture}[scale=\myscale,baseline=-3pt]
 \bond{0,0}; 
 \bond{1,0};
 \bond{2,0};
 \tdots{3,0};
 \bond{4,0};

 \StrChar{#1}{1}[\nodetype];
 \DDnode{\nodetype}{0,0}{\StrBefore{#2}{,}};
 \StrChar{#1}{2}[\nodetype];
 \DDnode{\nodetype}{1,0}{\StrBetween[1,2]{#2}{,}{,}};
 \StrChar{#1}{3}[\nodetype];
 \DDnode{\nodetype}{2,0}{\StrBetween[2,3]{#2}{,}{,}};
 \StrChar{#1}{4}[\nodetype];
 \DDnode{\nodetype}{3,0}{\StrBetween[3,4]{#2}{,}{,}};
 \StrChar{#1}{5}[\nodetype];
 \DDnode{\nodetype}{4,0}{\StrBetween[4,5]{#2}{,}{,}};
 \StrChar{#1}{6}[\nodetype];
 \DDnode{\nodetype}{5,0}{\StrBehind[5]{#2}{,}};
 \useasboundingbox (-.4,-.2) rectangle (5.4,0.55);
 \end{tikzpicture}
 \end{tiny}
 }
   
\begin{document}

 \title[The gap phenomenon in parabolic geometries]{The gap phenomenon in parabolic geometries}
 \author[B. Kruglikov]{Boris Kruglikov}
 \author[D. The]{Dennis The}
 \address{Department of Mathematics and Statistics, University of Troms\o, 90-37, Norway}
 \email{boris.kruglikov@uit.no} 
 \address{Mathematical Sciences Institute, Australian National University, ACT 0200, Australia}
 \email{dennis.the@anu.edu.au}


 \date{\today}
 \begin{abstract}
 The infinitesimal symmetry algebra of any Cartan geometry has maximum dimension realized by the flat model, but often this dimension drops significantly when considering non-flat geometries, so a gap phenomenon arises.  For general (regular, normal) parabolic geometries of type $(G,P)$, we use Tanaka theory to derive a universal upper bound on the submaximal symmetry dimension.  We use Kostant's version of the Bott--Borel--Weil theorem to show that this bound is in fact sharp in almost all complex and split-real cases by exhibiting (abstract) models.  We explicitly compute all submaximal symmetry dimensions when $G$ is any complex or split-real simple Lie group.
  \end{abstract}
 \maketitle
 
 
 \tableofcontents
 \section{Introduction}

A Riemannian metric $\metric$ on a (connected) $n$-manifold $M$ admits at most an $\binom{n+1}{2}$-dimensional symmetry (isometry) algebra, and equality occurs precisely for the constant curvature spaces.  Below this maximum, a gap phenomenon occurs: there are forbidden symmetry dimensions.  Submaximal symmetry dimensions are given in Table \ref{F:submax-Riem}.
 More generally, for geometric structures which admit a finite maximal symmetry dimension, the determination of the submaximal symmetry dimension $\fS$ is called the {\em symmetry gap problem}.  For various geometric structures, this problem has been studied since the late 19th century by many authors, including Lie, Tresse, Fubini, Cartan, Yano, Wakakuwa, Vranceanu, Egorov, Obata, and Kobayashi -- see \cite{Tresse1896, Fub1903, Ego1978, Kob1972} and references therein.
  
  \begin{table}[h]
$\begin{array}{|c|c|c|l|} \hline
 n & {\rm Max} & {\rm Submax} & {\rm References}\\ \hline\hline
 2 & 3 & 1 & \mbox{Lie (1882) \cite{Lie1882}, Darboux (1894) \cite{Dar1894}}\\
 3 & 6 & 4 & \mbox{Bianchi (1898) \cite{Bianchi1898}, Ricci (1898) \cite{Ricci1898}}\\
 4 & 10 & 8 & \mbox{Egorov (1955) \cite{Ego1955}}\\
 \geq 5 & \binom{n+1}{2} & \binom{n}{2} + 1 & \mbox{Wang (1947) \cite{Wang1947}} \\[2pt] \hline
 \end{array}$ 
 \caption{Maximal / submaximal symmetry dimensions for Riemannian geometry}
 \label{F:submax-Riem}
 \end{table}

 A Riemannian structure $(M,\metric)$ is equivalently described as a Cartan geometry $(\cG \to M, \omega)$ of type $(\oE(n), \tO(n))$, where $\oE(n)$ is the Euclidean group and $\tO(n)$ is the orthogonal group.  Indeed, $\cG$ is the orthonormal frame bundle and the Cartan connection $\omega$ is a sum of the Levi--Civita connection of $g$ and the soldering form.  Riemannian geometries are curved versions of the flat model $\oE(n) / \tO(n) \cong \bbR^n$, which admits the maximum $\dim(\oE(n)) = \binom{n+1}{2}$ symmetries.

 Many well-known geometric structures such as conformal, projective, CR, 2nd order ODE systems, and various bracket-generating distributions, are equivalently described as {\em parabolic geometries}, i.e.\ Cartan geometries $(\cG \to M, \omega)$ of type $(G,P)$, where $G$ is a real or complex \ss Lie group and $P \subset G$ is a parabolic subgroup.  There is an equivalence of categories between {\em regular, normal} parabolic geometries and underlying structures \cite{Tan1979, Mor1993, CS2000}.  For all of these, the infinitesimal symmetry algebra $\inf(\cG,\omega)$ has maximum dimension $\dim(G)$ and this is (locally) {\em uniquely} realized by the flat model $(G \to G/P,\omega_G)$, where $\omega_G$ is the Maurer--Cartan form on $G$.  Thus, the gap problem\footnote{We study {\em infinitesimal} symmetries here.  The problem of studying {\em global} symmetry dimensions which are ``less than maximal'' is quite different.  For example, take the flat model $(G \to G/P, \omega_G)$ and remove a point on the base and the fibre over it.  This yields a flat geometry with global automorphism group isomorphic to $P$.} can be phrased as: 
{\em Among all {\bf regular, normal} parabolic geometries $(\cG \to M, \omega)$ of type $(G,P)$ that are {\bf not flat}, what is the maximum $\fS$ of $\dim(\finf(\cG,\omega))$?}
 
 The advantage of the Cartan perspective is that despite the zoo of different geometric structures downstairs on the base manifold $M$, their description upstairs on $\cG$ becomes much more uniform, based on common structural features inherent in the groups $(G,P)$, which have been well studied in representation theory.  For instance, all parabolic geometries admit a fundamental curvature quantity called {\em harmonic curvature} $\kappa_H$, valued in a certain Lie algebra cohomology group $H^2_+(\fg_-,\fg)$, and which is a complete obstruction to flatness.  (In conformal geometry, $\Kh$ is the Weyl tensor.)

Previously known symmetry gap results for parabolic geometries are given in Table \ref{F:known-submax}.   These were obtained using a variety of techniques, e.g.\ computation of the algebra of all differential invariants of a pseudogroup (scalar 2nd order ODE), Cartan's method of equivalence ($(2,3,5)$-distributions), classification of Lie algebras of contact vector fields in the plane (3rd order ODE\footnote{Sophus Lie classified all finite-dimensional irreducible Lie algebras of {\em contact} vector fields in plane.  There are only three, $L_6,L_7,L_{10}$, with $\dim(L_i) = i$ with $L_6 \subset L_7$ and $L_6 \subset L_{10}$.  The only 3rd order ODE invariant under $L_6$ (and hence $L_7$ and $L_{10}$) is the flat model $y''' = 0$.  In \cite{WMQ2002}, a {\em point} symmetry classification was carried out.  Aside from $y''' = 0$ and $y''' = \frac{3 (y'')^2}{2 y'}$ (which are {\em contact}-equivalent and have 7 and 6 point symmetries respectively), the maximal point symmetry dimension is at most 5.  Indeed, for any $a \in \bbR$, $y''' = ay' + y$ has five point (contact) symmetries.}), or studying integrability conditions for the equations characterizing symmetries (projective structures).  Some of these results relied heavily on the low-dimensional setup.

All results in Table \ref{F:known-submax} are correct, but strictly speaking, only Egorov's result has no further assumptions and is equivalent to the result we obtain here.  For example for $(2,3,5)$-geometry, modelled on $G_2 / P_1$, Cartan \cite{Car1910} identified a binary quartic invariant $\cA$, i.e.\ $\Kh$, and studied the maximal symmetry {\em among structures with constant root type}.  The largest is 7 and occurs where $\cA$ has a single root of multiplicity $4$. However, non-flatness of the geometry for us means only that $\Kh(u) \neq 0$ at {\em some} $u \in \cG$, so the result $\fS=7$ that we will obtain is a sharpening of Cartan's result.

 More recently, for arbitrary parabolic geometries, \v{C}ap \& Neusser \cite{CN2009} gave a general algebraic strategy for finding upper bounds on $\fS$ using Kostant's version of the Bott--Borel--Weil theorem \cite{Kos1961}.  However, the implementation of their strategy must be carried out on a case-by-case basis; moreover, their upper bounds are in general not sharp.  (See also Remark \ref{RM:CN}.)  For structures determined by a bracket-generating distribution (not necessarily underlying  parabolic geometries), another approach based on an elaboration of Tanaka theory \cite{Tan1970, Tan1979} was proposed in the works \cite{Kru2011,Kru2012} of the first author.

  \begin{table}[h]
 \begin{tabular}{|c|c|c|c|c|} \hline
 Geometry & Model & Max & Submax & Reference\\ \hline\hline
 \begin{tabular}{c} Scalar 2nd order ODE\\ mod point transformations\end{tabular} & $\SL_3(\bbR) / P_{1,2}$ & 8 & 3 &  Tresse (1896) \cite{Tresse1896}\\ \hline
 $2$-dim.\ projective structures & $\SL_3(\bbR) / P_1$ & 8 & 3 & Tresse (1896) \cite{Tresse1896}\\ \hline
 $(2,3,5)$-distributions & $G_2 / P_1$ & 14 & 7 & Cartan (1910) \cite{Car1910}\\ \hline
 3-dim.\ CR structures & $\SU(2,1) / B$ & 8 & 3 & Cartan (1932) \cite{Car1932}\\ \hline
 Projective structures (dim.\ $\rkg \geq 3$) & $\SL_{\rkg+1}(\bbR) / P_1$ & $\rkg^2+2\rkg$ & $(\rkg-1)^2 + 4$ & Egorov (1951) \cite{Ego1951}\\ \hline
 \begin{tabular}{c} Scalar 3rd order ODE\\ mod contact transformations \end{tabular} & $\Sp_4(\bbR) / P_{1,2}$ & 10 & 5 &  \begin{tabular}{c} Wafo Soh, Mahomed,\\ Qu (2002) \cite{WMQ2002} \end{tabular}\\ \hline
  Pairs of 2nd order ODE & $\SL_4(\bbR) / P_{1,2}$ & 15 & 9 &  \begin{tabular}{c}Casey, Dunajski,\\ Tod (2013) \cite{CDT2013} \end{tabular}\\ \hline
 \end{tabular}
 \caption{Previously known submaximal symmetry dimensions for parabolic geometries}
 \label{F:known-submax}
 \end{table}

 The main idea behind our approach is to combine Tanaka theory with the \v{C}ap--Neusser approach based on Kostant theory.  This yields a {\em uniform algebraic approach} to the gap problem which is rooted in the structure theory of \ss Lie algebras. 
 The main results of this article are:
 \begin{itemize}
 \item We establish a universal upper bound $\fS \leq \fU$ (Theorem \ref{T:upper}), where $\fU$ is algebraically determined.  In arbitrary real cases, we have $\fS \leq \fU \leq \fU^\bbC$ (Corollary \ref{C:upper}).
 \item In complex or {\em split-real}\footnote{We refer to $\fg$ as split-real if it is a split real form of its complexification, e.g. $\fsl(n,\bbR)$, but not $\fsu(n)$.} cases, we:
 \begin{itemize}
 \item exhibit models with $\dim(\finf(\cG,\omega)) = \fU$ in almost all cases (Theorem \ref{T:realize}).  Thus, $\fS = \fU$ almost always (Theorem \ref{T:main-thm2}).  Exceptions are also studied (\S \ref{S:exceptions}).
 \item give a Dynkin diagram recipe to efficiently compute $\fU$ (\S \ref{S:Dynkin}).
 \item establish local homogeneity of all submaximally symmetric models\footnote{This is not universally true outside the parabolic context, e.g. Killing fields for metrics on surfaces (see Table \ref{F:submax-Riem}).} near non-flat {\em regular} points (Theorem \ref{T:transitive}); the set of all such points is open and dense in $M$ (Lemma \ref{L:dense}).
 \end{itemize}
 \item We recover all results in Table \ref{F:known-submax}; some sample new results are given in Table \ref{F:sample-submax}.  Our complete classification when $G$ is (complex or split-real) {\em simple} is presented in Appendix \ref{App:Submax}.
 \end{itemize}  

  \begin{table}[h]
 \begin{tabular}{|c|c|c|c|c|c|} \hline
 Geometry & Range & Model & Max & Submax \\ \hline\hline
 \begin{tabular}{c} Signature $(p,q)$ conformal \\ geometry in dim.\ $n = p+q$ \end{tabular} & $p,q \geq 2$ & $\SO_{p+1,q+1} / P_1$ & $\binom{n+2}{2}$ & $\binom{n-1}{2} +6$ \\ \hline
 \begin{tabular}{c} Systems of 2nd order ODE \\ in $m$ dependent variables \end{tabular} & $m \geq 2$ & $\SL_{m+2}(\bbR) / P_{1,2}$ & $(m+2)^2-1$ & $m^2 + 5$ \\ \hline
 $(2,m)$-Segr\'e structures & $m \geq 2$ & $\SL_{m+2}(\bbR) / P_2$ & $(m+2)^2-1$ & $m^2 + 5$ \\ \hline
 \begin{tabular}{c} Generic rank $\rkg$ distributions \\ on $\half \rkg(\rkg+1)$-dim.\ manifolds \end{tabular} & $\rkg \geq 3$ & $\SO_{\rkg,\rkg+1} / P_\rkg$ & $\binom{2\rkg+1}{2}$ & $\mycase{\frac{\rkg(3\rkg-7)}{2} + 10, & \rkg \geq 4; \\ 11, & \rkg = 3 }$ \\ \hline
 Lagrangean contact structures & $\rkg \geq 3$ & $\SL_{\rkg+1}(\bbR) / P_{1,\rkg}$ & $\rkg^2 + 2\rkg$ & $(\rkg-1)^2 + 4$ \\ \hline
 Contact projective structures & $\rkg \geq 2$ & $\Sp_{2\rkg}(\bbR) / P_1$ & $\rkg(2\rkg+1)$ & $\mycase{2\rkg^2 - 5\rkg + 8, & \rkg \geq 3;\\ 5, & \rkg = 2}$\\ \hline
 Contact path geometries & $\rkg \geq 3$ & $\Sp_{2\rkg}(\bbR) / P_{1,2}$ & $\rkg(2\rkg+1)$ & $2\rkg^2 - 5\rkg + 9$\\ \hline
 \begin{tabular}{c} Exotic parabolic contact\\ structure of type $E_8$ \end{tabular} & - & $E_8 / P_8$ & 248 & 147\\ \hline
 \end{tabular}
 \caption{Sample new results of submaximal symmetry dimensions for parabolic geometries}
 \label{F:sample-submax}
 \end{table}
 
 In \S \ref{S:upper-sec}, we review some background and recent results, and define in \S \ref{S:upper} a graded subalgebra $\fa^\phi := \prn^\fg(\fg_-,\fann(\phi))$ of $\fg$ called the {\em Tanaka prolongation algebra}.  The crucial new ingredient which creates a bridge to Tanaka theory is the notion of a {\em regular point} (Definition \ref{D:reg-pt}).  Such points form an open dense subset, and  $\dim(\inf(\cG,\omega)) \leq \dim(\fa^{\Kh(u)})$ at any regular point.  For non-flat geometries, we can always find a regular point which is {\em non-flat}, i.e.\ $\Kh(u) \neq 0$.  Defining $\fU := \max\{ \dim(\fa^\phi) \mid 0 \neq \phi \in H^2_+(\fg_-,\fg) \}$ quickly leads to the result $\fS \leq \fU$.
 
 \S \ref{S:PR-analysis} is devoted to studying the Tanaka algebra $\fa^\phi$.  Over $\bbC$, if $\bbV$ is a $\fg_0$-irreducible representation (irrep) and $\phi_0$ is any extremal (i.e. highest or lowest) weight vector, then $\dim(\fann(\phi)) \leq \dim(\fann(\phi_0))$, $\forall 0 \neq \phi \in \bbV$.  In Proposition \ref{P:Tanaka-lw}, we show that this property persists for the Tanaka algebra, i.e. $\dim(\fa^\phi) \leq \dim(\fa^{\phi_0})$, $\forall 0 \neq \phi \in \bbV$.  This leads in \S \ref{S:Dynkin} to a purely combinatorial (Dynkin diagram) recipe to compute $\dim(\fa^{\phi_0})$ and hence $\fU$.
 (The  $E_8 / P_8$ case becomes a simple exercise -- see Example \ref{ex:E8P8}.)  In \S \ref{S:PR}, we study {\em prolongation-rigidity}: $(\fg,\fp)$ is PR iff $\fa^\phi_+ = 0$, $\forall 0 \neq \phi \in H^2_+(\fg_-,\fg)$; otherwise, it is NPR.  An example of an NPR parabolic geometry is the geometry of 2nd order ODE systems.  When $\fg$ is (complex) simple, a complete classification of NPR geometries is given, and the grading height $\tilde\nu$ of $\fa^\phi$ (for any $\phi \neq 0$) is highly constrained: $0 \leq \tilde\nu \leq 2$ always, and $\tilde\nu = 2$ occurs only for a couple of cases.  In \S \ref{S:corr-twistor}, we study the effect of correspondence and twistor spaces on the Tanaka prolongation.  These provide an important tool for simplifying our calculations.

 We show in \S \ref{S:submax} that $\fS = \fU$ in almost all complex and split-real cases: a model can generally be constructed by deforming the Lie algebra structure on $\fa^{\phi_0}$ by $\phi_0$ (Theorem \ref{T:realize}).  To study exceptions, {\em filtration-rigidity} is introduced in \S \ref{S:deform}.  In \S \ref{S:exceptions}, we find $\fU-1 \leq \fS \leq \fU$ for all exceptions.

 Concrete examples are considered in \S \ref{S:specific}, and submaximally symmetric models are given explicitly (in coordinates) in terms of their underlying structure on the base manifold $M$.  This includes conformal geometry, $(2,3,5)$-distributions, $(3,6)$-distributions, 2nd order ODE systems, projective structures, and $(2,m)$-Segr\'e structures.  Four-dimensional Lorentzian conformal geometry and $(2,3,5)$-distributions are investigated in finer detail: the maximum symmetry in each Petrov type and root type is identified.  In doing so, we also exhibit in \S \ref{S:4d-Lor} the first known example of a Petrov type II metric with four (conformal) Killing vectors.

 The gap problem is more subtle for general real forms since $\max\{ \dim(\fann(\phi)) \mid 0 \neq \phi \in \bbV \}$ can be difficult to determine in the absence of vectors which complexify to extremal weight vectors.  Our upper bound $\fU$ is still valid, but its realizability must in general be checked.  For instance, for conformal geometry, we exhibit local models in all signatures except Riemannian and Lorentzian which realize the complex upper bound $\fU^\bbC$, so $\fS = \fU = \fU^\bbC$ for these cases.   The conformal Riemannian and Lorentzian cases are different; they have been recently settled  in \cite{DT-Weyl}.
  
 {\bf Conventions:}  We assume throughout that $M$ is a connected manifold.  When working with real and complex parabolic geometries, our results are formulated in the  smooth and holomorphic categories, respectively.  We use {\em left} cosets and {\em right} principal bundles.  For a Lie group $G$, we identify its Lie algebra $\fg := \Lie(G)$ with the {\em left}-invariant vector fields on $G$.  Parabolic subalgebras will be denoted $\fp$, $\fq$, and corresponding parabolic subgroups are $P,Q$.  We always assume that $G$ acts on $G/P$ {\em infinitesimally effectively}, i.e.\ the kernel of the $G$-action on $G/P$ is at most discrete.  Equivalently, simple ideals of $\fg$ are not contained in $\fg_0$.  (This avoids situations like $G = G' \times G''$ and $P = P' \times G''$, where the $G''$-action is not visible on $G/P$.)

 We write $A_\rkg, B_\rkg, C_\rkg, D_\rkg, G_2, F_4, E_6, E_7, E_8$ for the complex simple Lie algebras\footnote{The classical Lie algebras are $A_\rkg \cong \fsl(\rkg+1,\bbC)$, $B_\rkg \cong \fso(2\rkg+1,\bbC)$, $C_\rkg \cong \fsp(2\rkg,\bbC)$, $D_\rkg \cong \fso(2\rkg,\bbC)$.}, {\bf or} any complex Lie groups having these as their Lie algebras.  (In \S \ref{S:specific}, we abuse this notation further by letting it refer to real forms, and specify the precise real form as necessary.)  We use the Bourbaki \cite{Bourbaki1968} / {\tt LiE} \cite{LiE} ordering of simple roots.  This differs from the ordering used in \cite{CS2009} for $E_6,E_7,E_8,F_4$; also, their definition of the Cartan matrix is the transpose of ours.  
 
 Dynkin diagrams are drawn with open white circles.  This is the same notation as the Satake diagram for the corresponding split real form, and serves to emphasize that all our results are the same in both settings.  We use the ``minus lowest weight'' Dynkin diagram convention (see \S \ref{S:rep-th}) when referring to irreducible $\fg_0$-modules.  If $\fg$ is simple, we use $\lambda_\fg$ to denote its highest weight (root).\\
  
 {\bf Acknowledgements}: We are grateful for many helpful discussions with Boris Doubrov, Mike Eastwood, Katharina Neusser, Katja Sagerschnig, and Travis Willse.  Much progress on this paper was made during the conference ``The Interaction of Geometry and Representation Theory: Exploring New Frontiers'' devoted to Mike Eastwood's 60th birthday, and held in Vienna in September 2012 at the Erwin Schr\"odinger Institute.  Boris Doubrov gave some key insights during this conference which led to the proof of Proposition \ref{P:Tanaka-lw}.
 
 The representation theory software {\tt LiE}, as well as Ian Anderson's {\tt DifferentialGeometry} package in {\tt Maple} were invaluable tools for facilitating the analysis in this paper.
 
 B.K. was supported by the University of Troms\o{} while visiting the Australian National University, where this work was initiated. The hospitality of ANU is gratefully acknowledged.  D.T. was supported under the Australian Research Council's {\em Discovery Projects} funding scheme (project number DP110100416).
 
  
 \section{A universal upper bound}
 \label{S:upper-sec}
  
 We give some brief background (\S \ref{S:Tanaka}--\ref{S:parabolic}), review recent results (\S \ref{S:CNK}), and then establish a universal (algebraic) upper bound on the submaximal symmetry dimension (\S \ref{S:upper}).

 \subsection{Tanaka theory in a nutshell}
 \label{S:Tanaka}
  
 The aim of Tanaka theory \cite{Tan1970, Tan1979, Zel2009} is to study the local equivalence problem for filtered geometric structures.  The given geometric data is a manifold $M$ with a (vector) distribution $D \subseteq TM$ endowed possibly with some additional structure on it, e.g.\ a sub-Riemannian metric or a conformal structure.  For many such structures, one can canonically associate a Cartan geometry $(\cG \to M,\omega)$ of some type $(G,K)$.  We give an outline of these ideas.
  
 Iterating Lie brackets of sections of $D$, we form the weak derived flag
 \[
 D =: D^{-1} \subset D^{-2} \subset D^{-3} \subset...\,\,\,, \qbox{where} \Gamma(D^{i-1}) := [\Gamma(D^i), \Gamma(D^{-1})].
 \]
 We assume that all $D^i$ have constant rank, and $D$ is bracket-generating in $TM$, i.e.\ $D^{-\nu} = TM$ for some minimal $\nu \geq 1$ (called the {\em depth}).  The Lie bracket restricts to a map $\Gamma(D^i) \times \Gamma(D^j) \to \Gamma(D^{i+j})$, so $(M,\{ D^i \})$ is a {\em filtered manifold}, and induces a tensorial bracket on the associated-graded,
 \[
 \fm(x) = \fg_-(x) = \bop_{i < 0} \fg_i(x), \qquad  \fg_i(x) = D^i(x) / D^{i+1}(x).
 \]
 This {\em Levi bracket} turns $\fm(x)$ into a graded nilpotent Lie algebra called the {\em symbol algebra at $x$}. We further assume that $D$ is of {\em constant type}, i.e.\ there is a fixed $\fm = \fg_-$ such that $\fm(x) \cong \fm$, $\forall x \in M$.
 
 Consider the frame bundle $\cF_{gr}(M) \stackrel{\pi}{\to} M$ with $\pi^{-1}(x)$ consisting of all graded isomorphisms $u : \fm \to \fm(x)$.  The structure group $\Aut_{gr}(\fm)$ of this principal bundle is the group of graded automorphisms of $\fm$.  If $D$ is endowed with additional structure, we can reduce to a principal subbundle $\cG_0 \to M$ with structure group $G_0 \subseteq \Aut_{gr}(\fm)$.  (In conformal geometry, $D = TM$ and $G_0 = \CO(\fg_{-1})$; for $(2,3,5)$-distributions (\S \ref{S:G2P1}), $G_0 \cong \Aut_{gr}(\fm) \cong \GL(\fg_{-1}) \cong \GL_2$.)  From here, one constructs the {\em geometric prolongation} of the given structure, namely a tower of adapted bundles $...\to \cG_2 \to \cG_1 \to \cG_0 \to M$.  We refer to \cite{Zel2009} for details on this procedure.  This geometric prolongation is controlled by an {\em algebraic prolongation} which we describe below.

 The Lie algebra $\fg_0$ is contained in the algebra $\fder_{gr}(\fm)$ of grade-preserving derivations of $\fm$. Given $(\fm,\fg_0)$, the axioms for Tanaka's algebraic prolongation $\prn(\fm,\fg_0) := \bop_{i \in \bbZ} \fg_i(\fm,\fg_0)$ are:
 \begin{align}
 & \fg_{\leq 0}(\fm,\fg_0) = \fm \op \fg_0; \tag{T.1} \label{T1}\\
 & \mbox{If $X \in \fg_i(\fm,\fg_0)$ for $i \geq 0$ satisfies $[X,\fg_{-1}] = 0$, then $X = 0$;} \tag{T.2} \label{T2}\\
 & \mbox{$\prn(\fm,\fg_0)$ is the maximal graded Lie algebra  satisfying \eqref{T1} and \eqref{T2}.} \tag{T.3} \label{T3}
 \end{align}
 Write $\prn(\fm) := \prn(\fm,\fder_{gr}(\fm))$.  Up to isomorphism, there is a unique graded Lie algebra satisfying \eqref{T1}--\eqref{T3}.  In fact, Tanaka gives an explicit inductive realization of $\fg := \prn(\fm,\fg_0)$; for $i > 0$,
  \begin{align}
 \fg_i = \l\{ f \in \bop_{j < 0} \fg_j^* \ot \fg_{j+i} \mid f([v_1,v_2]) = [f(v_1),v_2] + [v_1,f(v_2)], \,\, \forall v_1,v_2 \in \fm \r\}.
 \label{E:alg-pr-explicit}
 \end{align}
 The brackets on $\fg$ are: (i) The given brackets on $\fm$ and $\fg_0$; (ii) If $f \in \fg_i$, $i \geq 0$, and $v \in \fm$, then define $[f,v] := f(v)$; (iii) Brackets on the non-negative part are defined inductively: If $f_1 \in \fg_i$ and $f_2 \in \fg_j$, for $i,j \geq 0$, then define $[f_1,f_2] \in \fg_{i+j}$ by $ [f_1,f_2](v) := [f_1(v),f_2] + [f_1,f_2(v)]$, $v \in \fm$.
 
 \begin{remark} \label{RM:Hom}
 Since $\fg_-$ is generated by $\fg_{-1}$, any $X \in \fg_i$, $i \geq 0$, is determined by its action on $\fg_{-1}$.  By \eqref{T2} the Lie bracket on $\fg$ induces $\fg_i \inj \fg_{-1}^* \ot \fg_{i-1}$, so $\fg_i \inj \Hom(\ot^i \fg_{-1}, \fg_0) \inj \Hom(\ot^{i+1} \fg_{-1}, \fg_{-1})$.
 \end{remark}
 
 Note that by \eqref{T2}, if $\fg_r = 0$ for some $r \geq 0$, then $\fg_i = 0$ for all $i > r$.  This case of finite termination is particularly important.
   
  \begin{thm} \cite[Thm.\ 8.4]{Tan1970}, \cite{Zel2009} Let $\cG_0 \to M$ be a structure of constant type $(\fm,\fg_0)$ and $\fg = \prn(\fm,\fg_0)$.  Suppose $r \geq 0$ is minimal such that $\fg_{r+1} = 0$.  Then there exists a canonical frame on the $r$-th geometric prolongation $\cG_r$ of $\cG_0$, and the symmetry dimension is bounded above by $\dim(\fg)$. 
 \end{thm}

 Often this framing is a Cartan connection.  (This is the case for all structures in this article.)
 
 \begin{defn} \label{D:Cartan} A Cartan geometry $(\cG \to M, \omega)$ of type $(G,K)$ (or a ``$G/K$ geometry'') is a principal $K$-bundle $\cG \to M$ endowed with a Cartan connection $\omega \in \Omega^1(\cG,\fg)$ with defining properties: 
 \begin{enumerate}
 \item[\rm (i)] $\omega$ is $K$-equivariant;
 \item[\rm (ii)] $\omega(\zeta_Y) = Y$, $\forall Y \in \fk$, where $\zeta_Y(u) = \l.\frac{d}{dt}\r|_{t=0} u \cdot \exp(tY)$;
 \item[\rm (iii)] $\omega_u : T_u \cG \to \fg$ is a linear isomorphism, $\forall u \in \cG$.
 \end{enumerate}
 The curvature of $\omega$ is $d\omega + \frac{1}{2} [\omega,\omega] \in \Omega^2(\cG,\fg)$.  The isomorphism (iii) identifies this 2-form with the curvature function $\kappa : \cG \to \bigwedge^2 \fg^* \ot \fg$, which is horizontal, so $\kappa : \cG \to \bigwedge^2 (\fg/\fk)^* \ot \fg$.  The infinitesimal symmetries are $\finf(\cG,\omega) = \{ \xi \in \fX(\cG)^K \mid \cL_\xi \omega = 0 \}$.
 \end{defn}

A Cartan geometry is {\em flat} if $\kappa\equiv 0$. A fundamental fact is that $(\cG \to M, \omega)$ is flat iff it is locally equivalent to the {\em flat model} $(G \to G/K,\omega_G)$ (equipped with the Maurer--Cartan form $\omega_G$).

 While calculating the Tanaka prolongation from given data $(\fm,\fg_0)$ is algorithmic, a naive application of \eqref{E:alg-pr-explicit}  generally leads to a computationally intensive exercise in linear algebra.  Moreover, even if this task is completed, there still remains the general problem of understanding the structure of the resulting Tanaka algebra.  However, in the context of parabolic geometries, these problems are resolved by Yamaguchi's prolongation theorem \cite{Yam1993}; see Theorem \ref{T:Y-pr-thm} of Appendix \ref{App:Yam}.

Computing the Tanaka prolongation is much simpler when one knows $\fg = \prn(\fm,\fg_0)$, and wants $\fa = \prn(\fm,\fa_0)$ for some subalgebra $\fa_0 \subseteq \fg_0$.  By Remark \ref{RM:Hom}, for $k > 0$, we have $\fa_k \inj \fg_{-1}^* \ot \fa_{k-1} \inj \fg_{-1}^* \ot \fg_{k-1}$, and $\fa_k \inj \Hom(\ot^k \fg_{-1},\fa_0)$.  The inclusion $\fa_0 \inj \fg_0$ induces inclusions $\fa_k \inj\fg_k$.  Hence,
  
 \begin{lemma} \label{L:T-subalg} Suppose $\fm = \fg_-$ is generated by $\fg_{-1}$.  If $\fa_0 \subseteq \fg_0$ is any subalgebra, then $\fa := \prn(\fm,\fa_0) \inj \fg := \prn(\fm,\fg_0)$.  Indeed for $k > 0$,
 \begin{align*}
  \fa_k = \{ X \in \fg_k \mid [X,\fg_{-1}] \subset \fa_{k-1} \} = \{ X \in \fg_k \mid \ad_{\fg_{-1}}^{k}(X) \subset \fa_0 \}.
 \end{align*}
 More precisely, $\fa_k = \{ X \in \fg_k \mid \ad_{u_1} \circ ... \circ \ad_{u_k}(X) \in \fa_0,\,\, \forall u_i \in \fg_{-1} \}$.
 \end{lemma}
 
 
 \subsection{Parabolic geometry in a nutshell}
 \label{S:parabolic}
 
 The following are standard facts from parabolic geometry -- see \cite[\S 3]{CS2009} for further details.
 
 Given a real or complex \ss Lie algebra $\fg$, a  {\em $\bbZ$-grading} (or $\nu$-grading) is a decomposition $\fg = \fg_{-\nu} \op ... \op \fg_\nu$ with $\fg_{\pm \nu} \neq 0$, $[\fg_i, \fg_j] \subset \fg_{i+j}$, and $[\fg_j, \fg_{-1}] = \fg_{j-1}$ for $j < 0$, i.e.\ $\fg_{-1}$ is {\em bracket-generating} in $\fg_-$.  The subalgebra $\fg_0 = \fz(\fg_0) \op \fg_0^{ss}$ is reductive, each $\fg_j$ is a $\fg_0$-module, and there exists a unique {\em grading element} $Z \in \fz(\fg_0)$, i.e.\ $[Z,X] = j X$ for $X \in \fg_j$.  (The eigenvalues of $Z$ on a given $\fg_0$-module will often be referred to as its {\em homogeneities}.)  Defining $\fg^i := \bop_{j \geq i} \fg_j$, we have $[\fg^i, \fg^j] \subseteq \fg^{i+j}$, so $\fg$ is canonically a filtered Lie algebra with associated-graded $\gr(\fg) \cong \fg$.  A subalgebra $\fp \subset \fg$ is {\em parabolic} if $\fp = \fg^0 = \fg_{\geq 0}$ for some $\nu$-grading of $\fg$, and then $\fp^{\opn} = \fg_{\leq 0}$ is the {\em opposite} parabolic.  Each filtrand $\fg^i$ is $\fp$-invariant, and the quotient $\fg / \fp$ is naturally filtered.  The Killing form $B$ on $\fg$ induces the dualities $\fg_{-i} \cong (\fg_i)^*$ (as $\fg_0$-modules) and $\fg^i = (\fg^{-i+1})^\perp$ (as $\fp$-modules).  The nilradical of $\fp$ is $\fp_+ := \fg^1 = \fp^\perp$, so $(\fg / \fp)^* \cong \fp_+$ (as $\fp$-modules).

  Let $G$ be a \ss Lie group with Lie algebra $\fg$ and parabolic subalgebra $\fp \subset \fg$.  A subgroup $P \subset G$ is {\em parabolic} if it lies between the normalizer $N_G(\fp)$ and its connected component of the identity.  Under the adjoint action, $P$ preserves the filtration on $\fg$.  Define $G_0 \subset P$ to be the subgroup which preserves the grading on $\fg$; its Lie algebra is $\fg_0$.  There is a decomposition $P = G_0 \ltimes P_+$ for some closed normal subgroup $P_+ \subset P$ with Lie algebra $\fp_+$.
  
  A {\em parabolic geometry} is a Cartan geometry $(\cG \to M, \omega)$ of type $(G,P)$.  Since $T\cG \cong \cG \times \fg$ via $\omega$ and $TM \cong \cG \times_P (\fg / \fp)$, the filtration $\fg = \fg^{-\nu} \supset ... \supset \fg^\nu$ induces a ($P$-invariant) filtration $T\cG = T^{-\nu} \cG \supset ... \supset T^\nu \cG$ which projects to a filtration $TM = T^{-\nu} M \supset ... \supset T^{-1} M$.  Given the principal $G_0$-bundle $\cG_0 := \cG / P_+$ over $M$, we have $\gr(TM) \cong \cG_0 \times_{G_0} \fg_-$.  Let $\kappa : \cG \to \bigwedge^2(\fg/\fp)^* \ot \fg$ be the curvature function of $\omega$.  The parabolic geometry is:
  \begin{enumerate}
  \item {\em regular} iff $\kappa(\fg^i,\fg^j) \subset \fg^{i+j+1}$, $\forall i, j < 0$, i.e.\ $\kappa$ has {\em positive} homogeneity.
  \item {\em normal} iff $\partial^* \kappa = 0$.  Here, $\partial^*$ is the ($P$-equivariant) Kostant codifferential on $\bigwedge^2 \fp_+ \ot \fg$, which is the negative of the Lie algebra homology differential.
  \end{enumerate}

  Regularity is equivalent to: (i) $M$ being a filtered manifold with respect to the above filtration, and (ii) the algebraic bracket on $\gr(TM)$ induced from $\fg_-$ is the same as the Levi bracket.  Indeed, we obtain a {\em regular infinitesimal flag structure} of type $(G,P)$ on $M$, i.e.\ a filtration of $TM$ generated by $D = T^{-1} M$ together with a reduction of structure group of the (graded) frame bundle of $\gr(TM)$ to $\cG_0$ corresponding to $\Ad: G_0 \to \Aut_{gr}(\fg_-)$.  For many parabolic geometries\footnote{See \S \ref{S:rep-th} for an explanation of notations from parabolic geometry.} such as $(2,3,5)$-geometry $(G_2 / P_1)$, the underlying structure is a bracket-generating distribution with no reduction of structure group.  For conformal structures $(B_\rkg / P_1, D_\rkg / P_1)$, the filtration on $M$ is trivial, so the geometry is determined by a reduction of structure group alone.  Parabolic contact structures have a non-trivial filtration as well as a reduction of structure group, e.g.\ CR and Lagrangean contact structures $(A_\rkg / P_{1, \rkg})$, and Lie contact structures $(B_\rkg / P_2, D_\rkg /P_2)$.  
  
 A fundamental result of Tanaka \cite{Tan1979}, Morimoto \cite{Mor1993}, \v{C}ap--Schichl \cite{CS2000} is that there is an equivalence of categories between regular, normal parabolic geometries and {\em underlying structures}.  According to Yamaguchi \cite{Yam1993} (see Appendix \ref{App:Yam}), in most cases the Tanaka prolongation satisfies $\prn(\fg_-,\fg_0) \cong \fg$.  Here, the ``underlying structure'' is precisely the regular infinitesimal flag structure.  Two notable exceptions are projective structures $(A_\rkg / P_1)$ and contact projective structures $(C_\rkg / P_1)$.  In these cases $\prn(\fg_-,\fg_0)$ is infinite-dimensional, so the grading one component needs to be constrained to $\fg_1$ through additional structure.
 
 Since $(\p^*)^2 = 0$, then for regular, normal parabolic geometries there is a fundamental quantity called {\em harmonic curvature} $\Kh : \cG \to \ker(\p^*) / \im(\p^*)$, which is much simpler than $\kappa$ and is still a complete obstruction to flatness.  
Namely, $\Kh \equiv 0$ iff the geometry is locally equivalent to the flat model $(G \to G/P,\omega_G$).  Since $\Kh$ is $P$-equivariant and $\ker(\p^*) / \im(\p^*)$ is completely reducible (so $P_+$ acts trivially), then $\Kh$ descends to a $G_0$-equivariant function on $\cG_0$.  As $G_0$-modules $(\fg/\fp)^* \cong (\fg_-)^* \cong \fp_+$, so consider the ($G_0$-equivariant) Lie algebra cohomology differential $\p$ and the {\em Kostant Laplacian} $\Box := \p \p^* + \p^* \p$ acting on co-chains $\bigwedge^k (\fg_-)^* \ot \fg$.  By a lemma of Kostant, we have the following $G_0$-module isomorphisms:
 \[
 \bigwedge{\!}^k (\fg_-)^* \ot \fg = \lefteqn{\overbrace{\phantom{\im(\partial^*) \op \ker(\Box)}}^{\ker(\partial^*)}}\im(\partial^*) \op \underbrace{\ker(\Box) \op  \im(\partial)}_{\ker(\partial)}, \qquad
  \ker(\Box) \cong \frac{\ker(\p^*)}{\im(\p^*)} \cong \frac{\ker(\p)}{\im(\p)} =: H^k(\fg_-,\fg).
 \]
 By $G_0$-equivariancy, $\Kh : \cG_0 \to H^2(\fg_-,\fg)$ maps fibres of $\cG_0 \to M$ to $G_0$-orbits in $H^2(\fg_-,\fg)$.    By regularity, $\Kh$ is valued in the $\fg_0$-submodule $H^2_+(\fg_-,\fg) \subseteq H^2(\fg_-,\fg)$ on which the grading element $Z$ acts with positive homogeneities.  In the complex case, Kostant's theorem (Recipe \ref{R:Kostant}) completely describes $H^2(\fg_-,\fg)$.
 
 \subsection{Filtration and embedding of the symmetry algebra} 
 \label{S:CNK} 
 
 We review some recent work.

 \subsubsection{\v{C}ap and Neusser \cite{Cap2005b, CN2009}} Let $(\cG \stackrel{\pi}{\to} M, \omega)$ be a regular, normal geometry of type $(G,P)$.  Then:
 \begin{enumerate}
 \item[(i)]  Let $u \in \cG$, and $x = \pi(u)$.  The map $\xi\mapsto\omega(\xi(u))$ is a linear injection $\widetilde\cS := \finf(\cG, \omega) \inj \fg$.\footnote{They state this for the global automorphism algebra $\aut(\cG,\omega)$, but their proof works also for $\finf(\cG, \omega)$.}
 \item[(ii)]   Letting $\ff(u) := \im(\omega_u|_{\finf(\cG,\omega)}) \subset \fg$, the bracket on $\finf(\cG, \omega)$ is mapped to the operation
 \begin{align} \label{E:def-bracket}
  [X,Y]_{\ff(u)} := [X,Y] - \kappa_u(X,Y), \qquad \forall X,Y \in \ff(u).
 \end{align}
 \item[(iii)] By regularity, restricting the canonical filtration on $\fg$ to $\ff(u)$ makes the latter into a filtered Lie algebra (which is generally not a filtered Lie subalgebra of $\fg$). 
 \item[(iv)]  $\fs(u) := \gr(\ff(u))$ is a graded subalgebra of $\fg$.  Note $\dim(\finf(\cG,\omega)) = \dim(\fs(u))$.
 \item[(v)]$\fs_0(u) \subseteq \fann(\Kh(u)) \subseteq \fg_0$.
 \item[(vi)] If $\Kh(u) \neq 0$, then (v) and Kostant's theorem (Recipe \ref{R:a0}) yield a bound on $\dim(\fs_0(u))$.
 \end{enumerate}
  
 Although (v) is stated in \cite[Cor.\ 5 (3)]{CN2009} for bracket-generating distributions, it holds in general.  Indeed, given $\xi \in \finf(\cG,\omega)$, if $Y = \omega(\xi(u)) \in \fp$, then the curve $\varphi_t(u) = u \cdot \exp(tY)$ has vertical tangent vector $\xi(u) = \zeta_Y(u) = \l.\frac{d}{dt}\r|_{t=0} \varphi_t(u)$.  By $P$-equivariancy of $\Kh$,
 \begin{align*}
 0 = (\xi \cdot \Kh)(u) =  \l.\frac{d}{dt}\r|_{t=0} \Kh(\varphi_t(u)) = \l.\frac{d}{dt}\r|_{t=0} \exp(-tY) \cdot (\Kh(u)) = -Y \cdot (\Kh(u)).
 \end{align*}
 Complete reducibility of $H^2_+(\fg_-,\fg)$ implies that this action depends only on $Y \mod \fp_+ \in \fp / \fp_+ \cong \fg_0$.  
  
 \begin{remark} \label{RM:CN}
 For non-flat geometries, \v{C}ap \& Neusser bound $\dim(\finf(\cG,\omega))$ via the maximum dimension of any proper graded subalgebra $\fb \subset \fg$ with $\dim(\fb_0)$ no larger than the Kostant bound.  However, this strategy lacks finer information about the annihilating subalgebra, i.e.\ they use \eqref{E:dim-a0}, while we will additionally use \eqref{E:a0}.  Thus, in general, their upper bounds are not sharp.
 \end{remark}

 The following result is stated in \cite[\S 2.5]{Cap2005b}.

 \begin{prop} \label{P:loc-flat}
 Let $(\cG \to M, \omega)$ be a regular, normal parabolic geometry of type $(G,P)$.  If $\dim(\finf(\cG,\omega)) = \dim(\fg)$, then the geometry is flat.
 \end{prop}
 
 \begin{proof}
 Let $u \in \cG$, so $\dim(\finf(\cG,\omega)) = \dim(\fs(u)) \leq \dim(\fg)$.  If equality holds, then the grading element has $Z \in \fs_0(u) \subseteq \fann(\Kh(u))$.  Since $\Kh(u) \in H^2_+(\fg_-,\fg)$, this forces $\Kh(u) = 0$.
 \end{proof}

 The filtration on $\fg$ induces ($P$-invariant) filtrations on $T\cG = T^{-\nu} \cG \supset ... \supset T^\nu \cG$ and $\wt\cS$:
  \begin{align} \label{E:S-filtration}
 \wt\cS(x)^i &= \{ \xi \in \wt\cS \mid \xi(u) \in T^i_u \cG,\,\, \forall u \in \pi^{-1}(x) \}, \qquad -\nu \leq i \leq \nu.
 \end{align}
 This projects to a filtration $\cS = \cS(x)^{-\nu} \supset ... \supset \cS(x)^\nu$ of the symmetry algebra $\cS$ of the underlying structure on $M$.  Note $\omega_u(\wt\cS(x)^i) = \ff(u)^i$, and $\fs(u) = \gr(\ff(u)) \cong \gr(\cS(x))$ is a canonical isomorphism, so the latter will be denoted by $\fs(x)$.

 \subsubsection{Kruglikov \cite{Kru2011}}   For bracket-generating distributions (not necessarily underlying parabolic geometries), we have a filtration $\{ \cS(x)^i \}$ of $\cS$.  Let $TM = D^{-\nu} \supset ... \supset D^{-1} =: D$ generated by $D$ with each $D^i$ of constant rank.\footnote{These conditions hold for underlying structures of parabolic geometries.}  Given $x \in M$, define $\fg_i(x) := D^i_x / D^{i+1}_x$ for $i < 0$.  The Tanaka prolongation $\fg(x) = \prn(\fg_-(x))$ has $\fg_i(x) \inj \Hom(\ot^{i+1} \fg_{-1}(x), \fg_{-1}(x))$ for $i \geq 0$, cf. Remark \ref{RM:Hom}.
 
 Letting $\ev_x : \cS \to T_x M$ be the evaluation map, define $\cS(x)^i :=\ev_x^{-1}(D^i)$ for $i < 0$, and $\cS(x)^0 := \ker(\ev_x)$.  Inductively, for $i \geq 0$, given $\bX\in\cS(x)^i$, there is a map $\Psi^{i+1}_\bX:\ot^{i+1}D_x\to T_xM$ given by
 \begin{align*}
 \Psi^{i+1}_\bX(Y_1,\dots,Y_{i+1})=\bigl[\bigl[\dots\bigl[[\bX,\bY_1],\bY_2\bigr],\dots\bigr],\bY_{i+1}\bigr](x),
 \end{align*}
 where $\bY_j \in \Gamma(D)$ and $\bY_j(x) = Y_j$.  Moreover, $\im(\Psi^{i+1}_\bX) \subseteq D_x$.  Define $\cS(x)^{i+1} :=\{\bX\in\cS(x)^i \mid \Psi_\bX^{i+1}=0\}$.  Then $\cS = \cS(x)$ is a filtered Lie algebra, and $\fs_i(x) \inj \fg_i(x)$ via $\bX \mod \cS(x)^{i+1} \mapsto \Psi^{i+1}_\bX$.
 
 There is a filtration\footnote{In the complex setting, we replace $C^\infty(M)$ by the algebra of germs of holomorphic functions about the point $x$.} on $C^\infty(M)$ by ideals $\mu^i_x$.  Let $\mu^1_x = \{ f \in C^\infty(M) \mid f(x) = 0 \}$ and for $i \geq 0$,
 \begin{align*}
 \mu^{i+1}_x = \{ f \in C^\infty(M) \mid (\bY_1 \cdots \bY_t \cdot f)(x) = 0, \,\forall \bY_j \in \Gamma(D), \, 0 \leq t \leq i \}.
 \end{align*}
 For any $\bY \in \Gamma(D)$, we have $\bY \cdot \mu^{i+1}_x \subset \mu^i_x$.  Let $\{ \bZ_{jk} \}_{j=1,...,\nu}$ be a local framing near $x$ with $\bZ_{jk} \in \Gamma(D^{-j}) \backslash \Gamma(D^{-j-1})$.  Let $\bX \in \cS$.  Then
 \begin{align} \label{E:Boris-filtration}
 i \geq 0: \qquad \bX = \sum_{j,k} f_{jk} \bZ_{jk}\in\mathcal{S}(x)^i \qbox{iff} f_{jk}\in\mu_x^{i+j},\quad \forall j,k.
 \end{align}
  The ``only if'' direction is proved in \cite[Lemma 1]{Kru2011}, while an easy induction establishes the converse.
 
 \begin{example}[$(2,3,5)$-geometry] \label{EX:G2P1-S-filtration} On a 5-manifold $(x,y,p,q,z)$, consider $D$ locally spanned by 
 \[
 \p_q, \qquad \tilde\p_x := \p_x + p \p_y + q \p_p + (p^3+q^2) \p_z.
 \]
   Then the symmetry algebra $\cS$ of $D$ is spanned by
 \[
 \bX_1 = \p_x, \quad \bX_2 = \p_y, \quad \bX_3 = \p_z, \quad \bX_4 = x\p_x - y\p_y - 2p\p_p - 3q\p_q - 5z\p_z.
 \]
 At $\bu_0 = (x_0,y_0,p_0,q_0,z_0)$, the dimensions of $\fs_i = \fs_i(\bu_0)$ are
 \[
 \begin{array}{c|cccc}
 & \fs_{-3} &  \fs_{-2} &  \fs_{-1} &  \fs_{0} \\ \hline
 q_0 \neq 0 \mbox{ or } p_0 \neq 0 & 2 & 1 & 1 & 0\\
 q_0 = p_0 = 0 & 2 & 0 & 1 & 1
 \end{array} 
 \]
 When $p_0=q_0=0$, $\cS(\bu_0)^0$ is spanned by $\bT = \bX_4 - x_0 \bX_1 + y_0 \bX_2 + 5z_0 \bX_3$.  Writing $\bT = \sum_{j,k} f_{jk} \bZ_{jk}$ in the framing $(\bZ_{11}, \bZ_{12}, \bZ_{21}, \bZ_{31}, \bZ_{32}) = (\p_q, \tilde\p_x, \p_p + 2q\p_z, \p_y, \p_z)$ shows $f_{jk} \in \mu^j_x$, so $\bT \not\in \cS(\bu_0)^1$.  Note $[\bT,\p_q] = 3\p_q$, which at $\bu_0$ is not contained in $\fs_{-1}$.  Thus, $[\fs_0,\fg_{-1}] \not\subset \fs_{-1}$ when $p_0=q_0=0$.
 \end{example}

 \begin{remark}
 The above results in Kruglikov's approach still hold if $D$ admits additional structure.  In the parabolic case, Kruglikov's filtration $\{ \cS(x)^i \}$ agrees with that of \v{C}ap--Neusser, and \eqref{E:Boris-filtration} gives a means to compute it explicitly.  Henceforth, we will mainly rely on the \v{C}ap--Neusser presentation.
 \end{remark}
 
 \subsection{Proof of the upper bound}
 \label{S:upper}
 
 Fixing $(G,P)$, the {\em submaximal symmetry dimension} $\fS$ is the maximum of  $\dim(\finf(\cG,\omega))$ among all {\bf regular, normal} geometries $(\cG \stackrel{\pi}{\to} M, \omega)$ of type $(G,P)$ which are {\bf not flat}, i.e.\ $\Kh \not\equiv 0$.  Equivalently, $\fS$ maximizes $\dim(\cS)$, where $\cS$ is the symmetry algebra of an underlying structure which is not flat.
 
 \begin{defn} \label{D:reg-pt}
 We say that $x \in M$ is a {\em regular point} if there exists a neighbourhood $N_x \subset M$ of $x$ such that for $-\nu \leq j \leq \nu$, $\dim(\fs_j(y))$ is a constant function of $y \in N_x$.  Otherwise, $x$ is {\em irregular}.
 \end{defn}
 
 We have a tower of bundles $\cG =: \cG_\nu \ra ... \ra \cG_0 \to M$, with $\cG_i = \cG / P_+^{i+1} = \cG / \exp(\fg^{i+1})$, projections $p_i : \cG \to \cG_i$ and $\pi_i : \cG_i \to M$, where $i=-1,..., \nu$.  (If $i=-1$, let $p_{-1} = \pi$, i.e.\ the projection to $M$.)  Given any $\xi \in \fX(\cG)^P$, let $\xi^{(i)} = (p_i)_* \xi \in \fX(\cG_i)$.  By $P$-invariancy, $\wt\cS$ descends to $\cS^{(i)} \subset \fX(\cG_i)^{P / P^{i+1}_+}$.  Given $x \in M$, the filtration $\wt\cS(x)$ projects to a filtration $\cS^{(i)}(x)$, and from \eqref{E:S-filtration}, 
 \[
 \cS^{(i)}(x)^{i+1} = \{ \xi^{(i)} \in \cS^{(i)} \mid \xi^{(i)}(u_i) = 0, \,\forall u_i \in \pi_i^{-1}(x) \}.
 \]
For any $u_i \in \pi_i^{-1}(x)$, $\cS^{(i)}|_{u_i} \subset T_{u_i} \cG_i$ has dimension $\dim\, \cS - \sum_{j=i+1}^\nu \dim\, \fs_j(x)$.  If $x$ is a regular point, then $\cS^{(i)}$ yields a constant rank distribution on a neighbourhood of $\pi_i^{-1}(x) \subset \cG_i$.  By the Frobenius theorem, there exist local rectifying coordinates and a local foliation by maximal integral submanifolds of $\cS^{(i)}$.  
 
 Motivated by Lemma \ref{L:T-subalg} and Theorem \ref{T:Y-pr-thm}, we define a variant of Tanaka prolongation:
 
  \begin{defn} \label{D:g-prolong} Let $\fg$ be a $\bbZ$-graded semisimple Lie algebra, and $\fa_0 \subset \fg_0$ a subalgebra.  Define:
  \[
   {\rm (i)}: \,\, \fa_{\leq 0} := \fg_{\leq 0}; \qquad {\rm (ii)}:\,\, \fa_i = \{ X \in \fg_i \mid [X,\fg_{-1}] \subseteq \fa_{i-1} \}, \quad i > 0. 
  \]
  Then $\fa = \bop_i \fa_i$ is a graded subalgebra of $\fg$ denoted by $\prn^\fg(\fg_-,\fa_0)$.  (In particular, $\prn^\fg(\fg_-,\fg_0) = \fg$.)
 \end{defn}

 If $G$ is (complex) simple and $G/P$ is not $A_\rkg / P_1$ (projective) or $C_\rkg / P_1$ (contact projective), then $\prn^\fg(\fg_-,\fa_0) \cong \prn(\fg_-,\fa_0)$ by Theorem \ref{T:Y-pr-thm} and Lemma \ref{L:T-subalg}.
 
 We have the following ``filtered'' generalization of a result of Morimoto \cite[Prop.\ 4.1]{Mor1977}.

 \begin{prop} \label{P:key-bracket} Let $x \in M$ be a regular point and $u \in \pi^{-1}(x)$.  Then $[\fs_{i+1}(u),\fg_{-1}] \subset \fs_i(u)$, $\forall i$.  In particular, $\fs_{i+1}(u) \subseteq \prn^\fg_{i+1}(\fg_-,\fs_0(u))$.
 \end{prop}
 
 \begin{proof}
 Suppose $i \geq -1$.  For any $\xi \in \wt\cS = \finf(\cG,\omega)$ and $\eta \in \fX(\cG)$, we have $0 = (\cL_\xi \omega)(\eta) = \xi \cdot \omega(\eta) - \omega([\xi,\eta])$.  Let $\xi \in \wt\cS(x)^{i+1}$ and $\eta \in \Gamma(T^{-1} \cG)^P$, so $X = \omega(\xi(u)) \in \ff(u)^{i+1} \subset \fg^{i+1}$ and $Y = \omega(\eta(u)) \in \fg^{-1}$.  Since $i \geq -1$, then $X \in \fp$, so $X =  \omega(\zeta_X(u))$, where $\zeta_X$ is a fundamental vertical vector field.  By $P$-equivariancy of $\omega(\eta)$, 
 \begin{align*}
 \omega([\xi,\eta](u)) &= (\xi \cdot \omega(\eta))(u) = (\zeta_X \cdot \omega(\eta))(u) = \l.\frac{d}{dt}\r|_{t=0} \omega(\eta)(u \cdot \exp(tX)) = -[X,Y].
 \end{align*}
 In particular, since $[X,Y] \in \fg^i$, then $[\xi,\eta](u) \in T^i_u \cG$.
 
 Let $u_i = p_i(u) \in \cG_i$.  Since $\xi \in \wt\cS(x)^{i+1}$, then $\xi^{(i)}(u_i) = 0$.  Since $x$ is a regular point, there exist local (functionally independent) functions $F_1,..., F_{t_i}$ on $\cG_i$ (smooth in the real setting or holomorphic in the complex setting), where $t_i = \dim\,\cG_i - \sum_{-\nu}^i\dim\,\fs_j(x)$, whose joint level sets define the local foliation by maximal integral submanifolds of $\cS^{(i)}$.  In particular, for our $\xi \in \wt\cS(x)^{i+1}$, we have $\xi^{(i)} \cdot F_j = 0$ for $j=1,..., t_i$, and since $\xi^{(i)}(u_i) = 0$, we have $ ([\xi^{(i)},\eta^{(i)}] \cdot F_j)(u_i) = 0$. Thus, $[\xi^{(i)},\eta^{(i)}](u_i) \in \cS^{(i)}|_{u_i} = (p_i)_*( \wt\cS|_u )$ is tangent to the foliation.  Since any element of $\wt\cS$ is uniquely determined by its value at $u$, then $[\xi^{(i)},\eta^{(i)}](u_i) = \xi'^{(i)}(u_i)$ for some $\xi' \in \wt\cS$.  Equivalently,
 \[
 [\xi,\eta](u) = \xi'(u) + \chi(u) \in T^i_u \cG, \qquad \xi'(u) \in \wt\cS|_u, \quad \chi(u) \in T^{i+1}_u \cG.
 \]
 Hence, $\omega([\xi,\eta](u)) \in \ff(u)^i + \fg^{i+1}$.  Thus, $[\ff(u)^{i+1}, \fg^{-1}] \subset \ff(u)^i + \fg^{i+1}$, so $[\fs_{i+1}(u),\fg_{-1}] \subset \fs_i(u)$.
 
Now suppose $i \leq -2$.  Since $x$ is regular, then $D^j_\cS|_y := D^j_y \cap \cS|_y$, $j \leq -1$, define constant rank distributions near $x$.  Let $\bX = \sum_i f_i \bX_i \in \Gamma(D^{i+1}_\cS)$ where $\{ \bX_i \}$ is basis of $\cS$.  Given $\bY \in \Gamma(D)$,
\[
 [\bX,\bY](x) = \sum_i f_i(x) [\bX_i,\bY](x) - (\bY \cdot f_i)(x) \bX_i(x) \in (D_x + \cS|_x) \cap D^i_x.
\]
 Quotient by $D^{i+1}_x \supset D_x$, so with respect to the Levi bracket, $[\fs_{i+1}(x),\fg_{-1}(x)] \subset \fs_i(x)$, hence the first claim.  The second claim then follows immediately.
 \end{proof}

 If $\cS$ is {\em transitive at $x$}, i.e.\ $\ev_x(\cS) = T_x M$, then $[\fs_{i+1}(x),\fg_{-1}(x)] = [\fs_{i+1}(x),\fs_{-1}(x)] \subset \fs_i(x)$ immediately follows since $\fs(x)$ is a graded Lie algebra.  Equivalently, $[\fs_{i+1}(u),\fg_{-1}] \subset \fs_i(u)$.

 \begin{example} In Example \ref{EX:G2P1-S-filtration}, the regular points are precisely those with $p_0 \neq 0$ or $q_0 \neq 0$.  At all irregular points, we saw that $[\fs_0, \fg_{-1}] \not\subset \fs_{-1}$.
 \end{example}

 \begin{lemma} \label{L:dense} The set of regular points is open and dense in $M$.
 \end{lemma}

 \begin{proof} For $-\nu \leq i \leq \nu$, define $q_i^\pm : M \to \bbZ$ by $q^-_i(x)= \sum_{j=-\nu}^i \dim\,\fs_j(x)$ and $q^+_i(x)=\dim\,\cS-q^-_i(x)=\dim\,\cS(x)^{i+1} = \sum_{j=i+1}^\nu \dim\,\fs_j(x)$.  Then $q_i^-$ is lower semi-continuous because linear independence of $\xi^{(i)}_1,\dots,\xi^{(i)}_s \in \cS^{(i)}$ at $u_i \in \cG_i$ persists near $u_i$.  Hence, $q_i^+$ is upper semi-continuous.  Given $U \subset M$, define $m_i(U) = \{ x \in U \mid q_i^+(x) \leq q_i^+(y),\, \forall y \in U \}$.  Upper semi-continuity implies that if $\emptyset \neq U \subset M$ is open, then $\emptyset \neq m_i(U) \subset M$ is open.
 
 Denote $M_0 := M$, and define $N_{j+1} = m_{-\nu} \circ\, \cdots\, \circ m_\nu(M_j)$ and $M_{j+1} = M \backslash (\bar{N}_1 \cup \cdots \cup \bar{N}_j)$ for $j \geq 1$.  Each $M_j \subset M$ is open, so $N_{j+1} \subset M$ is open; when the former is non-empty, so is the latter.  Each $\im(q_i^+) \subset [0,\dim\,\cS] \cap \bbZ$ is finite, so there exists a minimal $k \geq 0$ with $M_{k+1} = \emptyset$.  The open set $N = N_1 \cup \cdots \cup N_k$ is the set of all regular points, which is dense in $M$ since $\bar{N} = M$.
 \end{proof}
 
 We are now in a position to establish our universal upper bound.  For $0 \neq \phi \in H^2_+(\fg_-,\fg)$, define
 \begin{align} \label{E:fU}
 \fa^\phi := \prn^\fg(\fg_-,\fann(\phi)), \qquad \fU := \max\l\{  \dim(\fa^\phi) \mid 0 \neq \phi \in H^2_+(\fg_-,\fg) \r\} < \dim(\fg).
 \end{align}
 (If $\fU = \dim(\fg)$, then $\fa^\phi = \fg$ for some $\phi \neq 0$, and $Z \in \fann(\phi)$, which is a contradiction.)

 \begin{thm} \label{T:upper} Let $G$ be a semisimple Lie group and $P \subset G$ a parabolic subgroup.  
 Let $(\cG \stackrel{\pi}{\to} M, \omega)$ be a regular, normal geometry of type $(G,P)$, $x \in M$ a regular point, and $u \in \pi^{-1}(x)$.  Then $\fs(u) \subseteq \fa^{\Kh(u)}$ is a graded subalgebra, and $\dim(\finf(\cG,\omega)) \leq \dim(\fa^{\Kh(u)})$.  Moreover, $\fS \leq \fU < \dim(\fg)$.
 \end{thm}
 
 \begin{proof} We have $\fs(u) \subseteq \fg_- \op \fs_{\geq 0}(u) \subseteq \prn^\fg(\fg_-,\fs_0(u)) \subseteq \fa^{\Kh(u)}$, using Proposition \ref{P:key-bracket} and $\fs_0(u) \subseteq \fann(\Kh(u))$.  Thus, $\dim(\finf(\cG,\omega)) = \dim(\fs(u)) \leq \dim(\fa^{\Kh(u)})$.
 
 If the geometry is not flat, then $N = \{ x \in M \mid \Kh(u) \neq 0, \, \forall u \in \pi^{-1}(x) \}$ is non-empty and open, so by Lemma \ref{L:dense}, $N$ contains a non-flat regular point $x$.  At any $u \in \pi^{-1}(x)$, the previous bound applies.  Since $\Kh(u) \in H^2_+(\fg_-,\fg)$, then $\fS \leq \fU < \dim(\fg)$.
 \end{proof}
 
 \begin{remark}
 Only $[\fs_{i+1}(u),\fg_{-1}] \subseteq \fs_i(u)$ for $i \geq 0$ was essential for proving Theorem \ref{T:upper}.  This relied only on local constancy of $\dim(\fs_j(y))$ for $j > 0$, so our ``regular point'' definition could be weakened.  However, $[\fs_0(u),\fg_{-1}] \subseteq \fs_{-1}(u)$ will be used in proving Theorem \ref{T:Petrov}.
 \end{remark}

 Let $(\fg,\fp)$ be real with complexification $(\fg^\bbC,\fp^\bbC)$. Define 
 \[
 \fU^\bbC := \max\l\{  \dim(\fa^\phi) \mid 0 \neq \phi \in H^2_+(\fg^\bbC_-,\fg^\bbC) \r\}.
 \]
 
 \begin{cor} \label{C:upper}
 For {\em real} (regular, normal) parabolic geometries of type $(G,P)$, $\fS \leq \fU \leq \fU^\bbC$.
 \end{cor}
 
 \begin{proof}
 We have $H^2_+(\fg^\bbC_-,\fg^\bbC) \cong H^2_+(\fg_-,\fg) \ot_\bbR \bbC$.  In Proposition \ref{P:Tanaka-lw}, we show that over $\bbC$, $\dim(\fa^\phi)$ is maximized on some extremal vector in some sub-irrep.  Since there may not exist vectors $H^2_+(\fg_-,\fg)$ which complexify to an extremal vector of $H^2_+(\fg^\bbC_-,\fg^\bbC)$, then we immediately obtain $\fU \leq \fU^\bbC$.
 \end{proof}
 
 \begin{cor} \label{C:transitive} For regular, normal geometries of type $(G,P)$, suppose that $\fS = \fU$.  Then any submaximally symmetric model $(\cG \stackrel{\pi}{\to} M, \omega)$ is locally homogeneous near a non-flat regular point $x \in M$.
 \end{cor}
 
 \begin{proof} By Theorem \ref{T:upper}, $\fU = \fS = \dim(\finf(\cG,\omega)) = \dim(\fs(u)) \leq \dim(\fa^{\Kh(u)}) \leq \fU$, $\forall u \in \pi^{-1}(x)$.  Hence, $\fs(u) = \fa^{\Kh(u)} \supset \fg_-$, so $\fs(x) \supset \fg_-$, i.e.\ $\cS$ is transitive at $x$. Lie's third theorem implies the result.
 \end{proof}
 
 \begin{remark} \label{RM:constrained} Recall that the fibres of $\cG \to M$ are mapped by $\Kh$ into $G_0$-orbits in $H^2_+(\fg_-,\fg)$.  Let $\{ 0 \} \neq \cO \subset H^2_+(\fg_-,\fg)$ be a $G_0$-invariant subset, e.g.\ a $G_0$-orbit or a $G_0$-submodule.  Analogously defining $\fS_\cO$ and $\fU_\cO$ using the constraint $\im(\Kh) \subset \cO$, we obtain $\fS_\cO \leq \fU_\cO$ similar to Theorem \ref{T:upper}.  If $\fS_\cO = \fU_\cO$, then any model realizing the equality must be locally homogeneous near a non-flat regular point.
 \end{remark}
 
 We will use Remark \ref{RM:constrained} in \S \ref{S:4d-Lor} to study the Petrov types in 4-dimensional Lorentzian conformal geometry, and in \S \ref{S:G2P1} for $(2,3,5)$-distributions with $\Kh$ having constant root type.
 
 \section{Prolongation analysis}
 \label{S:PR-analysis}
 
 The universal upper bound $\fU$ of Theorem \ref{T:upper} was defined in \eqref{E:fU} via the Tanaka algebra $\fa^\phi$.  Observe that if $H^2_+(\fg_-,\fg) = \bop_i \bbV_i$ is the decomposition into $\fg_0$-irreps and $\phi = \sum_i \phi_i$, where $\phi_i \in \bbV_i$, $\forall i$, then $\fann(\phi) \subset \fann(\phi_i)$ and $\fa^\phi \subset \fa^{\phi_i}$, so $\fU = \max_i \{ \max\{ \dim(\fa^{\phi_i}) \mid 0 \neq \phi_i \in \bbV_i \} \}$.
 Thus, it suffices to understand $\max\{ \dim(\fa^\phi) \mid 0 \neq \phi \in \bbV \}$, where $\bbV$ is a $\fg_0$-irrep.  
 In the complex case, this maximum is realized on any extremal weight vector $\phi_0 \in \bbV$ (\S \ref{S:max-Tanaka}).  The structure of $\fa^{\phi_0}$ (in particular, its dimension) can be efficiently deduced from a Dynkin diagram recipe (\S \ref{S:Dynkin}).  We investigate the notion of {\em prolongation-rigidity} in \S \ref{S:PR}, and the behaviour of $\fa^{\phi_0}$ under correspondence and twistor space constructions in \S \ref{S:corr-twistor}.

  
  \subsection{Maximizing the Tanaka prolongation}
  \label{S:max-Tanaka}

 In the {\em complex} setting, given a $G_0$-irrep $\bbV$, where $G_0$ is reductive (with non-trivial semisimple part), a well-known consequence of the Borel fixed point theorem is that $\bbP(\bbV)$ has a {\em unique} Zariski closed $G_0$-orbit $\cO$, namely the orbit of any extremal line $[\phi_0]$.  This is the unique orbit of minimal dimension, so $\max\{ \dim(\fann(\phi) ) \mid 0 \neq \phi \in \bbV \}$ is realized precisely when $[\phi] \in \cO$.  This is the $k=0$ assertion in the proposition below.

 \begin{prop} \label{P:Tanaka-lw} Let $G$ be a complex \ss Lie group, and $P$ a parabolic subgroup.  Let $\bbV$ be an $G_0$-irrep, and $\phi_0 \in \bbV$ an extremal $\fg_0$-weight vector.  Then $\forall 0 \neq \phi \in \bbV$,  $\dim(\fa^\phi_k) \leq \dim(\fa^{\phi_0}_k)$, $\forall k \geq 0$.
 Thus, 
 \[
 \max\{ \dim(\fa^\phi) \mid 0 \neq \phi \in \bbV \} = \dim(\fa^{\phi_0}).
 \]
 Moreover, $\dim(\fa^\phi) = \dim(\fa^{\phi_0})$ iff $[\phi] \in G_0 \cdot [\phi_0] \subset \bbP(\bbV)$.
 \end{prop}
  
  \begin{proof} If $\fg_0^{ss} = 0$, then irreducibility implies $\bbV \cong \bbC$, and the result is immediate.  So suppose $\fg_0^{ss} \neq 0$ and let $k \geq 1$. Define $\Psi_k : \bbP( \bbV ) \to \bbZ_{\geq 0}$ by $\Psi_k([\phi]) = \dim(\fa_k^\phi)$, which is constant on $G_0$-orbits since $\fann(g \cdot \phi) = \Ad_g(\fann(\phi))$ and $\fa_k^{g \cdot \phi} = \Ad_g(\fa_k^\phi)$. From the definition \eqref{E:fU} of $\fa^\phi$,
  \begin{align*}
  \fa_k^\phi &= \prn_k^\fg(\fg_-,\fann(\phi)) = \{ X \in \fg_k \mid \ad_{\fg_{-1}}^k(X) \cdot \phi = 0 \}\\
  &= \{ X \in \fg_k \mid (\ad_{Y_{i_1}} \circ ... \circ \ad_{Y_{i_k}}(X)) \cdot \phi = 0,\,\, \forall Y_{i_j} \in \fg_{-1} \}.
  \end{align*}

 The function $(\ad_{Y_{i_1}} \circ ... \circ \ad_{Y_{i_k}}(X)) \cdot \phi$ is a multilinear function of $X,Y_{i_1},...,Y_{i_k},\phi$.  Its vanishing is determined by evaluating all $Y_{i_1},...,Y_{i_k}$ on all tuples of basis elements $\{ e_i \}$ of $\fg_{-1}$.  So there is a bilinear function $T(X,\phi)$ such that $\fa_k^\phi = \{ X \in \fg_k \mid T(X,\phi) = 0 \}$.  In any basis of $\fg_k$, there is a matrix $M(\phi)$ such that $X \in \fa_k^\phi$ iff its coordinate vector is in $\ker(M(\phi))$.  Since the rank of a matrix is a lower semi-continuous function of its entries, and $M(\phi)$ depends linearly on $\phi$, then $\rnk(M(\phi))$ is lower semi-continuous in $\phi$ (and depends only on $[\phi]$). We have $\dim(\fa_k^\phi) = \dim(\fg_k) - \rnk(M(\phi))$, so $\Psi_k$ is upper semi-continuous.

 We know that $\cO = G_0 \cdot [\phi_0]$ is the {\em unique} Zariski closed $G_0$-orbit in $\bbP(\bbV)$, so $[\phi_0]$ is in the Zariski closure of {\em every} $G_0$-orbit in $\bbP(\bbV)$.  Hence, $\forall 0 \neq \phi \in \bbV$, there is a sequence $\{ g_n \} \subset G_0$ such that $g_n \cdot [\phi] = [g_n \cdot \phi] \to [\phi_0]$ as $n \to \infty$.  Thus, $\Psi_k([\phi]) = \Psi_k([g_n\cdot\phi]) = \Psi_k(g_n \cdot [\phi]) \leq \Psi_k([\phi_0])$ by upper semi-continuity of $\Psi_k$.
 
 The final claim follows from the discussion preceding this lemma on the $k=0$ case.
 \end{proof}
 
  \begin{remark} In general, $\fann(\phi) \not\subset \fann(\phi_0)$, and $\fa^\phi_k \not\subset \fa^{\phi_0}_k$.  
 \end{remark}
 
 The recipe to compute $\fa_0^{\phi_0} = \fann(\phi_0)$ is well-known (see Recipe \ref{R:a0}).  In \S \ref{S:Dynkin}, we show that $\fa_+^{\phi_0}$ is also easily determined.  First, we fix notations from representation theory.

 
 \subsection{Representation theory notations and conventions}
 \label{S:rep-th}
 
 Fix a Borel subalgebra $\fb \subset \fg$ with associated Cartan subalgebra $\fh \subset \fb$, root system $\Delta \subset \fh^*$, and simple roots $\Delta^0 = \{ \alpha_1,..., \alpha_\rkg \}$, where $\rkg = \dim(\fh) = \rnk(\fg)$.  (We use the Bourbaki ordering.)   Let $\{ Z_1,...,Z_\rkg \} \subset \fh$ be the dual basis to $\Delta^0 \subset \fh^*$. If $S \subset \{ 1,..., \rkg \}$, then we let $Z_S := \sum_{k \in S} Z_k$.
  For any $\alpha \in \Delta$, let $e_\alpha$ denote a root vector and let $\fg_\alpha$ be its root space. If a subspace $U \subset \fg$ is a direct sum of root spaces and possibly some subspace of $\fh$, let $\Delta(U)$ denote the corresponding roots.  Thus, $\fb = \fh \op \bop_{\alpha \in \Delta^+} \fg_\alpha$ and $\Delta(\fb) = \Delta^+$.  If $\fg$ is simple, the unique highest root is $\lambda_\fg$. The Killing form $B$ induces a symmetric pairing $\langle \cdot, \cdot \rangle$ on $\fh^*$, which determines the Cartan matrix $c_{ij} = \langle \alpha_i, \alpha_j^\vee \rangle \in \bbZ$ (where $\alpha^\vee = \frac{2\alpha}{\langle \alpha, \alpha \rangle}$) and Dynkin diagram $\cD(\fg)$.  The nodes $\cN(j) = \{ i \mid c_{ij} \leq -1 \}$ are the neighbours of the $j$-th node of $\cD(\fg)$.  Let $\{ \lambda_1,..., \lambda_\rkg \} \subset \fh^*$ denote the fundamental weights of $\fg$, i.e.\ $\langle \lambda_i, \alpha_j^\vee \rangle = \delta_{ij}$. Given a weight $\lambda \in \fh^*$, we have $\lambda = \sum_i r_i(\lambda) \lambda_i$, where $r_i(\lambda) = \langle \lambda, \alpha_i^\vee \rangle$, and $\lambda$ is $\fg$-dominant iff $r_i(\lambda) \in \bbZ_{\geq 0}$, $\forall i$.  Encode $\lambda$ on $\cD(\fg)$ by inscribing $r_i(\lambda)$ over the $i$-th node.  The Cartan matrix is the transition matrix between $\{ \alpha_i \}$ and $\{ \lambda_i \}$, i.e.\ if $\lambda = \sum_i m_i \alpha_i = \sum_i r_i \lambda_i$, then $\sum_i m_i c_{ij} = r_j$. 
 
 To each (standard) parabolic subalgebra $\fp \supset \fb$, there is an index set $I_\fp := \{ i \mid \fg_{-\alpha_i} \not\subseteq \fp \} \subseteq \{ 1,..., \rkg \}$.  Conversely, given $I \subset \{ 1,..., \rkg \}$, the corresponding parabolic is $\fp_I$, which is the sum of $\fb$ and the negative root spaces $\fg_{-\alpha}$ with $\alpha = \sum_i m_i \alpha_i$ satisfying $m_j = 0$, $\forall j \in I_\fp$.   We encode $\fp$ on $\cD(\fg)$ by crossing all nodes in $I_\fp$, and denote this by $\cD(\fg,\fp)$. Thus, $I_\fb = \{ 1, ..., \rkg \}$ and all nodes are crossed.  The grading element is $Z = Z_{I_\fp}$, the $\bbZ$-grading is $\fg = \bop_j \fg_j = \fg_- \op \fp$ with $\fg_j = \{ x \in \fg \mid [Z,x] = j x \}$, and $\fz(\fg_0)$ has basis $\{ Z_i \}_{i \in I_\fp}$.  The grading has depth $\nu = \max\{ Z(\lambda_{\fg'}) \mid \fg' \subset \fg \mbox{ simple ideal} \}$.  Given $\lambda \in \fh^*$, its {\em homogeneity} is $Z(\lambda)$.  If $r_i(\lambda) \in \bbZ_{\geq 0}$, $\forall i \in \{ 1,..., \rkg \} \backslash I_\fp$, then $\lambda$ is $\fp$-dominant.

  Given $S \subset \{ 1, ..., \rkg \}$, the tuple $\widetilde{Z}_S := (Z_k)_{k \in S}$ induces the multi-grading $\fg = \bop_A \fg_A$, where $A = (a_k)_{k \in S} \in \bbZ^S$ is a multi-index, and $\fg_A = \{ X \in \fg \mid [Z_k,X] = a_k X,\, \forall k \in S \}$.  (Note $\fg_0$ with respect to $Z_S$ agrees with $\fg_\0$ with respect to $\widetilde{Z}_S$.  We use the boldface \0 when emphasizing the multi-grading.)  Then $\fz(\fg_\0)$ is spanned by $\{ Z_k \}_{k \in S}$, and an important result \cite[Thm.8.13.3]{Wolf1984} is that $\fg_A$ is a $\fg_\0$-irrep, $\forall A \neq \0$.

 The Weyl group $W \subset \operatorname{O}(\fh^*)$ preserves $\langle \cdot, \cdot \rangle$ and is generated by simple reflections $\sigma_i$, where $\sigma_i(\beta) = \beta - \langle \beta, \alpha_i^\vee \rangle \alpha_i$.  We write $(ij) := \sr_i \circ \sr_j$.  Given $w \in W$, let $|w|$ denote its length.  The affine $W$-action is $w \cdot \lambda := w(\lambda + \rho) - \rho$, where $\rho = \sum_{i=1}^\rkg \lambda_i$.  The inversion set $\Phi_w = w(\Delta^-) \cap \Delta^+$ satisfies $|\Phi_w| = |w|$ and $\sum_{\alpha \in \Phi_w} \alpha = -w\cdot 0$. The {\em Hasse diagram} is $W^\fp = \{ w \in W \mid \Phi_w \subset \Delta(\fg_+) \}$ with length $r$ elements $W^\fp(r)$.  (Equivalently, $w \in W^\fp$ sends $\fg$-dominant weights to $\fp$-dominant weights.)  Every $w \in W^\fp$ corresponds precisely to an element of the $W$-orbit through $\rho^\fp = \sum_{i \in I_\fp} \lambda_i$ under the {\em right} action $(\rho^\fp,w) \mapsto w^{-1} \rho^\fp$.  In particular, $W^\fp(2)$ is efficiently described by Recipe \ref{R:Hasse}.

Dynkin diagram recipes will play an important role in our analysis -- see Appendix \ref{App:Dynkin} and \S \ref{S:Dynkin}.  To compute $H^2(\fg_-,\fg)$ via Kostant's theorem (see Recipe \ref{R:Kostant}), we adopt the commonly used:
 
 \begin{framed}
  {\bf ``Minus lowest weight'' convention:}  We let $\bbV_\mu$ denote the unique $\fg_0$-irrep with {\em lowest} weight $\mu$, so  $-\mu$ is $\fp$-dominant, and $\forall i \in I_\fp$, $Z_i$ acts by the scalar $Z_i(\mu)$.  We denote $\bbV_\mu$ by the Dynkin diagram notation for $-\mu$.  The {\em homogeneity} of $\bbV_\mu$ is $Z(\mu)$.
 \end{framed}

  By Kostant's theorem, each $\fg_0$-irrep in $H^2(\fg_-,\fg)$ is of the form $\bbV_\mu$, where $\mu = -w\cdot\lambda$, with $\lambda$ the highest weight of a simple ideal in $\fg$, and $w \in W^\fp(2)$.  (In particular, $w\cdot\lambda$ is $\fp$-dominant.)
  
 Given $\mu \in \fh^*$ where $-\mu$ is $\fp$-dominant, let 
 \begin{align}
 J_\mu &:= \{ j \not\in I_\fp \mid \langle \mu, \alpha_j^\vee \rangle \neq 0 \} && \mbox{(uncrossed nodes with a nonzero coefficient)}, \label{E:J-mu}\\
 I_\mu &:= \{ j \in I_\fp \mid \langle \mu, \alpha_j^\vee \rangle = 0 \} && \mbox{(crossed nodes with a zero coefficient)}. \label{E:I-mu}
 \end{align}
 We augment our Dynkin diagram notation for $-\mu$ on $\cD(\fg,\fp)$ by putting a {\em square} around $I_\mu$ nodes, an {\em asterisk} on $J_\mu$ nodes, and denote this by $\cD(\fg,\fp,\mu)$.  Define $\fU_\mu := \max\{ \dim(\fa^\phi) \mid 0 \neq \phi \in \bbV_\mu \}$, so that by Proposition \ref{P:Tanaka-lw}, if $\phi_0$ is a lowest weight vector of $\bbV_\mu$, then
 \[
 \fU_\mu = \dim(\fa(\mu)), \qquad \fa(\mu) := \fa^{\phi_0} = \prn^\fg(\fg_-,\fann(\phi_0)).
 \]

 If $\fg$ is {\em simple}, then for $w \in W^\fp(2)$ and $\mu := -w\cdot\lambda_\fg$, we define $J_w := J_\mu$, $I_w := I_\mu$, $\fa(w) := \fa(\mu)$, and $\fU_w := \fU_\mu$. Let $W^\fp_+(2) := \{ w \in W^\fp(2) \mid Z(-w\cdot \lambda_\fg) \geq 1 \}$, so that $H^2_+(\fg_-,\fg) = \bop_{w \in W^\fp_+(2)} \bbV_{-w\cdot \lambda_\fg}$.


 \subsection{A Dynkin diagram recipe}
 \label{S:Dynkin}
 
 In this section, we work over $\bbC$.
 
 \begin{lemma} \label{L:Tan-h-mod} Let $\fg$ be a complex semisimple Lie algebra, and $\fp \subset \fg$ a parabolic subalgebra.  Let $\fa_0 \subset \fg_0$ be a subalgebra and $\fa := \prn^\fg(\fg_-,\fa_0)$.
 \begin{enumerate}
 \item[\rm (a)] For $k \geq 0$, if $\fa_k = 0$, then $\fa_{k+1} = 0$.\footnote{Recall that for $k=0$, we assume under ``Conventions'' that no simple ideal of $\fg$ is contained in $\fg_0$.}
 \item[\rm (b)] If $\fa_0$ is an $\fh$-module, then $\fa_k$ is a direct sum of root spaces, $\forall k > 0$.
 \end{enumerate}
 \end{lemma}
 
 \begin{proof} Let $X \in \fa_{k+1}$, so $[X,\fg_{-1}] \subset \fa_k = 0$.  By \cite[Prop.\ 3.1.2 (5)]{CS2009}, $X=0$, so we obtain claim (a).
  
  Since $\fa_0$ and $\fg_-$ are $\fh$-modules, then so is $\fa_k$, $\forall k > 0$.  Fix $k > 0$.  All root spaces are 1-dimensional, so let $v = \sum_{\alpha \in \cA} e_\alpha \in \fa_k$ be a sum of root vectors, where $\cA \subset \Delta(\fg_k)$.    We show by induction on $N = |\cA|$ that  $e_\alpha \in \fa_k$, $\forall \alpha \in \cA$.   The result is trivial for $N=1$, so assume the general case.  For any $h \in \fh$, $[h,v] = \sum_{\alpha \in \cA} \alpha(h) e_\alpha \in \fa_k$.  Fix $\beta \in \cA$.  Since all $\alpha \in \Delta^+$ (and hence $\cA$) are distinct, then $\exists h \in \fh$ so that $\alpha(h) - \beta(h) \neq 0$, $\forall \alpha \in \cA \backslash \{ \beta \}$.  Hence, $[h,v] - \beta(h) v = \sum_{\alpha \in \cA \backslash \{ \beta \} } (\alpha(h) - \beta(h)) e_\alpha \in \fa_k$.  The induction hypothesis implies claim (b).
 \end{proof}
 
 Consider $(\fg,\fp)$ and $-\mu$ a $\fp$-dominant weight.  The $\bbZ$-grading $\fg = \bop_j \fg_j$ is induced by $Z_{I_\fp}$, but in general $\fg_j$ are not $\fg_0$-irreducible.  Each $\fg_j$ decomposes into $\fg_\0$-irreps using $\widetilde{Z}_{I_\fp}$.
Each of these further decomposes into $\fg_{\0,\0}$-irreps using $\widetilde{Z}_{J_\mu}$.  Writing $1_i = (\delta_{ij})_{j \in I_\fp} \in \bbZ^{I_\fp}$, we observe that for $\fg_1$, we have the {\em irreducible} decompositions:
 \begin{align} \label{E:g00-decomp}
 \begin{array}{lll}
 \mbox{As $\fg_\0$-modules:} & \displaystyle\fg_1 = \bop_{i \in I} \fg_{1_i}, & \Delta(\fg_{1_i}) = \{ \alpha \in \Delta(\fg_1) \mid Z_i(\alpha) = 1 \};\\
 \mbox{As $\fg_{\0,\0}$-modules:}& \displaystyle\fg_{1_i} = \bop_{A \geq \0} \fg_{1_i,A}, & \Delta(\fg_{1_i,A}) = \{ \alpha \in \Delta(\fg_{1_i}) \mid \widetilde{Z}_{J_\mu}(\alpha) = A \}.
 \end{array}
 \end{align}
 Here ``$A \geq \0$'' means {\em all} entries are non-negative.
 The unique lowest root in $\Delta(\fg_{1_i,\0})$ is $\alpha_i$.
 
 \begin{defn}
 If $\cA,\cB \subset \Delta$, define $\cA \,\dot{+}\, \cB := \{ \alpha + \beta \mid \alpha \in \cA, \beta \in \cB \} \cap \Delta$.
 \end{defn}


 \begin{thm} \label{T:DD} Let $\fg$ be a complex semisimple Lie algebra, and $\fp \subset \fg$ a parabolic subalgebra.  If $-\mu$ is a $\fp$-dominant weight and $\fa := \fa(\mu)$, then
 \begin{enumerate}
 \item[\rm (a)] $\fa_k$ is a direct sum of root spaces, $\forall k > 0$;
 \item[\rm (b)] $\Delta(\fa_1) = \bigcup_{i \in I_\mu} \Delta(\fg_{1_i,\0})$.  In particular, $I_\mu = \emptyset$ iff $\fa_+ = 0$.
 \item[\rm (c)] If $I_\mu \neq \emptyset$, then $\fa_1$ is bracket-generating in $\fa_+$, and 
 \[
 \Delta(\fa_k) = \{ \alpha \in \Delta(\fg_k) \mid Z_{I_\mu}(\alpha) = k,\, Z_{J_\mu}(\alpha) = 0 \}, \qquad k \geq 1.
 \]
 In particular, $\forall \alpha \in \Delta(\fa_k)$, we have $Z_j(\alpha) = 0$, $\forall j \in (I_\fp \backslash I_\mu) \cup J_\mu$.
 \end{enumerate}
 \end{thm}
 
 \begin{proof} From Recipe \ref{R:a0}, $\fa_0 = \fann(\phi_0) = \{ H \in \fh \mid \mu(H) = 0 \} \op \bop_{\gamma \in \Delta(\fg_{0,\leq0})} \fg_\gamma$ is an $\fh$-module, so claim (a) follows from Lemma \ref{L:Tan-h-mod}.

 Recall that $\forall 0 \neq \gamma \in \fh^*$, $\exists ! H_\gamma \in \fh$ such that $\gamma(H) = B(H,H_\gamma)$.  Hence, $\alpha(H_\gamma) = \gamma(H_\alpha) = B(H_\alpha,H_\gamma) =: \langle \alpha, \gamma \rangle$.  Any root vectors $e_\gamma, e_{-\gamma}$ satisfy $[e_\gamma,e_{-\gamma}] = n_\gamma H_\gamma$, where $n_\gamma = B(e_\gamma, e_{-\gamma}) \neq 0$.
 
  Now $\fa_1 = \{ X \in \fg_1 \mid [X,\fg_{-1}] \subset \fa_0 \}$, so test $\{ e_\alpha \}_{\alpha \in \Delta(\fg_1)}$. Let $\beta \in \Delta(\fg_1)$.  Using $\fa_0$ from \eqref{E:a0},
 \Ben[(i)]
 \item $\beta = \alpha$: $[e_\alpha,e_{-\alpha}] = n_\alpha H_\alpha$.   Thus, if $H_\alpha \not\in \fa_0$, i.e.\ $\mu(H_\alpha) = \langle \mu,\alpha \rangle \neq 0$, then $\alpha \not\in \Delta(\fa_1)$.
 \item  $\beta \neq \alpha$: If $\alpha - \beta \not\in \Delta$, then $[e_\alpha, e_{-\beta}] = 0$ (so no constraints).  Otherwise, if $\alpha - \beta \in \Delta$, then $[e_\alpha, e_{-\beta}]$ is a nonzero multiple of $e_{\alpha - \beta}$.  Thus, $\alpha - \beta \in \Delta(\fg_{0,+})$ iff $e_{\alpha - \beta} \not\in \fa_0$ iff $\alpha \not\in \Delta(\fa_1)$.
 \Een
 Defining $\cT_1 = \{ \alpha \in \Delta(\fg_1) \mid \langle \mu, \alpha^\vee \rangle \neq 0 \}$ and
 $\cT_2 = \{ \alpha \in \Delta(\fg_1) \mid \exists \beta \in \Delta(\fg_1) \mbox{ with } \alpha - \beta \in \Delta(\fg_{0,+}) \}$, we have $\Delta(\fa_1) = \Delta(\fg_1) \backslash (\cT_1 \cup \cT_2)$.
 
 Fix $i \in I_\fp$.  By Schur's lemma, for each $\fg_{1_i,A}$ in \eqref{E:g00-decomp}, either $\fg_{1_i,A} \subset \fa_1$ or $\fa_1\cap \fg_{1_i,A} = 0$.  By $\fg_{\0,\0}$-irreducibility, it suffices to test the lowest root in $\Delta(\fg_{1_i,A})$.
 \begin{enumerate}
 \item[(i)] $A > \0$ (i.e.\ {\em at least} one entry of $A$ is positive): The lowest $\alpha \in \Delta(\fg_{1_i,A})$ is not the lowest in $\Delta(\fg_{1_i})$, i.e.\ $\alpha \neq \alpha_i$.  Since $\fg_{1_i}$ is a $\fg_\0$-irrep, there are simple roots $\{ \alpha_{j_k} \}_{k=1}^m \subset \Delta(\fg_\0)$ such that $\beta_k := \alpha - \alpha_{j_1} - ... - \alpha_{j_k} \in \Delta$, $\forall k$ and $\beta_m = \alpha_i$.  In particular, $\beta_1 = \alpha - \alpha_{j_1} \in \Delta(\fg_1)$.  Since $\alpha$ is the lowest in $\Delta(\fg_{1_i,A})$, then $j_1 \in J_\mu$ and $\alpha_{j_1} \in \Delta(\fg_{\0,+})$.  Hence, $\alpha \in \cT_2$, so $\Delta(\fg_{1_i,A}) \subset \cT_2$. 
 \item[(ii)] $A=\0$: the lowest root of $\Delta(\fg_{1_i,\0})$ is $\alpha_i$, and clearly $\alpha_i \not\in \cT_2$.  We have $i \not\in I_\mu$ iff $\alpha_i \in \cT_1$.  
 \end{enumerate}
 Thus, claim (b) follows.  (If $I_\mu = \emptyset$, then $\fa_1 = 0$, so $\fa_+ = 0$ by Lemma \ref{L:Tan-h-mod}.)
 
 Let $\cA_k := \{ \alpha \in \Delta(\fg_k) \mid Z_{I_\mu}(\alpha) = k,\, Z_{J_\mu}(\alpha) = 0 \}$.  Note $\Delta(\fa_1) = \cA_1$ follows from (b).  Let $k \geq 2$ and $\gamma \in \Delta(\fa_k) \subset \Delta(\fg_k)$.  Then $\gamma = \alpha + \beta$ for some $\alpha \in \Delta(\fg_{k-1})$ and $\beta \in \Delta(\fg_1)$, since $\fg_1$ is bracket-generating in $\fg_+$.  By definition, $\fa_k = \{ X \in \fg_k \mid [X,\fg_{-1}] \subset \fa_{k-1} \}$, so $\alpha = \gamma - \beta \in \Delta(\fa_{k-1})$.  Since $\fa$ is graded, then $[\fa_k, \fa_{-(k-1)}] \subset \fa_1$, so $\beta = \gamma - \alpha \in \Delta(\fa_1)$.  Thus, $\Delta(\fa_k) \subseteq \Delta(\fa_{k-1}) \,\dot{+}\, \Delta(\fa_1) \subseteq \cA_k$.
 
 Use induction on $k$ to show that $\Delta(\fa_k) = \cA_k$.  Let $\gamma \in \cA_k$ and $\beta \in \Delta(\fg_1)$.  If $\alpha = \gamma - \beta \in \Delta$, then $\alpha \in \Delta(\fg_{k-1})$.  Since $\alpha \in \Delta^+$, then $0 \leq Z_{I_\fp \backslash I_\mu}(\alpha) \leq Z_{I_\fp \backslash I_\mu}(\gamma) = 0$ and $0 \leq Z_{J_\mu}(\alpha) \leq Z_{J_\mu}(\gamma) = 0$, so $Z_{I_\fp \backslash I_\mu}(\beta) = 0$, $Z_{I_\mu}(\beta) = 1$, $Z_{J_\mu}(\beta) = 0$.  Thus, $\alpha \in \cA_{k-1}$, so $\alpha \in \Delta(\fa_{k-1})$ by induction.  Hence, $\gamma \in \Delta(\fa_k)$.  Thus, $\Delta(\fa_k) = \Delta(\fa_{k-1}) \,\dot{+}\, \Delta(\fa_1) = \cA_k$ for $k \geq 2$.  This proves claim (c).
 \end{proof} 
 
 From $\cD(\fg)$, remove the nodes in $I_\fp \backslash I_\mu$ and $J_\mu$.  In this diagram, any node $j \not\in I_\mu$ that is not in the connected component of an $I_\mu$ node corresponds to a simple root $\alpha_j \in \Delta(\fg_{\0,\0})$ with $\Delta(\fa_1) \,\dot{+}\, \{ \alpha_j \} = \emptyset$.  Such nodes play no role in the description of $\Delta(\fa_1)$ (and hence $\Delta(\fa_k)$ for $k > 0$), so can be removed.  Thus, to the data $(\fg,\fp,\mu)$ when $I_\mu \neq \emptyset$, we associate  a {\em reduced geometry} $(\overline\fg,\overline\fp)$ via Recipe \ref{R:red-geom}, and it is clear that $\Delta(\fa_+(\mu))$ and $\Delta(\overline\fg_+)$ naturally correspond, so $\fU_\mu = \dim(\fa(\mu))$ is easily computable.  In Theorem \ref{T:main-thm}, we show that this symmetry dimension is almost always realizable.

 When $\fg$ is simple, write $I_w := I_{-w\cdot\lambda_\fg}$, $\fa(w) := \fa(-w\cdot\lambda_\fg)$, etc.  For $G_2 / P_1$, see Example \ref{EX:G2P1}.
 
 \begin{example}[$E_8 / P_8$, $w = (87)$] \label{ex:E8P8} $\lambda_\fg = \lambda_8$ and $w\cdot \lambda_\fg = \lambda_6 +  \lambda_7 - 4\lambda_8$, so 
 \[
 \bbV_{-w\cdot\lambda_\fg} = \Edd{wwwwwssx}{0,\quad 0,0,0,0,1,1,-4}, \qquad J_w = \{ 6,7 \}, \quad I_w = \emptyset.
 \]
 By Recipe \ref{R:dim}, $\dim(\fg_0) = 1 + \dim(E_7) = 134$, so $\dim(\fg_-) = \half( \dim(E_8) - \dim(\fg_0)) = \half(248 - 134) = 57$.  By Recipe \ref{R:a0}, $\fp^{\opn}_w \cong \fp_{6,7} \subset E_7$ with $\dim(\fa_0) = \dim(\fp^{\opn}_w) =  \half(\dim(E_7) + 2 + \dim(D_5)) = 90$.  Since $I_w = \emptyset$, then $\fa_+ = 0$ (Recipe \ref{R:squares}).  Thus, $\dim(\fa(w)) = 57 + 90 = 147$.
 \end{example}

 \begin{example}[$C_6 / P_{1,2}$, $w = (21)$] $\lambda_\fg = 2\lambda_1$ and $w\cdot \lambda_\fg = -5\lambda_2 + 4 \lambda_3$, so $J_w = \{ 3 \}$, $I_w = \{ 1 \}$:
 \begin{align*}
 \bbV_{-w\cdot \lambda_\fg} &= \NPRC{w} 
 \quad \stackrel{\mbox{Recipe \ref{R:red-geom}}}{\leadsto} \quad 
 \Aone{x}{} \qquad (\bar\fg / \bar\fp = A_1 / P_1).
 \end{align*}
 Note $\Delta(\fa_1) = \{ \alpha_1 \}$. We have $\dim(\fg_-) = \half( \dim(C_6) - 2 - \dim(C_4) ) = \half ( 78 - 2 - 36 )  = 20$.  Since $\fp^{\opn}_w \cong \fp_1  \subset C_4$, then $\dim(\fa_0) = \dim(\fp^{\opn}_w) = 30$.  Thus, $\dim(\fa(w)) = 20 + 30 + 1 = 51$.
 \end{example}

 \begin{example}[$C_6 / P_{1,2,5}$, $w = (21)$] As above, but now $J_w = \{ 3 \}$, $I_w = \{ 1, 5 \}$:
 \begin{align*}
 \bbV_{-w\cdot \lambda_\fg} &= \NPRC{q}
 \quad \stackrel{\mbox{Recipe \ref{R:red-geom}}}{\leadsto} \quad
 \Aone{x}{} \quad \Cthree{wxw}{} \qquad (\bar\fg / \bar\fp = A_1 / P_1 \times C_3 / P_2).
 \end{align*}
 Note $\bar\fg$ is 2-graded, with $\dim(\fa_1) = \dim(\bar\fg_1) = 5$ and $\dim(\fa_2) = \dim(\bar\fg_2) = 3$.  Indeed,
 \begin{align*}
 \Delta(\fa_1) : &\quad  \alpha_1, \quad \alpha_5, \quad \alpha_4 + \alpha_5, \quad \alpha_5 + \alpha_6, \quad \alpha_4 + \alpha_5 + \alpha_6;\\
  \Delta(\fa_2) : &\quad 2\alpha_5 + \alpha_6, \quad \alpha_4 + 2\alpha_5 + \alpha_6, \quad 2\alpha_4 + 2\alpha_5 + \alpha_6.
 \end{align*}
 We calculate $\dim(\fa(w)) = \dim(\fg_-) + \dim(\fa_0) + \dim(\fa_+) = 32 + 11 + 8 = 51$.
 By Theorem \ref{T:corr-Tanaka}, it is no coincidence that this agrees with the result of the previous example.  Also, the property $\fa_2(w) \neq 0$ is rather exceptional: if $\fg$ is simple and $w \in W^\fp_+(2)$, then $\fa_2(w) = 0$ almost always (Theorem \ref{T:NPR}).
 \end{example}

 \subsection{Prolongation-rigidity} 
 \label{S:PR}
 
  For regular, normal parabolic geometries of type $(G,P)$, harmonic curvature $\Kh$ is valued in $H^2_+(\fg_-,\fg)$.  If $H^2_+(\fg_-,\fg) = 0$, then $\Kh \equiv 0$, so the geometry is flat: we say $(\fg,\fp)$ is {\em Yamaguchi-rigid}; otherwise, $(\fg,\fp)$ is {\em Yamaguchi-nonrigid}.   For our study of symmetry gaps, we need only study the latter, and when $G$ is simple these were classified by Yamaguchi \cite{Yam1993,Yam1997} (see Appendix \ref{App:Yam}).

  \begin{defn} \label{D:PR}  If $\fg$ is complex \ss and $-\mu$ is $\fp$-dominant, then $(\fg,\fp,\mu)$ (or $(\fg,\fp,w)$ with $\mu = -w\cdot \lambda_\fg$ if $\fg$ is simple) is {\em PR} {\em (prolongation-rigid)} if $\fa^\phi_+ = 0$ whenever $0 \neq \phi \in \bbV_{\mu}$; otherwise, $(\fg,\fp,\mu)$ is {\em NPR}.\footnote{Here, $\mu$ is {\em not} assumed to have positive homogeneity.}  
If $\fg$ is real or complex, $(\fg,\fp)$ is {\em PR} if $\fa^\phi_+ = 0$ whenever $0 \neq \phi \in H^2_+(\fg_-,\fg)$.  Otherwise, it is {\em NPR}.
 \end{defn}
  
 \begin{remark} \label{RM:PR} The following observations are immediate:
 \Ben
 \item $(\fg,\fp)$ is PR if and only if $(\fg,\fp,\mu)$ is PR for all $\bbV_\mu \subset H^2_+(\fg_-,\fg)$.
 \item Any Yamaguchi-rigid $(\fg,\fp)$ has $H^2_+(\fg_-,\fg) = 0$, so is (vacuously) PR.
 \item If $(\fg,\fp)$ is real and is NPR, then its complexification $(\fg^\bbC,\fp^\bbC)$ will also be NPR.
 \Een
 \end{remark}
 
 \begin{lemma} \label{L:PRw} $(\fg,\fp,\mu)$ is PR iff $I_\mu = \emptyset$ iff $\cD(\fg,\fp,\mu)$ has no squares.  (This is Recipe \ref{R:squares}.)
 \end{lemma}
 
 \begin{proof}
 This follows immediately from part (b) of Theorem \ref{T:DD}.
 \end{proof}
 
 \begin{example} Since $F_4 / P_{1,2}$ is Yamaguchi-rigid (see Theorem \ref{T:Y-rigid}), it is PR.  However, $w=(21) \in W^\fp(2) \backslash W^\fp_+(2)$ gives the (non-regular) irrep $\bbV_{-w \cdot \lambda_1} = \Fdd{qxsw}{0,-4,6,0}$, so $(\fg,\fp,w)$ is NPR. Its reduced geometry is $A_1/P_1$, so $\fa = \fa(w)$ has 1-dimensional $\fa_{+}=\fa_1$.
 \end{example}
 
 In terms of the grading element $Z = Z_{I_\fp}$, we have:
 
 \begin{lemma}
 Suppose $\lambda = \sum_a r_a \lambda_a$ is the highest weight of a simple ideal of $\fg$.  Let $w = (jk) \in W^\fp(2)$ and $\mu = -w\cdot \lambda$.  The regularity condition $Z(\mu) \geq 1$ is equivalent to:
 \begin{align} \label{E:reg}
 Z(\lambda) \leq r_j + (r_k +1) (Z(\alpha_k) - c_{kj}).
 \end{align}
 Independent of regularity, we have $i \in I_\mu$ iff
 \begin{align}
 \l\{\begin{array}{rll} 
 {\rm (a):} & i,j \in I_\fp \mbox{ distinct}, & k \in \cN(j) \backslash I_\fp, \quad 0 = r_i = c_{ji} = c_{ki}; \quad\mbox{ or},\\
 {\rm (b):} & i,j,k \in I_\fp \mbox{ distinct}, & 0 = r_i = c_{ji} = c_{ki}; \quad\mbox{ or},\\
 {\rm (c):} & i,j \in I_\fp \mbox{ distinct}, & k=i,\quad c_{ij} = c_{ji} = -1,\quad r_j = 0.
 \end{array} \r. \label{E:F1}
 \end{align}
 \end{lemma}
 
 \begin{proof}
 Since $\lambda \in \Delta^+$, then $r_a = \langle \lambda, \alpha_a^\vee \rangle \geq 0$, $\forall a$.  By Recipe \ref{R:Kostant}, $\Phi_w = \{ \alpha_j, \sr_j(\alpha_k) \}$, so:
 \begin{align}
 -\mu = w\cdot \lambda &= w(\lambda) + w\cdot 0 = w(\lambda) - \sum_{\alpha \in \Phi_w} \alpha
 = \sr_j (\sr_k(\lambda)) - \alpha_j - \sr_j(\alpha_k) \label{E:w-lambda}\\
 &= \sr_j(\lambda - r_k \alpha_k) - \alpha_j - \sr_j(\alpha_k) 
 = \lambda - (r_j + 1) \alpha_j - (r_k +1)(\alpha_k - c_{kj} \alpha_j). \nonumber
 \end{align}
 Simplify $Z(\mu) \geq 1$ using $j \in I_\fp$ (Recipe \ref{R:Hasse}) to obtain \eqref{E:reg}.
 
 Now $i \in I_\mu$ iff $\langle \mu,\alpha_i^\vee \rangle = 0$, so using \eqref{E:w-lambda},
 \begin{align} \label{E:PR-a}
 0 = r_i - (r_j + 1 ) c_{ji} - (r_k +1) (c_{ki} - c_{kj} c_{ji}).
 \end{align}
 If $i=j$, then $c_{ji} = 2$ and $c_{kj} \leq 0$ (since $k \neq j$ by Recipe \ref{R:Hasse}), so \eqref{E:PR-a} implies $0 = r_j + 2 - (r_k +1) c_{kj} \geq 2$,
 a contradiction.  Thus, we take $i \neq j$ below, i.e.\ $c_{ij}, c_{ji} \leq 0$.
 
 From Recipe \ref{R:Hasse}, $j \in I_\fp$ and $j \neq k \in I_\fp \cup \cN(j)$, where $\cN(j) = \{ i \mid c_{ij} \leq -1 \}$.
 \begin{itemize}
 \item If $k \in \cN(j) \backslash I_\fp$ or $i,j,k\in I_\fp$ are distinct, then $c_{ki}, c_{kj} \leq 0$, so \eqref{E:PR-a} implies
  \begin{align*}
  0 &= \underbrace{r_i}_{\geq 0} - \underbrace{(r_j + 1)}_{\geq 1} \underbrace{c_{ji}}_{\leq 0} - \underbrace{(r_k +1)}_{\geq 1} (\underbrace{c_{ki} - c_{kj} c_{ji}}_{\leq 0}) \qRa r_i = c_{ji} = c_{ki} = 0.
  \end{align*}
 \item If $k=i$, then \eqref{E:PR-a} implies $(r_i+1) (2 - c_{ij} c_{ji}) = r_i - (r_j+1) c_{ji} \geq 0$.  Thus, $0 \leq c_{ij} c_{ji} \leq 2$.  If $c_{ij} c_{ji} = 0$, then $c_{ij} = c_{ji} = 0$, so $r_i = -2$.  If $c_{ij} c_{ji} = 2$, then $r_i = c_{ji} = 0$.  Both give contradictions, so $c_{ij} c_{ji} = 1$ remains, i.e.\ $c_{ij} = c_{ji} = -1$, and \eqref{E:PR-a} reduces to $r_j = 0$.
 \end{itemize}
 Hence, \eqref{E:PR-a} is equivalent to \eqref{E:F1}.
 \end{proof}

 \begin{example} \label{E:reg-quick}
 Condition \eqref{E:reg} can be used to quickly calculate $W^\fp_+(2)$.  Consider $A_\rkg / P_{1,2,s,t}$ for $4 \leq s < t < \rkg$.  Since $\lambda_\fg = \lambda_1 + \lambda_\rkg = \alpha_1 + ... + \alpha_\rkg$, then $Z(\lambda_\fg) = 4$.  Since $r_j, r_k, Z(\alpha_k), -c_{kj} \in \{ 0, 1 \}$, then \eqref{E:reg} forces $r_k = Z(\alpha_k) = -c_{kj} = 1$, so $k=1$, $j=2$, and $W^\fp_+(2) = \{ (21) \}$, while $|W^\fp(2)| = 11$.
 \end{example}

 \begin{prop} \label{P:PR} Let $\lambda$ be the highest weight of a simple ideal in $\fg$.  Let $w = (jk) \in W^\fp(2)$ and $\mu = -w\cdot \lambda$.  Then $(\fg,\fp,\mu)$ is PR if: (a) $|I_\fp| = 1$, or (b) $I_\fp = \{ j,k \}$ and $c_{jk}  =0$.
 \end{prop}
 
 \begin{proof}
 In both cases, it is impossible to satisfy \eqref{E:F1}, so $I_\mu = \emptyset$.  Lemma \ref{L:PRw} gives the result.
 \end{proof}
  
 Suppose $\fg = \fg' \times \fg''$ is a product of simple ideals, $\fp = \fp' \times \fp''$, with $I_{\fp'},I_{\fp''} \neq \emptyset$.  Let $\lambda = \lambda_{\fg'}$.  For any $i'' \in I_{\fp''}$, we have $r_{i''} :=  \langle \lambda, \alpha_{i''}^\vee \rangle = 0$.  Pick any $w = (j'k') \in W^{\fp'}(2)$ and let $\mu = -w\cdot\lambda$.  Then \eqref{E:F1}(a) or (b) is satisfied, so $i'' \in I_\mu$, i.e.\ $(\fg,\fp,\mu)$ is NPR.  Indeed, $\bbV_\mu \subset H^2(\fg_-',\fg') \boxtimes \bbC$ (external tensor product with trivial $\fg_0''$-action on $\bbC$), so $\fa(\mu) \supset \fg''$.  If $(\fg',\fp')$ is Yamaguchi non-rigid, then $w$ can be chosen so that $Z(\mu) \geq 1$.  Thus, for most {\em \ss} cases of interest, $(\fg,\fp)$ is NPR.

 If $\fg$ is {\em simple}, the NPR condition on $(\fg,\fp)$ is more restrictive.  In particular, the highest weight (root) $\lambda_\fg = \sum_{a=1}^\rkg m_a \alpha_a$ satisfies $m_a \geq 1$, so $|I_\fp| \leq Z(\lambda_\fg)$, which constrains the right side of \eqref{E:reg}.

 \begin{cor} \label{C:g-PR} Let $\fg$ be complex simple.  Then $(\fg,\fp)$ is PR if: {\rm (a)} $\fg$ is 1-graded; {\rm (b)} $\fg$ has a contact gradation; or {\rm (c)} $\fg$ is a Lie algebra of exceptional type.
 \end{cor}
 
 \begin{proof}
 All Yamaguchi-rigid $(\fg,\fp)$ are PR.  By Proposition \ref{P:PR}, consider those Yamaguchi-nonrigid $(\fg,\fp)$ with $|I_\fp| > 1$ and $\fp \neq \fp_{j,k}$ for $c_{jk} = 0$.  Among 1-gradings and contact gradings, this leaves $A_2 / P_{1,2}$ (see Theorem \ref{T:Y-pr-thm}), and among Lie algebras of exceptional type, this leaves $G_2 / P_{1,2}$ (see Theorem \ref{T:Y-rigid}).  Both fail \eqref{E:F1}.
 \end{proof}

 Hence, with the exception of pairs of 2nd order ODE $(A_3 / P_{1,2})$, all known symmetry gap results for parabolic geometries (see Table \ref{F:known-submax}) have been for PR geometries.  (Note $C_2 / P_{1,2}$ is also PR.) Table \ref{F:NPR} gives a classification of NPR geometries.  (See Appendix \ref{App:NPR} for details.)

 \begin{table}[h]
 \begin{tabular}{cc} $\begin{array}{|c|c|c|} \hline
 G & P & \mbox{Range} \\ \hline\hline
 A_\rkg & P_{1,2} & \rkg \geq 3 \\
         & P_{1,2,s} & 3 \leq s \leq \rkg \\
         & P_{1,2,s,t} & 3 \leq s < t \leq \rkg \\
         & P_{1,s,\rkg} & 3 \leq s \leq \rkg-2 \\
         & P_{i,i+1} & 2 \leq i \leq \rkg-2 \\
         & P_{1,i} & 4 \leq i \leq \rkg-1 \\
         & P_{2,i} & 4 \leq i \leq \rkg-1 \\ \hline
 B_\rkg & P_{1,2} & \rkg \geq 3 \\
 		& P_{2,3} & \rkg \geq 4 \\ \hline
 \end{array}$ &
 $\begin{array}{|c|c|c|c|} \hline
 G & P & \mbox{Range} \\ \hline\hline
  C_\rkg & P_{1,2} & \rkg \geq 3 \\
 	& P_{1,\rkg} & \rkg \geq 4 \\
 	& P_{2,\rkg} & \rkg \geq 4 \\
         & P_{1,2,s} & 3 \leq s \leq \rkg \\ \hline
 D_\rkg & P_{1,2} & \rkg \geq 4 \\
         & P_{1,\rkg} & \rkg \geq 5 \\
  	& P_{2,3} & \rkg \geq 5 \\
	& P_{1,2,\rkg} & \rkg \geq 4 \\ \hline
 \end{array}$
 \end{tabular}
 \caption{Classification of NPR geometries $G/P$ with $G$ simple}
 \label{F:NPR}
 \end{table}
 
 The {\em height} of (the $Z$-grading on) $\fa$ is the maximal $\tilde\nu \geq 0$ such that $\fa_{\tilde\nu} \neq 0$.

 \begin{thm} \label{T:NPR} Let $\fg$ be a complex simple Lie algebra, $\fp \subset \fg$ a parabolic subalgebra.  If $(\fg,\fp)$ is NPR, then the height of $\fa(w)$, where $w \in W^\fp_+(2)$, always satisfies $0 \leq \tilde\nu \leq 1$, except $\tilde\nu = 2$ for:
 \[
 \begin{array}{c|c|c|c|c|c}
 \mbox{Label} & G/P & \mbox{Range} & w & J_w & I_w \\ \hline
 \mbox{{\rm NPR--A}} & A_\rkg / P_{1,2,s,t} & 4 \leq s < t < \rkg & (21) & \{ 3, \rkg \} & \{ 1, s, t \} \\
 \mbox{{\rm NPR--C}} & C_\rkg / P_{1,2,s} & 4 \leq s < \rkg & (21) & \{ 3 \} & \{ 1 , s \}
 \end{array}
 \]
 \end{thm}
 
 \begin{proof}
 Using Table \ref{F:NPR}, examine all NPR $(\fg,\fp,w)$, $w \in W^\fp_+(2)$, in Tables \ref{F:AB} and \ref{F:CDG}.  The reduced geometry $(\overline\fg,\overline\fp)$ is 1-graded in all cases, except for NPR-A and NPR-C where it is 2-graded.
 \end{proof}
 
 If we remove the regularity assumption, higher prolongations can exist:
 
 \begin{example} For $\rkg \geq 5$, consider $A_\rkg / B = A_\rkg / P_{1,2, \ldots, \rkg}$ and $w = (21) \in W^\fp(2) \backslash W^\fp_+(2)$.  Then:
 \[
 \bbV_{-w\cdot\lambda_\fg} = 
 \begin{tiny}
 \begin{tikzpicture}[scale=\myscale,baseline=-3pt]
 \DDnode{q}{0,0}{0};
 \bond{0,0}; 
 \DDnode{x}{1,0}{-4};
 \bond{1,0};
 \DDnode{x}{2,0}{3};
 \bond{2,0};
 \DDnode{q}{3,0}{0};
 \tdots{3,0};
 \DDnode{q}{4,0}{0};
 \bond{4,0};
 \DDnode{x}{5,0}{1};
 \useasboundingbox (-.4,-.2) rectangle (5.4,0.55);
 \end{tikzpicture}
 \end{tiny}, \qquad Z(-w\cdot\lambda_\fg) = 5 - \rkg.
 \]
 Thus, $J_w = \emptyset$, $I_w = \{ 1, 4, 5, \ldots, \rkg-1 \}$, so $\bar\fg / \bar\fp \cong A_1 / P_1 \times A_{\rkg-4} / P_{1,...,\rkg - 4}$.  Note $\nu = \rkg$, while $\tilde\nu = \rkg - 4$.
 \end{example}


 \subsection{Correspondence and twistor spaces}
 \label{S:corr-twistor}

  We review correspondence and twistor space constructions from \cite{Cap2005,CS2009}.
 Let $G$ be a \ss Lie group $G$, and $Q \subset P \subset G$ parabolic subgroups, so there is a projection $G / Q \to G / P$.  Write the decompositions of $\fg$ corresponding to $\fq,\fp$ as
 \begin{align} \label{E:pq-decomp}
 \fg = \fq_- \op \fq_0 \op \fq_+ = \fp_- \op \fp_0 \op \fp_+,
 \end{align}
 where (see Lemma 4.4.1 in \cite{CS2009}), 
 \begin{align} \label{E:pq-grading}
 I_\fp \subset I_\fq, \qquad \fp_\pm \subset \fq_\pm, \qquad \fq_0 \subset \fp_0, \qquad \fz(\fp_0) \subset \fz(\fq_0).
 \end{align}
 Given a $G/P$ geometry $(\cG \to N, \omega)$, the {\em correspondence space} for $Q \subset P$ is $\cC N = \cG / Q$.  This carries a canonical $G/Q$ geometry $(\cG \to \cC N, \omega)$, and the fibers of $\cC N \to N$ are diffeomorphic to $P/Q$.  The vertical subbundle $V \cC N \subset T\cC N$ corresponds to the $Q$-submodule $\fp/\fq \subset \fg / \fq$, and the Cartan curvature $\kappa^{\cC N}$ of $(\cG \to \cC N,\omega)$ satisfies $i_\xi \kappa^{\cC N} = 0$ for any $\xi \in \Gamma(V\cC N)$.
 
 Conversely, given a $G/Q$ geometry $(\cG \to M, \omega)$, we may ask if there exists a $G/P$ geometry $(\cG \to N,\omega)$ such that $M$ is locally isomorphic to $\cC N$.  If so, we call $N$ the (local) {\em twistor space} for $M$.  From above, a necessary condition is that the subbundle $VM \subset TM$ induced from $\fp/\fq \subset \fg / \fq$ must satisfy $i_\xi \kappa = 0$ for all $\xi \in \Gamma(VM)$.  Locally, this is sufficient, and moreover $(\cG \to N,\omega)$ is unique when $P/Q$ is assumed to be connected \cite{Cap2005} (or Corollary 4.4.1 in \cite{CS2009}).  More convenient still is that it suffices to only check that $\im(\Kh)$ vanishes on $\fp \cap \fq_-$, cf. Theorem 3.3 in \cite{Cap2005}).
 
 Two additional nice properties of correspondence spaces include:
 \begin{itemize}
 \item $\finf(\cG \to N, \omega) = \finf(\cG \to \cC N, \omega)$, provided $P/Q$ is connected.
 \item $(\cG \to N, \omega)$ is normal iff $(\cG \to \cC N, \omega)$ is normal.
 \end{itemize}
 In contrast to normality, regularity is not preserved in passing to a correspondence space.
  
   \begin{prop} \label{P:twistor} Let $G$ be a complex semisimple Lie group, and $Q \subset G$ a parabolic subgroup.  Consider a $G/Q$ geometry $(\cG \to M, \omega)$ with $\im(\Kh) \subset \bbV_{\mu}$, where $\mu = -w\cdot \lambda$, with $\lambda$ the highest weight of a simple ideal in $\fg$, and $w = (jk) \in W^\fq(2)$.  Let {\rm (a)} $\fp := \fp_j$ if $c_{kj} < 0$; or, {\rm (b)} $\fp := \fp_{j,k}$ if $c_{kj} = 0$.  Then $\fq \subset \fp \subset \fg$ and $\im(\Kh)$ consists of cochains which vanish using  $\fp \cap \fq_- \subset \fp_0$ insertions.
 \end{prop}
 
 \begin{proof}  By Recipe \ref{R:Hasse}, $j \in I_\fq$ and $j \neq k \in I_\fq \cup \cN(j)$.  From Recipe \ref{R:Kostant}, the lowest $\fq_0$-weight vector in the $\fq_0$-module $\bbV_{\mu}$ is $\phi_0 = e_{\alpha_j} \wedge e_{\sr_j(\alpha_k)} \ot e_{w(-\lambda)} \in \bigwedge^2 \fq_-^* \ot \fg$.
 
 If $c_{kj} < 0$, take $\fp := \fp_j$.  Then $\sr_j(\alpha_k) = \alpha_k - c_{kj} \alpha_j$ has positive $Z_{I_\fp}$-grading, so $\phi_0 \in \fp_1 \wedge \fp_+ \ot \fg$.  Now recall that the $\fq_0$-module $\bbV_{\mu}$ is generated by applying $\fq_0$-raising operators to $\phi_0$.  Since $\fq_0 \subset \fp_0$, and $\fp_1 \wedge \fp_+ \ot \fg$ is $\fp_0$-invariant (hence $\fq_0$-invariant), then $\bbV_{\mu} \subset \fp_1 \wedge \fp_+ \ot \fg$.  Thus, any map in $\im(\Kh) \subset \bbV_{\mu}$ vanishes on $\fp_0$ and hence on $\fp \cap \fq_- \subset \fp_0$.
  
  If $c_{kj} = 0$, take $\fp := \fp_{j,k}$.  Then $\sr_j(\alpha_k) = \alpha_k$ has positive $Z_{I_\fp}$-grading, and proceed as before.
 \end{proof}

\begin{cor} \label{C:P-exists}
 Under the hypothesis of Proposition \ref{P:twistor}, if there exists a parabolic subgroup $P$ with $Q \subset P \subset G$, and Lie algebra $\fp$ given as in (a) or (b), then $(\cG \to M, \omega)$ admits a twistor space $N$ with $(\cG \to N, \omega)$ of type $(G,P)$.
 \end{cor}

 \begin{remark} \label{RM:Q-connected}
 If $Q$ is connected, then by the subalgebra--subgroup correspondence, there exists a unique connected (parabolic) subgroup $P$ such that $Q \subset P \subset G$ with $\Lie(P) = \fp$.  This guarantees the existence of a twistor space, and will play a substantial simplifying step in the calculation of submaximal symmetry dimensions -- see Remark \ref{E:coverings}.
 \end{remark}
 
 \begin{thm} \label{T:corr-Tanaka} Let $\fg$ be complex \ss.  Let $\fq \subset \fp \subset \fg$ be parabolic subalgebras, and let $\lambda$ be the highest weight of a simple ideal in $\fg$.  Let $w \in W^\fp(2) \subset W^\fq(2)$ and $\mu = -w\cdot \lambda$.  Choosing root vectors in $\fg$, let $\phi_0$ be the Kostant lowest weight vector as in \eqref{E:phi-0}.  Let $\fa = \prn^\fg(\fp_-,\fann_{\fp_0}(\phi_0))$ and $\fb = \prn^\fg(\fq_-, \fann_{\fq_0}(\phi_0))$, which are respectively $Z_{I_\fp}$-graded and $Z_{I_\fq}$-graded.  Then as $Z_{I_\fp}$-graded Lie algebras, $\fa = \fb \subset \fg$.
 \end{thm}
 
 \begin{proof}
 Write $w = (jk) \in W^\fp(2) \subset W^\fq(2)$, so by Recipe \ref{R:Hasse}, $k \neq j \in I_\fp$, and if $c_{kj} = 0$, then $k \in I_\fp$.  We always have $\fq \subset \fp \subset \fp_j \subset \fg$, but if $c_{kj} = 0$, we also have $ \fq \subset \fp \subset \fp_{j,k} \subset \fg$.  Hence, it suffices to consider either $\fp = \fp_j$ if $c_{kj} < 0$, or $\fp = \fp_{j,k}$ if $c_{kj} = 0$.  We assume this below.
 
 By Proposition \ref{P:PR}, $(\fg,\fp,\mu)$ is PR, so $\fa = \fp_- \op \fa_0$.  Thus, $I^\fp_\mu = \emptyset$.  We know $\fp_- \subset \fq_-$, and by Recipe \ref{R:a0}, $\ker(\mu) \subset \fa_0 \cap \fb_0$.  Let us show $\Delta(\fa_0) \subset \Delta(\fb)$.
 Given $\alpha \in \Delta(\fa_0)$, we have $Z_{I_\fp}(\alpha) = 0$ and $Z_{J_\mu^\fp}(\alpha) \leq 0$.  Since $\fq_- \op \fq_{0,\leq 0} \subset \fb$, we may assume $\alpha \in \Delta(\fq_{0,+} \op \fq_+) \subset \Delta^+$, hence $Z_{J_\mu^\fp}(\alpha) = 0$.  (The nonzero coefficients of a root must all have the same sign.)  Since $J^\fp_\mu = J^\fq_\mu \,\dot\cup\, (I_\fq \backslash (I_\fp \cup I^\fq_\mu))$, then:
 \Ben[(i)]
 \item $Z_{J_\mu^\fq}(\alpha) = 0$.  Thus, $\alpha \not\in \Delta(\fq_{0,+})$, so $\alpha \in \Delta(\fq_+)$, i.e.\ $r = Z_{I_\fq}(\alpha) > 0$;
 \item $Z_{I_\fq \backslash (I_\fp \cup I^\fq_\mu)}(\alpha) = 0$.  Since $Z_{I_\fp}(\alpha) = 0$, then $Z_{I_\fq \backslash I_\mu^\fq}(\alpha) = 0$.  
 \Een
 Therefore, $r = Z_{I_\mu^\fq}(\alpha) > 0$.  By Theorem \ref{T:DD}, $\alpha \in \Delta(\fb_r)$.  Thus, $\fa_0 \subset \fb$.  Hence, $\fa \subset \fb$.

Conversely, let $\alpha \in \Delta(\fb)$.  We show that $\alpha \in \Delta(\fa)$.  We have three cases:
 \Ben[(i)]
 \item $Z_{I_\fp}(\alpha) < 0$:  Then $\alpha \in \Delta(\fp_-) \subset \Delta(\fa)$.  
 \item $Z_{I_\fp}(\alpha) = 0$:  Assume $Z_{J_\mu^\fp}(\alpha) > 0$, so $Z_{J^\fq_\mu}(\alpha) \geq 0$ since $J^\fq_\mu \subset J^\fp_\mu$.  Hence, $\alpha \in \Delta(\fb_{\geq 0})$, so by Theorem \ref{T:DD}, $Z_{J_\mu^\fq}(\alpha) \leq 0$.  Thus, $Z_{J_\mu^\fq}(\alpha) = 0$ and $Z_{J_\mu^\fp \backslash J_\mu^\fq}(\alpha) > 0$.  Since $J^\fp_\mu \backslash J^\fq_\mu \subset I_\fq \backslash I^\fq_\mu$, then $Z_{I_\fq \backslash I^\fq_\mu}(\alpha) > 0$.  But by Theorem \ref{T:DD}, $Z_{I_\fq \backslash I^\fq_\mu}(\alpha) = 0$, a contradiction.  Thus, $Z_{J_\mu^\fp}(\alpha) \leq 0$, so $\alpha \in \Delta(\fa_0)$.
 \item $Z_{I_\fp}(\alpha) > 0$:  Since $I_\fp \subset I_\fq \backslash I^\fq_\mu$, then $Z_{I_\fq \backslash I^\fq_\mu}(\alpha) > 0$ and $\alpha \in \Delta(\fb_+)$.  But the latter implies $Z_{I_\fq \backslash I^\fq_\mu}(\alpha) = 0$ by Theorem \ref{T:DD}, a contradiction.
 \Een
 Thus,  $\fb = \fp_- \op \fa_0 = \fa$.
 \end{proof}
  
 \newcommand\ph{\hphantom{-1}}
 \begin{example} For each of $A_6 / P_2$, $A_6 / P_{1,2}$, $A_6 / P_{1,2,4}$, $A_6 / P_{1,2,4,5}$, we have $w = (21) \in W^\fp_+(2)$, with respective gradings on $\fa(w) \subset A_6 = \fsl_7(\bbC)$ given below. (Diagonal entries are suppressed.)
 \begin{tiny}
 \begin{center}
 \begin{tabular}{cccc}
 \begin{tabular}{|@{\,\,}c@{\,\,}|@{\,\,}c@{\,\,}|@{\,\,}c@{\,\,}|@{\,\,}c@{\,\,}|@{\,\,}c@{\,\,}|@{\,\,}c@{\,\,}|@{\,\,}c@{\,\,}|}
\hline
 $*$ & 0 & \ph & \ph & \ph & \ph & \ph \\ \hline
 0 & $*$ & \ph & \ph & \ph & \ph & \ph \\ \hline
 -1 & -1 & $*$ & \ph & \ph & \ph & \ph \\ \hline
 -1 & -1 & 0 & $*$ & 0 & 0 & \ph \\ \hline
 -1 & -1 & 0 & 0 & $*$ & 0 & \ph \\ \hline
 -1 & -1 & 0 & 0 & 0 & $*$ & \ph \\ \hline
 -1 & -1 & 0 & 0 & 0 & 0 & $*$ \\ \hline
 \end{tabular} &
 \begin{tabular}{|@{\,\,}c@{\,\,}|@{\,\,}c@{\,\,}|@{\,\,}c@{\,\,}|@{\,\,}c@{\,\,}|@{\,\,}c@{\,\,}|@{\,\,}c@{\,\,}|@{\,\,}c@{\,\,}|} \hline
 $*$ & 1 & \ph & \ph & \ph & \ph & \ph \\ \hline
 -1 & $*$ & \ph & \ph & \ph & \ph & \ph \\ \hline
 -2 & -1 & $*$ & \ph & \ph & \ph & \ph \\ \hline
 -2 & -1 & 0 & $*$ & 0 & 0 & \ph \\ \hline
 -2 & -1 & 0 & 0 & $*$ & 0 & \ph \\ \hline
 -2 & -1 & 0 & 0 & 0 & $*$ & \ph \\ \hline
 -2 & -1 & 0 & 0 & 0 & 0 & $*$ \\ \hline
 \end{tabular} &
 \begin{tabular}{|@{\,\,}c@{\,\,}|@{\,\,}c@{\,\,}|@{\,\,}c@{\,\,}|@{\,\,}c@{\,\,}|@{\,\,}c@{\,\,}|@{\,\,}c@{\,\,}|@{\,\,}c@{\,\,}|} \hline
 $*$ & 1 & \ph & \ph & \ph & \ph & \ph \\ \hline
-1 & $*$ & \ph & \ph & \ph & \ph & \ph \\ \hline
-2 & -1 & $*$ & \ph & \ph & \ph & \ph \\ \hline
-2 & -1 & 0 & $*$ & 1 & 1 & \ph \\ \hline
-3 & -2 & -1 & -1 & $*$ & 0 & \ph \\ \hline
-3 & -2 & -1 & -1 & 0 & $*$ & \ph \\ \hline
-3 & -2 & -1 & -1 & 0 & 0 & $*$ \\ \hline
 \end{tabular} &
 \begin{tabular}{|@{\,\,}c@{\,\,}|@{\,\,}c@{\,\,}|@{\,\,}c@{\,\,}|@{\,\,}c@{\,\,}|@{\,\,}c@{\,\,}|@{\,\,}c@{\,\,}|@{\,\,}c@{\,\,}|} \hline
 $*$ & 1 & \ph & \ph & \ph & \ph & \ph\\ \hline
-1 & $*$ & \ph & \ph & \ph & \ph & \ph \\ \hline
-2 & -1 & $*$ & \ph & \ph & \ph & \ph \\ \hline
-2 & -1 & 0 & $*$ & 1 & 2 & \ph \\ \hline
-3 & -2 & -1 & -1 & $*$ & 1 & \ph \\ \hline
-4 & -3 & -2 & -2 & -1 & $*$ & \ph \\ \hline
-4 & -3 & -2 & -2 & -1 & 0 & $*$ \\ \hline
 \end{tabular} \\
 $A_6 / P_2$ & $A_6 / P_{1,2}$ & $A_6 / P_{1,2,4}$ & $A_6 / P_{1,2,4,5}$
  \end{tabular}
 \end{center}
 \end{tiny}
 \end{example}
 

  
 \section{Submaximal symmetry dimensions}
 \label{S:submax}

 Consider the {\em complex} or {\em split-real} case.  Define $\fS_\mu$ analogous to $\fS$ but with the constraint $\im(\Kh) \subset \bbV_\mu \subset H^2_+(\fg_-,\fg)$.  As a consequence of Theorem \ref{T:upper}, using $\cO = \bbV_\mu$ in Remark \ref{RM:constrained}, we obtain $\fS_\mu \leq \fU_\mu = \dim(\fa(\mu))$.  We now show that this upper bound is almost always sharp.

 
  \subsection{Establishing submaximal symmetry dimensions}
  \label{S:realize}

 Kostant's theorem (Recipe \ref{R:Kostant}) gives an explicit description of the lowest $\fg_0$-weight vector $\phi_0 \in \bbV_\mu$.  Using the $\fg_0$-module isomorphism $H^2(\fg_-,\fg) \cong \ker(\Box) \subset \bigwedge^2 \fg_-^* \ot \fg$, we have $\phi_0$ naturally realized as a 2-cochain.  Suppose that $\im(\phi_0) \subset \fa := \fa^{\phi_0}$.  In view of \eqref{E:def-bracket}, our strategy for producing a model is to {\em deform the Lie bracket on $\fa$ by $\phi_0$}.  Namely, define $\ff$ to be the vector space $\fa$ with the deformed bracket
 \begin{align} \label{E:f-bracket}
 [\cdot,\cdot]_\ff := [\cdot,\cdot] - \phi_0(\cdot,\cdot).
 \end{align}
 Lemma \ref{L:def-a} guarantees when $[\cdot,\cdot]_\ff$ is a Lie bracket.  Since $\phi_0$ vanishes on $\fp = \fg_{\geq 0}$, then $\fk := \fa_{\geq 0}$ is a subalgebra of both $\fa$ and $\ff$.  We construct a {\em local} homogeneous space $M = F/K$ (i.e.\ $F$ may only be a local Lie group) and $P$-bundle $\cG = F \times_K P$ with an $F$-invariant normal Cartan connection $\omega$ whose curvature is determined by $\phi_0$.  Thus, $\fS_\mu \geq \dim(\finf(\cG,\omega)) \geq \dim(\ff) = \dim(\fa)$.  (In fact, this construction does not depend on regularity.)
  
 \begin{lemma} \label{L:def-a} Let $\lambda$ be the highest weight of a simple ideal in $\fg$ and $w \in W^\fp(2)$.  Suppose that $w(-\lambda) \in \Delta^-$.  Let $\phi_0$ be the lowest weight vector of $\bbV_{-w\cdot\lambda}$ and $\fa := \fa^{\phi_0}$.  Define $\ff$ to be the vector space $\fa$ with deformed Lie bracket \eqref{E:f-bracket}.  Then $\ff$ is a Lie algebra.
 \end{lemma}
 
 \begin{proof} 
 Write $w = (jk)$.  From \eqref{E:phi-0}, $\phi_0 = e_{\alpha_j} \wedge e_{\sr_j(\alpha_k)} \ot e_{w(-\lambda)}$ with $w(-\lambda) \in \Delta^- = \Delta(\fg_-) \cup \Delta^-(\fg_0)$.  Root vectors from $\Delta^-(\fg_0)$ annihilate the lowest $\fg_0$-weight vector $\phi_0$, so $e_{w(-\lambda)} \in \fg_- \op\fann(\phi_0) \subset \fa$.  Thus, $\im(\phi_0) \subset \fa$, so $\phi_0 \in \bigwedge^2 \fg_-^* \ot \fa \inj \bigwedge^2 \fa^* \ot \fa$.  Note $e_{\alpha_j}, e_{\sr_j(\alpha_k)} \in \fg_+ \cong \fg_-^*$ vanish on $\fa_{\geq 0}$.
 
 By Theorem \ref{T:corr-Tanaka}, the Lie algebra structure of $\fa$ is the same if we take $\fp = \fp_j$ if $c_{kj} < 0$ or $\fp = \fp_{j,k}$ if $c_{kj} = 0$.  Assuming this, then by Proposition \ref{P:PR}, $(\fg,\fp,-w\cdot\lambda)$ is PR, i.e.\ $\fa = \fg_- \op \fa_0$.
 
 It suffices to verify the Jacobi identity $\Jac_\ff(x,y,z) = 0$ for $[\cdot,\cdot]_\ff$.  For $x,y,z \in \ff$,
 \begin{align*}
 [[x,y]_\ff,z]_\ff &= [[x,y]-\phi_0(x,y),z]_\ff 
= [[x,y],z] - \phi_0([x,y],z) - [\phi_0(x,y),z] + \phi_0(\phi_0(x,y),z).
 \end{align*}
 Since $\phi_0$ vanishes on $\fa_0$, clearly $\Jac_\ff(x,y,z) = 0$ when at least two of $x,y,z \in \fa_0$.  Suppose $x,y \in \fg_-$ and $z \in \fa_0$.  Since $\Jac_\fa(x,y,z) = 0$ and $z \in \fa_0 = \fann(\phi_0)$, then
   \[
 \Jac_\ff(x,y,z)=[z,\phi_0(x,y)] - \phi_0([z,x],y) - \phi_0([y,z],x) = (z\cdot \phi_0)(x,y) = 0.
 \]
 
 Finally, suppose $x,y,z \in \fg_-$. Since $\phi_0 \in \ker(\Box)$, it is in particular $\p$-closed.  This means
 \begin{align}
 0 = (\p\phi_0)(x,y,z) &= [x,\phi_0(y,z)] - [y,\phi_0(x,z)] + [z,\phi_0(x,y)] \label{E:d-phi}\\
 &\qquad - \phi_0([x,y],z) + \phi_0([x,z],y) - \phi_0([y,z],x), \qquad \forall x,y,z \in \fg_-. \nonumber
 \end{align}
 This equation uses the bracket on $\fg$, but restricts to $\fa$ since $\phi_0$ has image in $\fa$, and the bracket on $\fa$ is inherited from that of $\fg$ by Lemma \ref{L:T-subalg}.  Using \eqref{E:d-phi} and $\Jac_\fa(x,y,z) = 0$, we obtain
 \begin{align*}
 \Jac_\ff(x,y,z) 
 = \phi_0(\phi_0(x,y),z) + \phi_0(\phi_0(y,z),x) + \phi_0(\phi_0(z,x),y).
 \end{align*}
 Since $e_{w(-\lambda)}$ spans $\im(\phi_0)$, it suffices to show $w(-\lambda) \not\in \{ -\alpha_j, -\sr_j(\alpha_k) \}$.  If $w(-\lambda) = -\alpha_j$, then $\lambda = -\sr_k(\alpha_j) \in \Delta^-$.  If $w(-\lambda) = -\sr_j(\alpha_k)$, then $\lambda = -\alpha_k \in \Delta^-$.  These contradict $\lambda \in \Delta^+$.
 \end{proof}
 
 Let us determine all exceptions:
 
 \begin{lemma} \label{L:exceptions} Let $G$ be complex simple.  Let $w \in W^\fp$ satisfy $w(-\lambda_\fg) \in \Delta^+$.
 \Ben[(a)]
 \item If $|w| = 2$, then $G / P$ is one of: \quad $A_2 / P_1, \quad A_2 / P_{1,2},\quad B_2 / P_1,\quad B_2 / P_{1,2}$.  
 \item If $|w| = 1$, then $G / P \cong A_1 / P_1$.
 \Een
 For $B_2 / P_{1,2}$, $w(-\lambda_\fg) \in \Delta^+$ if $w = (12)$, but not for $w = (21)$.
 \end{lemma}
 
 \begin{proof}  Part (b) is obvious, so we prove (a).
  Since $W^\fp(2) \neq \emptyset$, then $\rnk(\fg) \geq 2$.  Write $w = (jk)$.   If $w(-\lambda_\fg) \in \Delta^+$, then $w(-\lambda_\fg) \in w\Delta^- \cap \Delta^+ = \Phi_w = \{ \alpha_j, \sr_j(\alpha_k) \}$.   There are two cases:
 \begin{enumerate}[(i)]
 \item $w(-\lambda_\fg) = \alpha_j$: Then $\lambda_\fg = \sr_k(\alpha_j) = \alpha_j - c_{jk} \alpha_k$, which is the highest root, so $\rnk(\fg) = 2$. 
 \item $w(-\lambda_\fg) = \sr_j(\alpha_k)$: Then $\lambda_\fg = \alpha_k$.  Hence, $\rnk(\fg) = 1$, a contradiction.
 \end{enumerate}
 For $B_2 / P_2$, $\exists! w = (21) \in W^\fp(2)$, so $w(-\lambda_\fg) = w(-2\lambda_2) = -2\lambda_1 + 2\lambda_2 = -\alpha_1 \not\in \Delta^+$.  We verify the final claim for $B_2 / P_{1,2}$ directly.
 \end{proof}
 
 \begin{remark}Part (b) is relevant only when $G$ is {\em semisimple} and $G / P$ contains $A_1 / P_1 \times G' / P'$.  Let
 $\lambda = 2\lambda_1$ be the highest weight of $A_1$, and $w = (1,k')$, where the $1$ comes from the first factor, $k' \in I_{\fp'}$ is arbitrary.  Note $w(-\lambda) = \alpha_1 \in \Delta^+$, so Lemma \ref{L:def-a} cannot be applied.
 \end{remark}
 
 We first establish a result outside the parabolic context:
 
 \begin{lemma} \label{L:loc-model} Let $G$ be a real or complex Lie group, and $P \subset G$ a closed subgroup. Let $\ff$ be a Lie algebra, $\fk \subset \ff$ a subalgebra, and $\iota : \fk \to \fp$ a Lie algebra injection.  Suppose there exists a linear map $\vartheta : \ff \to \fg$ such that: (i) $\vartheta|_\fk = \iota$; (ii) $\vartheta([z,x]_\ff) = [\iota(z),\vartheta(x)]$, $\forall z \in \fk$ and $\forall x \in \ff$;  (iii) $\vartheta$ induces an isomorphism from $\ff / \fk$ to $\fg / \fp$.  Then there exists a geometry $(\cG \to M, \omega)$ of type $(G,P)$ with $\finf(\cG,\omega)$ containing a subalgebra isomorphic to $\ff$.
 \end{lemma}

  \begin{proof} Let $K$ and $F$ be the unique connected, simply-connected Lie groups with Lie algebras $\fk$ and $\ff$, respectively.  There are unique Lie group homomorphisms $\Phi : K \to F$ and $\Psi : K \to P$ whose differentials $\Phi' : \fk \to \ff$ and $\Psi' : \fk \to \fp$ are the given inclusions $\fk \inj \ff$ and $\iota : \fk \inj \fp$.\\

{\em Case 1: $\Phi$ is injective and $\Phi(K) \subset F$ is a Lie subgroup}.  Define $\cG=F\times_K P = (F \times P) / \!\sim_K$, where
 $(u,p) \sim_K (u\cdot\Phi(k),\Psi(k^{-1})\cdot p)$ for any $k \in K$.  This has a natural projection $\cG\to M := F/\Phi(K)$ (note $\Phi(K) \subset F$ is closed).  Since $\Phi$ is injective, identify $K$ with its image $\Phi(K)$.  Any $F$-invariant Cartan connection $\omega$ on $\cG$ corresponds \cite{Ham2007}, \cite{CS2009} to a linear map $\vartheta : \ff \to \fg$ satisfying (i)--(iii).  The curvature function of $\omega = \omega_\vartheta$ corresponds to
 \begin{align} \label{E:kappa-alpha}
 \kappa_\vartheta(x,y) = [\vartheta(x),\vartheta(y)] - \vartheta([x,y]_\ff) \qquad \forall x,y \in \ff.
 \end{align}
 By construction, $\ff\subset\finf(\cG,\omega)$.\\

{\em Case 2: $\Phi$ is not injective or $\Phi(K) \subset F$ is not a Lie subgroup}.  Since $\Phi'$ is injective, then $\Phi$ has a discrete kernel.  In place of $F$ in the definition of the Cartan bundle $\cG$ in Case 1, we will define a {\em semi-local} Lie group $U$, an inclusion $K \hookrightarrow U$ as a closed subgroup which acts on the right on $U$, and such that $\Lie(U) = \ff$.  (``Semi-locality'' will be clarified below.)  Then $\ff$-invariant Cartan connections on $\cG = U \times_K P \to M :=U/K$ correspond to linear maps $\vartheta : \ff \to \fg$ as in Case 1.

Let $\fv$ be a vector space complementary to $\fk$ in $\ff$.  Choose two norms 
$\|\cdot\|_\fv$, $\|\cdot\|_\fk$ on these vector spaces and
define the norm on the Lie algebra $\ff=\fv\oplus\fk$ by the formula
$\|v+k\|_\ff=\max\{\|v\|_\fv,\|k\|_\fk\}$ for $v\in\fv$,
$k\in\fk$. Then the balls in these spaces satisfy: $B^\ff_\epsilon=B^\fv_\epsilon\times B^\fk_\epsilon$.

We choose $\epsilon>0$ so small that the maps $\exp:B^\fk_{5\epsilon}\to K$,
$\exp:B^\ff_{5\epsilon}\to F$ are injective,
$\exp(B^\fk_{5\epsilon})\cap\Phi^{-1}(1_F)=1_K$,
and we will further specify $\epsilon$ below. One of the requirements is that
$\exp\cdot\exp:B^\fv_\epsilon\times B^\fk_\epsilon\to F$ is an embedding.
This holds for small $\epsilon$ because
$\log\bigl(\exp(\epsilon v)\exp(\epsilon k)\exp(-\epsilon v-\epsilon k)\bigr)=o(\epsilon)$
by the Baker--Campbell--Hausdorff formula ($\|v\|_\fv\le1$, $\|k\|_\fk\le1$),
implying that for $\epsilon\ll1$, $U_0:=\exp(B^\fv_\epsilon)\cdot\exp(B^\fk_\epsilon)
\subset\exp(B^\ff_{2\epsilon})$ and
 $U_0^{-1}=\{g^{-1} \mid g\in U_0\}=
\exp(B^\fk_\epsilon)\cdot\exp(B^\fv_\epsilon)
\subset\exp(B^\ff_{2\epsilon})$.

Let $K_0=\exp(B^\fk_\epsilon)\subset K$ and $V_0=\Phi(K_0)\subset U_0$.
Denote $U_0^2:=U_0^{-1}\cdot U_0\subset F$ and
$V_0^2=(U_0^2\cap\im(\Phi))_0$, where the last zero means the connected component of $1_F$.
We have $V_0^2=\Phi(K_0^2)$ for a unique connected neighborhood $K_0^2$ of $1_K$.
Denote also $K_0^4:=K_0^2\cdot K_0^2$. Using the Baker--Campbell--Hausdorff formula as above we get
$U_0^2\subset\exp(B^\ff_{3\epsilon})$, $K_0^2\subset\exp(B^\fk_{3\epsilon})$
and $K_0^4\subset\exp(B^\fk_{5\epsilon})$
for small enough $\epsilon$. This implies that
$\Phi:K_0^4\to V_0^4=V_0^2\cdot V_0^2\subset F$ is an embedding and that
$V_0^4\cap U_0^2=V_0^2$. Indeed, from $V_0^4\subset\exp(B^\ff_{5\epsilon})$ we get:
$\log(V_0^4\cap U_0^2)\subset\fk\cap\log(U_0^2)=\log(V_0^2)\subset\ff$.
Consequently $V_0^4\cap U_0^2\subset V_0^2$ and the reverse inclusion is obvious.

Consider the sheaf of neighborhoods $\{U_k=U_0\cdot\Phi(k)\}_{k\in K}$ over $K$.
We introduce an equivalence relation on it as follows. Take $x_1\in U_{k_1}$ and $x_2\in U_{k_2}$,
i.e.\ $x_1=u_1\cdot \Phi(k_1)$, $x_2=u_2\cdot \Phi(k_2)$ for some $u_1,u_2\in U_0$.
We let $x_1\sim x_2$ iff $u_2=u_1\cdot\Phi(k_1\,k_2^{-1})$ and $k_1\,k_2^{-1}\in K_0^2$.

Let us check that it is an equivalence relation. That $\sim$ is reflexive is obvious and the
symmetry follows from $(K_0^2)^{-1}=K_0^2$ $\Leftrightarrow$ $(U_0^2)^{-1}=U_0^2$.
Let now $x_1\sim x_2$ and $x_2\sim x_3$. Then
$u_3=u_2\cdot\Phi(k_2\,k_3^{-1})=u_1\cdot\Phi(k_1\,k_2^{-1})\cdot\Phi(k_2\,k_3^{-1})=u_1\cdot\Phi(k_1\,k_3^{-1})$
and $k_1\,k_3^{-1}=(k_1\,k_2^{-1})(k_2\,k_3^{-1})\in K_0^4$. On the other hand
$\Phi(k_1\,k_3^{-1})=u_1^{-1}u_3\in U_0^2$, and these two imply $k_1\,k_3^{-1}\in K_0^2$.
Thus we get transitivity.

Define $U = \sqcup_{k \in K} U_k/\!\sim$. This is a manifold with the chart centered at $k$
given by the embedding $R_{\Phi(k)}:U_0\to U$. Overlap maps are given by the equivalence relation.
The group $K$ is embedded into $U$ by the formula $K\ni k\mapsto R_{\Phi(k)}(1_F)\in U_k\hookrightarrow U$.
Indeed, if $k\sim k'$ i.e.\ $1_F\cdot\Phi(k)=1_F\cdot\Phi(k')$ and $k'k^{-1}\in K_0^2$, 
then $\Phi(k'k^{-1})=1_F$ yields $k=k'$. This implies that topologically $U\simeq B^\fv_\epsilon\times K$.

In fact, we can make this isomorphism canonical by strengthening the condition on $\epsilon$
(taking it possibly smaller). Namely let us require that
$\exp\cdot\exp:B^\fv_\epsilon\times B^\fk_{5\epsilon}\to F$ is an embedding.
Every element of $U$ can be represented in the form
$U_k\ni u_0\cdot\Phi(k)=\exp(\epsilon v)\cdot\Phi(\exp(\epsilon\kappa)\cdot k)$, where
$\|v\|_\fv\le1$ and $\|\kappa\|_\fk\le1$.
We map this element to $(\epsilon v,\exp(\epsilon\kappa)\cdot k)$.
Independence of the representative follows from the above condition on the map
$\exp\cdot\exp$.

Moreover, $U$ is a semi-local Lie group, which means that 
for any two points $k_1,k_2\in K$ there exists two neighborhoods $U_{k_i}'\subset U_{k_i}$ containing
 $\exp(B^\fk_\epsilon)\cdot\Phi(k_i)$ in $U$
such that the multiplication (induced from $F$)
acts as $U'_{k_1}\times U'_{k_2}\to U_{k_1\cdot k_2}\subset U$ by
$(u_1\cdot\Phi(k_1))\cdot(u_2\cdot\Phi(k_2))=(u_1\cdot\mathop{Ad}_{\Phi(k_1)}u_2)\cdot\Phi(k_1k_2)$.
Notice that $u_1\cdot\mathop{Ad}_{\Phi(k_1)}u_2\in U_0$ if $u_1,u_2\in U_0$ are sufficiently close
to $1_F$.  (This condition specifies $U_{k_i}'$.)
Similarly, we impose locality of the inverse along $K$, as a map 
$U'_k\to U_{k^{-1}}$ for some neighborhood $U_k'\subset U_k$ of $\exp(B^\fk_\epsilon)\cdot\Phi(k)$ given by
$(u\cdot\Phi(k))^{-1}=\mathop{Ad}_{\Phi(k^{-1})}u^{-1}\cdot\Phi(k^{-1})$.

The semi-local Lie group $U$ is naturally homomorphically mapped to $F$ (this is an immersion,
not necessarily an embedding). The chart $U_{1_K}\hookrightarrow U$ is isomorphically mapped to
$U_0\subset F$. Therefore the Lie algebra of $U$ is naturally identified with $\ff$.

In addition, $K\hookrightarrow U$ is a closed subgroup,
which acts from the right on $U$. 
The orbit of the right $K$-action is given by:
$K\ni k\mapsto u_0\cdot\Phi(k)\in U_k\hookrightarrow U$, $u_0\in U_0$.
All orbits are isomorphic to $K$ and the two orbits either do not intersect or coincide.
Moreover (provided we adopt the stronger condition on $\epsilon$ above) we can
represent the orbit uniquely via $u_0$ by claiming that $u_0\in\exp(B^\fv_\epsilon)$.
The orbit space is then $U/K=U_0/K_0=\exp(B^\fv_\epsilon)\simeq B^\fv_\epsilon$,
and this is a neighborhood of the point $K/K\equiv0\in B^\fv_\epsilon$.
Thus we get the local homogeneous space as the desired model.

Now define the Cartan bundle $\cG=U\times_K P = (U \times P) /\sim_{\Psi}$, where
$(q_1,p_1)\sim_{\Psi}(q_2,p_2)$ iff $q_2=q_1\cdot k$ and $p_2=\Psi(k)^{-1} \cdot p_1$ for some $k \in K$.
This $\cG$ is a principal $P$-bundle over $U/K \simeq B^\fv_\epsilon$,
in particular $\cG\simeq B^\fv_\epsilon\times P$ topologically.  We conclude as in Case 1.
 \end{proof}
 
 The following result is independent of the regularity condition on curvature.

 \begin{thm}[Realizability] \label{T:realize} Let $G$ be a complex \ss Lie group, $P \subset G$ a parabolic subgroup, $\lambda$ the highest weight of a simple ideal in $\fg$, $w \in W^\fp(2)$, and $\mu = -w\cdot \lambda$.  Suppose $w(-\lambda) \in \Delta^-$.  Then there exists a non-flat normal geometry $(\cG \to M, \omega)$ of type $(G,P)$ with $\im(\Kh) \subset \bbV_\mu$ and $\dim(\finf(\cG,\omega)) \geq \dim(\fa(\mu)) = \fU_\mu$.
 \end{thm}

 \begin{proof}  Define $\ff$ to be the vector space $\fa := \fa(\mu)$ with deformed bracket \eqref{E:f-bracket}.  Since $w(-\lambda) \in \Delta^-$, then by Lemma \ref{L:def-a}, $\ff$ is a Lie algebra.  Since $\phi_0$ is trivial on $\fk := \fa_{\geq 0}$, then $\fk \subset \ff$ is a subalgebra.  Also, $\fk \subset \fp = \fg_{\geq 0}$ is a subalgebra.  Since $\ff = \fa$ as vector spaces, and $\fa \subset \fg$, take $\vartheta : \ff \to \fg$ to be the induced {\em vector space} inclusion.  Then (i)--(iii) in Lemma \ref{L:loc-model} hold, so an $\ff$-invariant geometry $(\cG \to M, \omega)$ of type $(G,P)$ exists, i.e.\ $\dim(\ff) \leq \dim(\inf(\cG,\omega))$, with $\kappa$ determined by  
 \begin{align*}
 \kappa_\theta(x,y) = [\vartheta(x),\vartheta(y)] - \vartheta([x,y]_\ff) = [x,y] - [x,y]_\ff = \phi_0(x,y), \qquad \forall x,y \in \ff.
 \end{align*}
 Since $\phi_0 \in \ker(\Box)$, then $\p^*\phi_0 = 0$, so $(\cG \to M, \omega)$ is normal.
 \end{proof}
  
For {\em split real forms}, notions of roots and weights from complex representation theory are similarly defined.  All preceding statements and proofs continue to hold.  Using Theorem \ref{T:realize} when $Z(\mu) > 0$, we have $\fS_\mu = \fU_\mu$ almost always, c.f.\ Lemma \ref{L:exceptions}.  Combined with Corollary \ref{C:transitive}, we obtain:
 
 \begin{thm} \label{T:main-thm} Let $G$ be a complex or split-real \ss Lie group, $P \subset G$ a parabolic subgroup, $\lambda$ the highest weight of a simple ideal in $\fg$, $w \in W^\fp(2)$, and $\mu = -w\cdot \lambda$ satisfies $Z(\mu) > 0$.  We always have $\fS_\mu \leq \fU_\mu$.  If $w(-\lambda) \in \Delta^-$ or $G/P$ does not contain the factors $A_1 / P_1,\, A_2 / P_1,\, A_2 / P_{1,2},\, B_2 / P_1,\, B_2 / P_{1,2}$, then $\fS_\mu = \fU_\mu = \dim(\fa(\mu))$ and any submaximally symmetric model is locally homogeneous about a non-flat regular point.
 \end{thm}

 \begin{example}[3rd order ODE up to contact transformations, $B_2 / P_{1,2}$] \label{ex:B2P12} This geometry is PR,  $\lambda_\fg = 2\lambda_2$, and $\dim(\fa(w)) = 5$ for $w \in W^\fp_+(2) = \{ (12), (21) \}$, so $\fS \leq \fU = 5$.  Since $(21)(-\lambda_\fg) \in \Delta^-$, then by Theorem \ref{T:realize}, a model with 5-dimensional symmetry exists, so $\fS = \fU = 5$, and has twistor space type $B_2 / P_2$, i.e.\ a 3-dimensional contact projective structure.  In \S \ref{S:exceptions}, we show that 5 is not realizable in the $(12)$-branch, having twistor type $B_2 / P_1$, i.e.\ a 3-dimensional conformal structure.
 \end{example}

 Consequently, $B_2 / P_{1,2}$ is not exceptional with regard to $\fS = \fU$.  Thus,
 
 
 \begin{thm} \label{T:main-thm2}
 Let $G$ be a complex or split-real \ss Lie group, $P \subset G$ a parabolic subgroup.  If $G/P$ does not contain any of the factors $A_1 / P_1,\, A_2 / P_1,\, A_2 / P_{1,2},\, B_2 / P_1$, then 
 \[
 \fS = \fU = \max\{ \fU_\mu \mid \bbV_\mu \subset H^2_+(\fg_-,\fg) \}, \qquad \fU_\mu = \dim(\fa(\mu)),
 \]
 and any submaximally symmetric model is locally homogeneous about a non-flat regular point.
 \end{thm}
  
 \begin{remark} \label{RM:reduce-to-two} For each $\bbV_\mu \subset H^2_+(\fg_-,\fg)$, the corresponding $w = (jk) \in W^\fp_+(2)$ is generated from the Weyl group of at most two simple ideals in $\fg$.  Any remaining simple factor $(\hat\fg,\hat\fp)$ yields a flat factor, so $\hat\fg \subset \fa(\mu)$.  Thus, the general (complex or split-real) semisimple (non-simple) case reduces to studying two factors $\fg = \fg' \times \fg''$.
 \end{remark}
 
 \begin{remark} \label{E:coverings}
 Aside from the exceptions in Theorem \ref{T:main-thm}, given $(\fg,\fp,\mu)$, $\fS_\mu$ is the same for all $(G,P)$ with $\Lie(G) = \fg$ and $\Lie(P) = \fp$.  To simplify the calculation of $\fS_\mu = \fU_\mu$, we may assume $P$ is {\em connected} and pass to the minimal twistor space.  (See Corollary \ref{C:P-exists} and Remark \ref{RM:Q-connected}.)
 \end{remark}
 
 
 \subsection{Deformations}
 \label{S:deform}
 
A graded subalgebra $\fa \subset \fg$ admits a canonical filtration, i.e.\ $\fa^i = \bop_{j \geq i} \fa_j$.
 
 \begin{defn}
 A graded subalgebra $\fa$ is {\em filtration-rigid} if any filtered Lie algebra $\ff$ with $\gr(\ff) = \fa$ as graded Lie algebras has $\ff \cong \fa$ as filtered Lie algebras.
\end{defn}

 \begin{prop} \label{P:FR} Let $G$ be a real or complex \ss Lie group and $P \subset G$ a parabolic subgroup, and suppose that $\prn(\fg_-,\fg_0) \cong \fg$.  Let $(\cG \stackrel{\pi}{\to} M, \omega)$ be a regular, normal $G/P$ geometry which is not flat.  Given $u \in \cG$ with $\Kh(u) \neq 0$, if $\fg_- \subset \fs(u)$, then $\fs(u)$ is not filtration-rigid.
 \end{prop}

 \begin{proof} Since $\prn(\fg_-,\fg_0) \cong \fg$, the geometry is completely determined by its regular infinitesimal flag structure.  Assume $\ff(u) \cong \fs(u)$, so $\ff(u)$ and hence $\cS(x)$ is graded, where $x = \pi(u)$.  Then $\cS(x)/\cS(x)^0\cong \fg_- =: \fm$ acts transitively on $M$ at $x$.  Thus, $M$ can be locally identified near $x$ with the simply-connected Lie group $M(\fm)$ with Lie algebra $\fm$.  This is endowed with the $M(\fm)$-invariant distribution generated by $\fg_{-1}$, and there is a $M(\fm)$-invariant principal $G_0$-bundle $\cG_0 \to M(\fm)$, where $G_0 \subseteq \Aut_{gr}(\fm)$, which is a reduction of the graded frame bundle over $M(\fm)$.

The triple $(M(\fm),\fg_{-1},\fg_0)$ is the standard model in Tanaka theory, and since $\prn(\fg_-,\fg_0) \cong \fg$ is finite-dimensional, its symmetry algebra $\cS$ is isomorphic to $\fg$ \cite[\S 6]{Tan1970}. 
 But since $\Kh(u) \neq 0$, then $\dim(\cS) \leq \fS < \dim(\fg)$, c.f. Proposition \ref{P:loc-flat} or Theorem \ref{T:upper}, a contradiction.
 \end{proof}

 In the exceptional cases, we will show that $\fa(\mu)$ is filtration-rigid, and hence $\fS_\mu < \fU_\mu$.

 \begin{example} \label{EX:proj-FR}
 Without the $\prn(\fg_-,\fg_0) \cong \fg$ assumption, Proposition \ref{P:FR} generally fails.  Consider a projective structure $[\nabla]$, c.f.\ \S \ref{S:projective}.  The condition $\fg_- \subset \fs(u)$ implies local homogeneity, and since $\fg_-$ is abelian, then near $\pi(u)$, we have local coordinates $(t,x^1,...,x^m)$, and symmetries $\p_t, \p_{x^1},..., \p_{x^m}$ of $[\nabla]$.  The geodesic equations \eqref{E:proj-ODE} admit these symmetries, so all coefficients in these equations are constants.  Hence, all connection coefficients $\Gamma^a_{bc}$ can be made constant by a projective change.  For $m \geq 2$, this does not imply the structure is flat, e.g. taking all $\Gamma^a_{bc} = 0$ except $\Gamma^1_{01} = -\frac{1}{2}$, the Fels invariants \eqref{E:Fels-inv} satisfy $S \equiv 0$ and $T \neq 0$.
 
 However, for $m=1$, the Tresse invariants \eqref{E:Tresse} vanish everywhere for $\ddot{x} = F(\dot{x})$, where $F$ is a cubic polynomial.  Thus, Proposition \ref{P:FR} is true for 2-dimensional projective structures.
 \end{example}
 
 Consider $\fa = \bop_i \fa_i$ with (graded) basis $\{ e_i^\alpha \}$ and structure equations
$[e_i^\a,e_j^\b]=\sum_\gamma c^{\a\b}_\gamma e^\gamma_{i+j}$.  Any filtered Lie algebra $\ff \subset \fg$ with $\gr(\ff) = \fa$ has a (filtered) basis $\{ f_i^\alpha\}$ with 
 \begin{align} \label{E:f-deform}
 [f_i^\alpha,f_j^\beta]_\ff =\sum_\gamma c^{\alpha\beta}_{ij\gamma} f_{i+j}^\gamma+\sum\limits_{k>i+j} \sum_\gamma b^{\alpha\beta k}_{ij\gamma} f_k^\gamma,
 \end{align}
 and we consider $\ff$ up to filtration-preserving automorphisms $f_i^\alpha\mapsto f_i^\alpha+\sum_{j > i} \sum_\beta \chi^{\alpha j}_{i\beta} f^\beta_j$.  We refer to the terms in the double summation in \eqref{E:f-deform} as {\em tails}.  So establishing filtration-rigidity amounts to eliminating the presence of tails (in some basis).



 \begin{prop} \label{P:f-rigid} Let $\ff$ be a filtered Lie algebra, $\gr(\ff) =: \fa = \fa_{-\hat\nu} \op ... \op \fa_0 \op ... \op \fa_{\tilde\nu}$ its associated-graded, and let $\tau_j : \ff^j \to \fa_j$ denote the canonical projections.  Suppose there exists $h \in \fa_0$ such that $\ad_h$ is diagonal in some (graded) basis $\{ e_i^\alpha \}$ of $\fa$, and let $\cE_i = \Spec(\ad_h|_{\fa_i})$.  If $\cE_i \cap \cE_j = \emptyset$ for distinct $i,j$, then for any $\tilde{h} \in \ff^0$ with $\tau_0(\tilde{h}) = h$, there is a (filtered) basis $\{ f_i^\alpha \}$ of $\ff$, with $\tau_i(f_i^\alpha) = e_i^\alpha$, and with respect to which $\ad_{\tilde{h}}$ is diagonal.  Moreover, if:\footnote{Here
$T+T :=\{a+b \mid a,b\in T\}$ and $\bigwedge^2 T :=\{a+b \mid a, b \in T, \, a \neq b\}$.}
  \Ben[(i)]
  \item $\forall i \neq j$: $(\cE_i + \cE_j) \cap \bigcup_{k > i+j} \cE_k = \emptyset$; and
  \item $\forall i$: $\bigwedge^2 \cE_i \cap \bigcup_{k > 2i} \cE_k = \emptyset$,
  \Een
 then $\fa$ is filtration-rigid.
 \end{prop}
 
 \begin{proof}
 Use (reverse) induction on $r$, where $-\hat\nu \leq r \leq \tilde\nu$.  The $r = \tilde\nu$ case is trivial since $\tau_{\tilde\nu}|_{\ff^{\tilde\nu}} : \ff^{\tilde\nu} \to \fa_{\tilde\nu}$ is a Lie algebra isomorphism.  Assume there is a basis $\{ f_i^\alpha \}_{i=r+1}^{\tilde\nu}$ of $\ff^{r+1}$ with $\tau_i(f_i^\alpha) = e_i^\alpha$ such that $\ad_{\tilde{h}}$ is diagonal (with eigenvalues $\mu_i^\alpha$).  We have $[\tilde{h},f_r^\alpha] = \mu^\alpha_r f_r^\alpha + \sum_{k > r} \sum_\gamma  b^{\alpha,k}_\gamma f_k^\gamma$.  Since $\cE_i \cap \cE_j = \emptyset$ for $i \neq j$, then $\mu^\alpha_r \neq \mu^\beta_{r+1}$, so letting $\hat{f}_r^\alpha = f_r^\alpha + \sum_\beta \sigma^\alpha_\beta f_{r+1}^\beta$ with $\sigma^\alpha_\beta = (\mu^\alpha_r - \mu^\beta_{r+1})^{-1} b^{\alpha,r+1}_\beta$, we can normalize $b^{\alpha, r+1}_\beta = 0$.  Iteratively, we normalize all $b^{\alpha,k}_\gamma = 0$.
 
 We have $[\tilde{h}, [f_i^\alpha,f_j^\beta]] = (\mu^\alpha_i + \mu^\beta_j) [f_i^\alpha,f_j^\beta]$.  If $i \neq j$, then by (i), there is no tail.  If $i = j$, then for $\alpha \neq \beta$, (ii) implies there is no tail.  Thus, $\fa$ is filtration-rigid.
 \end{proof}

 
 \subsection{Exceptional cases and local homogeneity} 
 \label{S:exceptions}

 \subsubsection{Simple exceptions}
 \label{S:simple-ex}
 
 In the complex or split-real case, by passing to twistor spaces the exceptions $A_2 / P_1$, $A_2 / P_{1,2}$, $B_2 / P_1$, $B_2 / P_{1,2}$ (Lemma \ref{L:exceptions}) reduce to\footnote{By Example \ref{ex:B2P12}, only the $(12)$-branch of $B_2 / P_{1,2}$ is exceptional.}: $A_2 / P_1$, $B_2 / P_1$.  (Notice that regularity is preserved here.)  By Recipe \ref{R:red-geom}, we have:
 \[
 \begin{array}{ccccl}
 G/P & H^2_+(\fg_-,\fg) & \fU & \fa = \fg_{-1} \op \fa_0 & \begin{tabular}{c} Eigenvalues for diagonal\\ element in $\fa_0$ \end{tabular} \\ \hline
 A_2 / P_1 & \Atwo{xs}{-5,1} & 4 & 
 {\tiny \l(\begin{array}{c|cc} h & 0 & 0\\ \hline u & 4h & 0 \\ v & s & -5h \end{array}\r)} &
 \begin{array}{l@{\,\,}c@{\,\,}l} \cE_{-1} &=& \{ 3, -6 \}, \\ \cE_0 &=& \{ -9, 0 \} \end{array}\\
 B_2 / P_1 & \Btwo{xs}{-5,4} & 5 & 
 {\tiny \l(\begin{array}{c|ccc|c}
 2h & 0 & 0 & 0 & 0\\ \hline
 u & 3h & 0 & 0 & 0\\
 v & s & 0 & 0 & 0\\
 w & 0 & -s & -3h & 0\\ \hline
 0 & -w & -v & -u & -2h
 \end{array}\r)} & 
 \begin{array}{l@{\,\,}c@{\,\,}l} \cE_{-1} &=& \{ 1, -2, -5 \}, \\ \cE_0 &=& \{ -3, 0 \} \end{array}
 \end{array}
 \]
 By Proposition \ref{P:f-rigid}, $\fa$ is filtration-rigid for both, and so $\fS \leq \fU - 1$.  (By Example \ref{EX:proj-FR}, this is valid for $A_2 / P_1$.)  The models \eqref{E:2d-proj} and \eqref{E:3d-conf} imply $\fS = \fU - 1$.  See also \eqref{E:scalar-2-ODE} for $A_2 / P_{1,2}$.  Thus:
 
 \begin{prop} For complex or split-real (regular, normal) parabolic geometries of type $A_2 / P_1$, $A_2 / P_{1,2}$, or $B_2 / P_1$, we have $\fS = \fU - 1$.  For $B_2 / P_{1,2}$ with $w = (12)$, we have $\fS_w = \fU_w - 1$.
 \end{prop}
 
 Among the exception list, only $A_2 / P_{1,2}$ and $B_2 / P_1$ admit two non-split-real forms:
 \begin{enumerate}
 \item 3-dim.\ CR:  $\fg \cong \fsu(2,1)$, $\fg_0 \cong \bbC = \bbR \op i\bbR$, $\fg_- = \fg_{-2} \op \fg_{-1} \cong \bbR \op \bbC$ (Heisenberg algebra), and $H^2_+(\fg_-,\fg) \cong \bbC$ (homogeneity +4) has action $z \cdot \phi := z\phi$.  If $\phi \neq 0$, then $\fann(\phi) = 0$, so $\fS \leq \fU = 3$. In \cite{Car1932}, Cartan gave models with 3 symmetries.  Thus, $\fS = 3$.

 \item 3-dim.\ Riem.\ conformal: $\fg \cong \fso_{1,4}$, $\fg_0 \cong \bbR \op \fso(3)$, $\fg_- = \fg_{-1} \cong \bbR^3$, and $H^2_+(\fg_-,\fg) \cong \{ \phi \in \tMat_3(\bbR) \mid  \tr(\phi) = 0,\, \phi^\mathsf{T} = \phi \}$ (homogeneity +3) with $\fso(3)$-action $A \cdot \phi = [A,\phi]$.  Any $\phi$ is orthogonally diagonalizable.
If $\phi$ has three distinct eigenvalues, then $\fann(\phi) = 0$.  Otherwise, $\phi$ has two equal nonzero eigenvalues, and so $\dim(\fann(\phi)) = 1$.  Thus, $\fS \leq \fU = 4$, and indeed $\fS = 4$: see \eqref{E:3d-conf}.
 \end{enumerate}
 
  \subsubsection{Semisimple exceptions}
 \label{S:semisimple}
  
 \begin{prop} \label{P:ex} For complex or split-real (regular, normal) $G/P = A_1 / P_1 \times G'/P'$ geometries with $\im(\Kh) \subset \bbV_{\mu}$, $w = (1,k') \in W^\fp(2)$, and $\mu = -w\cdot 2\lambda_1$, we have $\fU_\mu - 1 \leq \fS_\mu \leq \fU_\mu$.
 \end{prop}
 
 \begin{proof}

 Note $\mu = 2\alpha_1 + \alpha_{k'}$, $\phi_0 = e_{\alpha_1} \wedge e_{\alpha_{k'}} \ot e_{\alpha_1} \in \bbV_\mu$.  Take the standard $\fsl_2$-basis $H,X,Y$ with $[H,X] = 2X$, $[H,Y] = -2Y$, $[X,Y] = H$.  Thus, $\fg_- = \tspan\{ Y \} \op \fg_-'$ and $\alpha_1(H) = 2$.
  By Recipe \ref{R:a0},
 \begin{align} \label{E:ss-a0}
 \fa_0 = \fann(\phi_0) = \fh_0 \op \bop_{\gamma \in \Delta(\fg'_{0,\leq 0})} \fg'_\gamma, \qquad \fh_0 := \l\{ -\frac{1}{4} \alpha_{k'}(h') H + h' \mid h' \in \fh' \r\}.
 \end{align}
 Note $\ker(\alpha_{k'}) \subset \fh' \subset \fh_0$ defines a subalgebra $\tilde\fa_0 \subset \fa_0 \cap \fg_0'$ and $\tilde\fa = \tspan\{ Y \} \op \prn^{\fg'}(\fg_-',\fann(e_{\alpha_{k'}})) \subset \fa$, with $\dim(\tilde\fa) = \fU_\mu - 1$.  As in \S \ref{S:realize}, define $\tilde\ff = \tilde\fa$ as vector spaces, but deform the $\tilde\fa$-brackets by the cochain $\phi = e_{\alpha_1} \wedge e_{\alpha_{k'}} \ot e_{-\alpha_1}$, i.e.\ $[Y,e_{-\alpha_{k'}}]_{\tilde\ff} = Y$.  We need to show that $\tilde\ff$ is a Lie algebra.
 
 We make two simplifications: (i) By Remark \ref{RM:reduce-to-two}, take $G'$ to be simple, and (ii) By Theorem \ref{T:corr-Tanaka}, for the Lie algebra structure of $\fa := \fa(\mu)$, take $\fp' = \fp_{k'} \subset \fg'$. (These are assumed only in this paragraph to check that $\tilde{\ff}$ is a Lie algebra.)  By Proposition \ref{P:PR}, $(\fg,\fp,\mu)$ is PR, so $\fa = \fg_- \op \fa_0$.    Write $[Y,u]_{\tilde\ff} = T(u) Y$, $\forall u \in \tilde\ff$, where $T$ vanishes on $\fh'$ and on all root vectors, except $T(e_{-\alpha_{k'}}) = 1$.  For Jacobi, it suffices to check $0 = \Jac_{\tilde\ff}(Y,u,v)$, $\forall u,v \in \tilde\fa' := \tilde\fa \cap \fg'$, i.e.\ $T([u,v]_{\tilde\ff}) = 0$, or $[u,v]_{\tilde\ff}$ has no $e_{-\alpha_{k'}}$-component.  If $u,v \in \fg_-$, this is clear.  Otherwise, if $u \in \tilde\fa_0$, then $[u,v]_{\tilde\ff} = [u,v]$.  If $u \in \ker(\alpha_{k'})$, then $[u,e_{-\alpha_{k'}}] = 0$ and so $T([u,v]) = 0$, $\forall v \in \tilde\fa'$.  On the other hand, if $u = e_\gamma$, with $\gamma \in \Delta(\fg'_{0,\leq 0})$, then since $u \cdot \phi_0 = 0$ and $[u,e_{\alpha_1}] = 0$, then $[u,e_{\alpha_{k'}}] = 0$, i.e. $\gamma + \alpha_{k'} \not\in \Delta$.  Thus, if $v = e_\beta$, with $\beta \in \Delta(\fa)$, then $\gamma+\beta \neq -\alpha_{k'}$.  Hence, $T([u,v]_{\tilde\ff}) = 0$ again, and $\tilde\ff$ is a Lie algebra.
 
 Now returning to the general setting, let $\tilde\fk = \tilde\fa_{\geq 0} \subset \tilde\ff$.  Define the linear map $\vartheta : \tilde\ff \to \fg$ by $\vartheta=\id+\frac{1}{2} H\otimes e_{\alpha_{k'}} -\frac{1}{2} X \otimes e_{\alpha_1}$.  Thus, $\vartheta|_{\tilde{\fk}}$ is the natural inclusion, $\vartheta(Y) = Y - \frac{1}{2} X$, and $\vartheta(e_{-\alpha_{k'}}) = e_{-\alpha_{k'}} + \frac{1}{2} H$. For $u \in \tilde\fk$ and $v \in \tilde\ff$, since $e_{\alpha_1}([u,v])=0$ and $e_{\alpha_{k'}}([u,v])=0$, then $\vartheta([u,v]_{\tilde\ff}) = \vartheta([u,v]) = [u,v] = [u,\vartheta(v)] = [\vartheta(u),\vartheta(v)]$, so $\vartheta$ is $\tilde\fk$-equivariant.  Also, $\vartheta$ induces a linear isomorphism $\tilde\ff / \tilde\fk \cong \fg / \fp$.  The hypotheses of Lemma \ref{L:loc-model} are satisfied.
 
 By construction, $\vartheta$ defines an $\tilde\ff$-invariant Cartan connection $\omega_\vartheta$ on some $\cG = F \times_K P$ over $M = F/K$ (where $F$ may be a semi-local Lie group as before).
Using \eqref{E:kappa-alpha}, $\kappa_\vartheta$ vanishes except for
 \begin{align*}
 \kappa_\vartheta(Y,e_{-\alpha_{k'}}) = [\vartheta(Y),\vartheta(e_{-\alpha_{k'}})] - \vartheta([Y,e_{-\alpha_{k'}}]_{\tilde\ff}) 
 = \l[Y - \frac{1}{2} X, e_{-\alpha_{k'}} + \frac{1}{2} H\r] - \vartheta(Y) = X,
 \end{align*}
 i.e.\ $\kappa_\vartheta = \phi_0$.  Thus, the geometry is regular and normal, and $\fU_\mu -1 \leq \fS_\mu \leq \fU_\mu$.
 \end{proof}
 
 \subsubsection{Local homogeneity}
 
 \begin{thm} \label{T:transitive}  Among all complex or split-real (regular, normal) parabolic geometries of type $(G,P)$, all submaximally symmetric models are locally homogeneous near a non-flat regular point.  Given a $\fg_0$-irrep $\bbV_\mu \subset H^2_+(\fg_-,\fg)$, the same conclusion holds with the restriction that $\im(\Kh) \subset \bbV_\mu$.
 \end{thm}

 \begin{proof}  It suffices to prove the second statement.  For all non-exceptional cases, Theorem \ref{T:main-thm} asserts $\fS_\mu = \fU_\mu$, so the result follows from Corollary \ref{C:transitive} (and Remark \ref{RM:constrained}).
 
 Consider the exceptional geometries. Let $(\cG \stackrel{\pi}{\rightarrow} M, \omega)$ be a submaximally symmetric geometry with $\im(\Kh) \subset \bbV_\mu$.  Given a regular point $x \in M$ and $u \in \pi^{-1}(x)$, we have $\fs(u) \subseteq \fa^{\Kh(u)}$ by Theorem \ref{T:upper}.  We show that $\fg_- \subset \fs(u)$.  From \S \ref{S:simple-ex} and \S \ref{S:semisimple}, $\dim(\fs(u)) = \dim(\inf(\cG,\omega)) \geq \fU_\mu - 1$.  Let $\cO \subset \bbV_\mu$ be the minimal $G_0$-orbit.  We have two cases:
 \begin{enumerate}
 \item[\rm (i)] $\Kh(u) \not\in \cO$: By Proposition \ref{P:Tanaka-lw}, $\dim(\fa^{\Kh(u)}) < \dim(\fa^{\phi_0}) = \fU_\mu$.  Thus, $\fs(u) = \fa^{\Kh(u)} \supset \fg_-$.
 \item[\rm (ii)] $\Kh(u) \in \cO$: We may assume $\Kh(u) = \phi_0$, so $\fs(u) \subseteq \fa^{\phi_0} =: \fa$.  Since $\fg_{-1} \subset \fg_-$ is bracket-generating, then $\fg_- \not\subset \fs(u)$ implies $\fs_{-1}(u) \neq \fg_{-1}$.  Assuming this, $\fs(u) \subsetneq \fa$ must have codimension one, so $\fs_0(u) = \fa_0$.  From Proposition \ref{P:key-bracket}, $[\fs_0(u),\fg_{-1}] \subset \fs_{-1}(u) \neq \fg_{-1}$.  Let $\fh_0 := \fa_0 \cap \fh$.  We easily confirm $[\fh_0,\fg_{-1}] = \fg_{-1}$ for the simple exceptions, and the same follows from \eqref{E:ss-a0} for the semisimple exceptions.  Thus, $[\fs_0(u),\fg_{-1}] = \fg_{-1}$, a contradiction.
 \end{enumerate}
 Thus, in all cases the submaximal symmetric structures are homogeneous near $x$.
 \end{proof}
 
 
 \section{Results and local models for specific geometries}
 \label{S:specific}

 In almost all complex or split-real cases, our Theorem \ref{T:main-thm} asserts that $\fS_\mu = \fU_\mu$.  Exceptions include 2-dim projective, scalar 2nd order ODE, and 3-dim conformal structures (see \S \ref{S:exceptions}).  Recipe \ref{R:red-geom} describes how to efficiently compute $\fU_\mu$.   Sample new results are given in Table \ref{F:sample-submax}, and a complete classification (when $G$ is simple and complex or split-real) is given in Appendix \ref{App:Submax}.

 In this section, we focus on a discussion of parabolic geometries in terms of underlying structures, and describe local models which realize our results on submaximal symmetry dimensions.  Among the examples in this section, the only NPR geometry is 2nd order ODE systems (\S \ref{S:2-ODE}).  All others are PR, so $\fa_+(w) = 0$, $\forall w \in W^\fp_+(2)$, and Recipe \ref{R:red-geom} reduces to using only Recipes \ref{R:dim} and \ref{R:a0}.

 
 \subsection{Conformal geometry}
 \label{S:conf}
 
 Let $n = p+q \geq 3$.  The pseudo-conformal sphere $\mathbb{S}^{p,q}$ embedded as the null projective quadric in $\bbP(\mathbb{R}^{p+1,q+1})$ is the flat model for conformal geometry in signature $(p,q)$.  This is equivalently described as a parabolic geometry of type $\SO_{p+1,q+1} / P_1$, where $P_1$ is the stabilizer of a null line.  The Lie algebra $\fg = \fso_{p+1,q+1}$ is a real form of $\fso_{n+2}(\bbC)$, which is $B_\rkg$ or $D_\rkg$ when $n=2\rkg-1$ or $n=2\rkg - 2$, respectively; $\fg$ is 1-graded by $P_1$.  The space $\bbW = H^2(\fg_-,\fg)$ is the space of Weyl tensors for $n \geq 4$ or Cotton tensors  for $n=3$.
 
 First work over $\bbC$, so $\fg = \fso_{n+2}(\bbC)$, and $\fg_0^{ss} \cong B_{\rkg-1}$ or $D_{\rkg-1}$.  For $H^2_+(\fg_-,\fg)$, Recipe \ref{R:Kostant} gives
 \begin{align*}
 & B_2 / P_1: \begin{tiny} \begin{tikzpicture}[scale=\myscale,baseline=-3pt]
 \dbond{r}{0,0};
 \DDnode{x}{0,0}{-5};
 \DDnode{s}{1,0}{4};
 \useasboundingbox (-.4,-.2) rectangle (1.2,0.55);
 \end{tikzpicture} \end{tiny}, \quad
 B_3 / P_1:  \begin{tiny} \begin{tikzpicture}[scale=\myscale,baseline=-3pt]
 \bond{0,0};
 \dbond{r}{1,0};
 \DDnode{x}{0,0}{-4};
 \DDnode{w}{1,0}{0};
 \DDnode{s}{2,0}{4};
 \useasboundingbox (-.4,-.2) rectangle (2.2,0.55);
 \end{tikzpicture} \end{tiny}, \quad
 B_\rkg / P_1\,\, (\rkg \geq 4): \begin{tiny} \begin{tikzpicture}[scale=\myscale,baseline=-3pt]
 \bond{0,0};
 \bond{1,0};
 \bond{2,0};
 \tdots{3,0};
 \dbond{r}{4,0};
 \DDnode{x}{0,0}{-4};
 \DDnode{w}{1,0}{0};
 \DDnode{s}{2,0}{2};
 \DDnode{w}{3,0}{0};
 \DDnode{w}{4,0}{0};
 \DDnode{w}{5,0}{0};
 \useasboundingbox (-.4,-.2) rectangle (2.2,0.55);
 \end{tikzpicture} \end{tiny}, \\
 &  D_3 / P_1: \begin{tiny} \begin{tikzpicture}[scale=\myscale,baseline=-3pt]
 \diagbond{u}{0,0};
 \diagbond{d}{0,0};
 \DDnode{x}{0,0}{\!\!\!\!-4};
 \DDnode{s}{0.5,0.865}{4};
 \DDnode{w}{0.5,-0.865}{0};
 \useasboundingbox (-.4,-.2) rectangle (0.7,0.55);
 \end{tikzpicture} \end{tiny} \op 
 \begin{tiny} \begin{tikzpicture}[scale=\myscale,baseline=-3pt]
 \diagbond{u}{0,0};
 \diagbond{d}{0,0};
 \DDnode{x}{0,0}{\!\!\!\!-4};
 \DDnode{w}{0.5,0.865}{0};
 \DDnode{s}{0.5,-0.865}{4};
 \useasboundingbox (-.4,-.2) rectangle (0.7,0.55);
 \end{tikzpicture} \end{tiny}, \quad
 D_4 / P_1: \begin{tiny} \begin{tikzpicture}[scale=\myscale,baseline=-3pt]
 \bond{0,0};
 \diagbond{u}{1,0};
 \diagbond{d}{1,0};
 \DDnode{x}{0,0}{-4};
 \DDnode{w}{1,0}{0};
 \DDnode{s}{1.5,0.865}{2};
 \DDnode{s}{1.5,-0.865}{2};
 \useasboundingbox (-.4,-.2) rectangle (1.7,0.55);
 \end{tikzpicture} \end{tiny}, \quad
 D_\rkg / P_1\,\, (\rkg \geq 5): \begin{tiny} \begin{tikzpicture}[scale=\myscale,baseline=-3pt]
 \bond{0,0};
 \bond{1,0};
 \bond{2,0};
 \tdots{3,0};
 \diagbond{u}{4,0};
 \diagbond{d}{4,0};
 \DDnode{x}{0,0}{-4};
 \DDnode{w}{1,0}{0};
 \DDnode{s}{2,0}{2};
 \DDnode{w}{3,0}{0};
 \DDnode{w}{4,0}{0};
 \DDnode{w}{4.5,0.865}{0};
 \DDnode{w}{4.5,-0.865}{0};
 \useasboundingbox (-.4,-.2) rectangle (5.4,0.55);
 \end{tikzpicture} \end{tiny}
 \end{align*} 
 Above, $W^\fp_+(2) = \{ (12) \}$ always, except $W^\fp_+(2) = \{ (12), (13) \}$ for $D_3 / P_1$.  Let $n \geq 4$.  The geometry is PR by Corollary \ref{C:g-PR}. Then using Recipes \ref{R:dim} and \ref{R:a0},
 \Ben
 \item $n \geq 5$: $\lambda_\fg = \lambda_2$, $w = (12)$, $J_w = \{ 3 \}$.  For $n \neq 6$, we have $\fp_w^{\opn} \cong \fp_2 \subset \fg_0^{ss}$.
 \begin{align*}
 \dim(\fa_0) &= \dim(\fp_2) 
 = \mycase{
 \half\l(\dim(B_{\rkg-1}) + 1 + \dim(A_1) + \dim(B_{\rkg-3})\r), &\quad n \mbox{ odd};\\ 
 \half\l(\dim(D_{\rkg-1}) + 1 + \dim(A_1) + \dim(D_{\rkg-3})\r), &\quad n \mbox{ even};}\\
 &= \l\{ \begin{array}{ll} 
 2\rkg^2 - 7\rkg + 10, & n \mbox{ odd}; \\
 2\rkg^2 - 9\rkg + 14, & n \mbox{ even} 
 \end{array} \r.  \qRa \dim(\fa) = \dim(\fg_{-1}) + \dim(\fa_0) = \binom{n-1}{2} + 6.
 \end{align*}
 The $n=6$ case agrees with this formula, but here $\fp_w^{\opn} \cong \fp_{2,3} \subset \fg_0^{ss}$.
 \item $n = 4$: $D_3 / P_1$, $\lambda_\fg = \lambda_2 + \lambda_3$, $\fg_0^{ss} = A_1 \times A_1$.  For $w = (12)$, $J_w = \{ 3 \}$, $\fp_w^{\opn} \cong A_1 \times \fp_1 \subset \fg_0^{ss}$, so $\dim(\fa_0) = 5$ and $\dim(\fa(w)) = 9$.  By symmetry, for $w = (13)$, $\dim(\fa(w)) = 9$ as well.
 \Een
 Thus, $\fU^\bbC = \binom{n-1}{2} + 6$ for $n \geq 4$.  In any signature over $\bbR$, $\fS \leq \fU \leq \fU^\bbC$ by Corollary \ref{C:upper}. For $n \geq 4$, $\fa_0$ stabilizes a null 2-plane (in the standard representation of $\fg_0^{ss}$).  These do not exist in Riemannian and Lorentzian signatures, so $\fS \leq \fU < \fU^\bbC$ in these cases.  In all other signatures, we exhibit a model realizing the upper bound $\fU^\bbC$.
 
 \subsubsection{Non-Riemannian and non-Lorentzian signatures} 
 \label{S:nR-nL}
  
 Consider the $(2,2)$ pp-wave metric:\footnote{In  \cite{Kru2012}, the signature $(2,2)$ pp-wave metric was announced as having submaximal conformal symmetry dimension.  No proof was given there, but using tools developed in this paper, we can confirm that this is indeed correct.}
 \begin{align} \label{E:pp-2,2}
 \ppmetric{2,2} = y^2 dw^2 + dw dx + dy dz.
 \end{align}
 This has 9 conformal symmetries, of which $\bX_1, ..., \bX_8$ are Killing fields, and $\bT$ is a homothety:
 \begin{align*}
 \bX_1 &= \p_x, \quad
 \bX_2 = \p_z, \quad
 \bX_3 = \p_w, \quad
 \bX_4 = -y\p_x + w \p_z, \quad \bX_5 = 3(z + yw^2)\p_x - 3w\p_y - w^3 \p_z, \\
 \bX_6 &= 2yw\p_x - \p_y - w^2 \p_z, \quad
 \bX_7 = -x\p_x - y\p_y + w\p_w + z\p_z, \quad
 \bX_8 = 2 y^3 \p_x - 3 y \p_w + 3 x \p_z,\\
  \bT &= 2x \p_x + y \p_y + z\p_z.
 \end{align*}
 
 \begin{lemma} Let $n = p+q +4$, and $\eucmetric{p,q}$ the flat Euclidean metric of signature $(p,q)$.  Then $\metric = \ppmetric{2,2} + \eucmetric{p,q}$ has conformal symmetry algebra of dimension $\binom{n-1}{2} + 6$.
 \end{lemma}
 
 \begin{proof}  The Weyl tensor of $\metric$ is $4 (dy \wedge dw)^2$, so $\metric$ is not conformally flat.  Writing $\epsilon_i = \pm 1$, the metric $\eucmetric{p,q} = \sum_{i=1}^{p+q} \epsilon_i (du_i)^2$ admits the Killing fields $\bU_i = \p_{u_i}$ and $\bV_{ij} = \epsilon_i u_i \p_{u_j} - \epsilon_j u_j \p_{u_i}$, where $i < j$.  The metric $\metric$ admits the Killing fields $\bU_i$, $\bV_{ij}$, $\bX_k$, as well as $\bY_i = 2\epsilon_i u_i\p_z - y \p_{u_i}$, and $\bW_i = 2\epsilon_i u_i \p_x - w\p_{u_i}$.
 The vector field $\tilde\bT = \bT +  \sum_{i=1}^{p+q} u_i \p_{u_i}$ is a homothety for $\metric$.
 \end{proof}

 Thus, we have proven (see \S \ref{S:3d-conf} for the $n=3$ case):
 
 \begin{thm} \label{T:conf-sharp} For conformal geometry in dimension $n \geq 4$, we have $\fS \leq \binom{n-1}{2} + 6$.  Except for Riemannian and Lorentzian signatures, this upper bound is sharp.  For $n=3$, we have $\fS = 4$.
  \end{thm}
 
 \subsubsection{4-dimensional Lorentzian}
 \label{S:4d-Lor}
 
 Here, $\fg = \fso_{2,4}$ and $\fg_0 \cong \bbR \op \fsl_2(\bbC)_\bbR$, where $\fsl_2(\bbC)_\bbR \cong \fso_{1,3}$ is the {\em real} Lie algebra underlying $\fsl_2(\bbC)$.  The grading element $Z \in \fz(\fg_0) \cong \bbR$ acts by $+2$ on the $\fsl_2(\bbC)_\bbR$-irrep $\bbW = H^2_+(\fg_-,\fg) \cong \bigodot{}^{\!4} (\bbC^2)$.  Weyl tensors are identified with complex binary quartics, and their classification according to root type is the well-known Petrov classification \cite{Pet1954, Pen1960}. 
 
 To each Petrov type, there is a collection $\cO$ of $G_0$-orbits.  We have (by prolongation-rigidity):
 \[
 \fU_\cO := \max\{ \dim(\fa^\phi) \mid 0 \neq \phi \in \cO \} = 4 + \max\{  \dim(\fann(\phi)) \mid 0 \neq \phi \in \cO \}.
 \]
 It suffices to maximize $\dim(\fann(\phi))$ among representative elements from $\cO$.  Then $\fS_\cO \leq \fU_\cO$ by Remark \ref{RM:constrained}.  Let $\bbC^2 = \tspan_\bbC\{ x,y \}$, and
 $H = \pmat{1 & 0 \\ 0 & -1}$, 
 $X = \pmat{0 & 1 \\ 0 & 0}$, 
 $Y = \pmat{0 & 0 \\ 1 & 0}$, so $\fsl_2(\bbC)_\bbR$ has $\bbR$-basis $\{ H, iH, X, iX, Y, iY \}$.  A simple case analysis yields Table \ref{tbl:Petrov} (except the last column).
 
 \begin{table}[h]
 $\begin{array}{|c|c|c|c|c|c|} \hline
 \mbox{Type} & \mbox{Normal form $\phi$} & \bbR\mbox{-basis for } \fann(\phi) & \dim(\fa^\phi)& \fa^\phi \mbox{ filtration-rigid?} \\ \hline
 \mbox{N} & y^4 & Y, iY, 2Z+ H & 7 & \times\\
 \mbox{III} & xy^3 & Z+ H & 5 & \checkmark\\
 \mbox{D} & x^2 y^2 & H, iH & 6 & \times\\
 \mbox{II} & x^2 y(x-y) & \cdot & 4 & \times\\
 \mbox{I} & xy(x-y)(x-ky) & \cdot & 4 & \times\\ \hline
 \end{array}$
 \caption{Petrov types and upper bounds on conformal symmetry dimensions}
 \label{tbl:Petrov}
 \end{table}

 The conformal structures below establish sharpness of the upper bounds (and filtration non-rigidity) in Table \ref{tbl:Petrov} in all cases except type III.
 \Ben
 \item[N:] $\ppmetric{3,1} = dy^2 + dz^2 + dwdx + y^2 dw^2$ (signature $(3,1)$ pp-wave):
 \begin{align*}
 \bX_1 &= \p_x, \quad \bX_2 = \p_z, \quad \bX_3 = \p_w, \quad
 \bX_4 = -2z\p_x + w\p_z, \\
 \bX_5 &= e^{-w}(2y\p_x + \p_y), \quad
 \bX_6 = e^w(-2y\p_x + \p_y), \quad \bT = 2x\p_x + y\p_y + z\p_z.
 \end{align*}
 \item[III:] $\metric = \frac{3}{|\Lambda|} dz^2 + e^{4z} dx^2 + 4 e^z dx dy + 2e^{-2z}(dy^2 + du dx)$ (Kaigorodov metric \cite[(12.35)]{SKMH2003}):
 \[
 \bX_1 = \p_u, \quad \bX_2 = \p_x,  \quad \bX_3 = \p_y, \quad  \bX_4 = 2x \p_x - y \p_y - \p_z - 4u\p_u.
 \]
 Another metric is
 $\metric = \frac{r^2}{x^3} (dx^2 + dy^2) - 2 du dr + \frac{3}{2} x du^2$ (Siklos metric \cite[(38.1)]{SKMH2003}):
 \[
 \bX_1 = \p_y, \quad \bX_2 = \p_u, \quad \bX_3 = 2(x\p_x + y\p_y) + r\p_r - u\p_u, \quad \bT = u\p_u + r\p_r.
 \]
 \item[D:] $\metric = a^2(dx^2 + \sinh^2(x) dy^2) + b^2(dz^2 - \sinh^2(z) dt^2)$ ($a,b$ constant)
 \begin{align*}
 \bX_1 &= \p_y, \quad \bX_2 = \p_t, \quad
 \bX_3 = e^{-t} (\p_z + \coth(z) \p_t), \quad
 \bX_4 = e^t (\p_z - \coth(z) \p_t), \\
 \bX_5 &= -\cos(y) \p_x + \coth(x) \sin(y) \p_y, \quad
 \bX_6 = \sin(y) \p_x + \coth(x) \cos(y) \p_y.
 \end{align*}
 This is a product of two spaces of constant curvature \cite[(12.8)]{SKMH2003}.
 \item[II:] $\metric = dz^2+ e^{-2z} (dy^2 + 2dxdu) - 35 e^{4z} dx^2 + e^{-8z} du^2$
 \[
  \bX_1 = \p_u, \quad \bX_2 = \p_x, \quad \bX_3 = \p_y, \quad  \bX_4 = 2x \p_x - y \p_y - \p_z - 4u\p_u.
 \]
 This metric is not Einstein, has Ricci scalar a nonzero constant, and appears to be new.\footnote{In \cite{SKMH2003}, Table 38.3 incorrectly lists (12.29) as a type II metric with 4-dimensional isometry group, while (12.29) is in fact type D, as indicated at the bottom of p.179.  Also, the type II metric (13.65) is indicated as having four Killing vectors (13.66), but the fourth listed vector field is incorrect.  Professor Malcolm MacCallum has indicated to us that no type II metric with 4-dimensional isometry group appears to have been known in the literature.}  Symmetry-preserving deformations of the type III Kaigorodov metric led to the ansatz
 \[
 \metric = dz^2 + a_1 e^{-2z} dy^2 + a_2 e^{-2z} dxdu + a_3 e^{4z} dx^2 + a_4 e^z dy dx + a_5 e^{-8z} du^2, \quad a_i \in \bbR.
 \]
 Imposing the type II condition led to the metric indicated above.
 \item[I:] $\metric = dx^2 + e^{-2x} dy^2 + e^x \cos(\sqrt{3} x) (dz^2 - dt^2) - 2e^x\sin(\sqrt{3}x) dz  dt$ (Petrov metric \cite[(12.14)]{SKMH2003}):
 \[
  \bX_1 = \p_y, \quad \bX_2 = \p_z, \quad \bX_3 = \p_t, \quad \bX_4 = \p_x + y \p_y + \half (\sqrt{3} t - z) \p_z - \half ( t + \sqrt{3} z) \p_t.
 \]
 \Een

 \begin{thm} \label{T:Petrov} In 4-dimensional conformal Lorentzian geometry, the maximal dimension of the conformal symmetry algebra for metrics of constant Petrov type is:
 \begin{center}
 \begin{tabular}{|c||c|c|c|c|c|} \hline
 Petrov type & N & III & D & II & I\\ \hline
 max. sym. dim.  & 7 & 4 & 6 & 4 & 4 \\ \hline
 \end{tabular}
 \end{center}
 All models realizing these upper bounds are locally homogeneous near a regular point.
 \end{thm}
 
 \begin{proof}  For type III metrics, we must show that the associated (5-dimensional) $\fa = \fg_{-1} \op \fa_0$ is filtration-rigid.
  As a $\fsl_2(\bbC)_\bbR$-representation, $\fg_{-1}$ is the space of Hermitian $2 \times 2$-matrices $\cH$ with the action of $A \in \fsl_2(\bbC)_\bbR$ given by $M \mapsto AM + M A^*$.  For $\cH$, take the standard basis
 \[
 e_{-1}^1 = \pmat{1 & 0\\ 0 & 0}, \quad 
 e_{-1}^2 = \pmat{0 & 1\\ 1 & 0}, \quad
 e_{-1}^3 = \pmat{0 & i\\ -i & 0}, \quad
 e_{-1}^4 = \pmat{0 & 0\\ 0 & 1}.
 \]
 Since $H$ acts by $\diag(2,0,0,-2)$, and $Z$ acts as $-1$ on $\fg_{-1}$, then $e_0 := Z+H$ acts diagonally with eigenvalues $\cE_{-1} = \{ 1, -1, -1, -3 \}$, $\cE_0 = \{ 0 \}$.  (All other brackets on $\fa$ are trivial.)  Let $\ff$ be a filtered deformation of $\fa$.  The first hypothesis of Proposition \ref{P:f-rigid} is satisfied, as is (i), but (ii) is not.  Following the argument there, we still have $\tilde{e}_0 \in \ff^0$ which acts diagonally in some basis $\{ f_i^\alpha \}$ with the same eigenvalues as above, and with the only other non-trivial brackets $[f_{-1}^1,f_{-1}^2] = a \tilde{e}_0$, $[f_{-1}^1,f_{-1}^3] = b \tilde{e}_0$.  By the Jacobi identity, $0 = \Jac_\ff(f_{-1}^1, f_{-1}^2,f_{-1}^3) = bf_{-1}^2 - af_{-1}^3$, so $a = b=0$.  Thus, $\fa$ is filtration-rigid.
 In non-type III cases where the upper bound is realized, $\fg_- \subset \fs(u)$, so local homogeneity follows.  For type III, note $[\fa_0,\fg_{-1}] = \fg_{-1}$.  By Proposition \ref{P:key-bracket}, $[\fs_0(u),\fg_{-1}] \subset \fs_{-1}(u)$ with $\fs_0(u) \supset \fa_0$.  Thus, we must have $\fs_0(u) = 0$ and $\fs_{-1}(u) = \fg_{-1}$, and local homogeneity follows.
 \end{proof}
 
 \subsubsection{General Lorentzian and Riemannian}
 
 For the Lorentzian case, let $\metric = \ppmetric{3,1} + \eucmetric{p,0}$, where $\ppmetric{3,1}$ was given in \S \ref{S:4d-Lor}, and $\eucmetric{p,0} = \sum_{i=1}^p (du^i)^2$.  Then $\metric$ has Killing fields $\bX_k$ (see those for $\ppmetric{3,1}$), $\bU_i = \p_{u^i}$, $\bV_{ij} = u^i \p_{u^j} - u^j \p_{u^i}$, $\bY_i = u^i \p_z - z \p_{u^i}$, $\bW_i = 2 u^i \p_x - w\p_{u^i}$, and the homothety $\wt\bT = \bT +  \sum_{i=1}^{p} u_i \p_{u_i}$, so $\dim(\cS) = \binom{n-1}{2} + 4$.  Since $\fU^\bbC = \binom{n-1}{2} + 6$ is not realizable, it remains to show that neither is $\binom{n-1}{2} + 5$.  This is done in \cite{DT-Weyl}, from which it follows that $\fS = \binom{n-1}{2} + 4$.

 Given a Riemannian metric $\metric$ with nowhere vanishing Weyl tensor, there exists a conformally equivalent metric $\tilde\metric = \lambda^2 \metric$ such that all conformal Killing fields of $\metric$ are Killing fields for $\tilde\metric$ \cite{Nag1958}.  Since the symmetry algebra is determined by restriction to any neighborhood, $\fS$ in the conformal Riemannian case is equal to the maximal dimension of the isometry algebra among metrics which are not conformally flat.  According to Egorov's final result in \cite{Ego1962} (incorporating results in his earlier article \cite{Ego1956}), this number is exactly $\binom{n-1}{2} + 3$.  This is realized by the product of spheres $\bbS^2 \times \bbS^{n-2}$ with their round metrics which is not conformally flat for $n \geq 4$, and has $\SO_3 \times \SO_{n-1}$ symmetry.  However, we caution that there are omitted exceptions to Egorov's statement when $n=4$ and $n=6$.\footnote{For $n=4$, the submaximally symmetric structure is {\em unique}: $\bbC\bbP^2$, cf. Egorov \cite{Ego1955}.} In particular, with its Fubini--Study metric, $\bbC\bbP^m$ is not conformally flat, has real dimension $n=2m$, and has $\SU(m+1)$ symmetry group of dimension $m^2 + 2m$.  This is strictly less than $\binom{n-1}{2} + 3$ for $n > 8$, equal to it when $n=8$, and greater than it when $n=4$ or $6$.

 For an algebraic proof of $\fS$ in the conformal Lorentzian and Riemannian cases, see \cite{DT-Weyl}.
 
 \subsubsection{3-dimensional conformal} 
 \label{S:3d-conf}
 
 In \S \ref{S:exceptions}, we showed $\fS \leq 4$.  Indeed, $\fS = 4$ via the models:
 \begin{align} \label{E:3d-conf}
 \begin{array}{c|c}
 dx^2 + dy^2 + (dz - xdy)^2 & \begin{array}{l} \bX_1 = \p_y, \quad \bX_2 = \p_z, \quad \bX_3 = \p_x + y \p_z, \\ \bX_4 = y\p_x - x\p_y + \half (y^2 - x^2) \p_z \end{array}\\ \hline
 dxdy + (dz - xdy)^2 & \begin{array}{l} \bX_1 = \p_y, \quad \bX_2 = \p_z, \quad \bX_3 = \p_x + y \p_z, \\ \bX_4 = x\p_x - y\p_y \end{array}
 \end{array}
 \end{align}

 
 \subsection{Bracket-generating distributions}
 
 \subsubsection{$(2,3,5)$-geometry} \label{S:G2P1} The structure underlying a $G_2 / P_1$ geometry is a $(2,3,5)$-distribution $D$.  Locally, these can always be put in Goursat form: there exist coordinates $(x,y,p,q,z)$, such that $D$ is spanned by $\p_q$ and $\frac{d}{dx} := \p_x + p \p_y + q\p_p + F \p_z$, where $F = F(x,y,p,q,z)$. We have
 \[
 \l[\p_q,\frac{d}{dx}\r] = \p_p + F_q \p_z, \quad 
 \l[\p_q, \p_p+F_q\p_z\r] = F_{qq} \p_z, \quad
 \l[\frac{d}{dx}, \p_p + F_q\p_z\r] = - \p_y + H \p_z,
 \]
 for some $H = H(F)$.  Thus, $D$ has generic growth $(2,3,5)$ iff $F_{qq} \neq 0$.
The flat model is obtained when $F = q^2$, and this has 14-dimensional symmetry algebra isomorphic to $\Lie(G_2)$ \cite{Car1910}.  
 
 We work over $\bbC$. We have a 3-grading $\fg = \fg_{-3} \op ... \op \fg_3$ with $\fg_0 = \fgl(\fg_{-1}) \cong \fgl_2(\bbC)$ and 
 \[
 \fg_{-1} \cong \bbC^2, \quad \fg_{-2} \cong \bbC, \quad \fg_{-3} \cong \bbC^2, \quad \fg_j \cong (\fg_{-j})^*,
 \]
 as $\fg_0$-modules.  In Example \ref{EX:G2P1}, we derive $H^2_+(\fg_-,\fg) = \Gdd{xs}{-8,4} \cong \bigodot^4(\bbC^2)^*$ and $\fS = 7$, which recovers Cartan's result \cite{Car1910}.  (The geometry is PR.)  To each root type, we have a collection of $G_0 \cong \GL_2(\bbC)$ orbits $\cO$ in $\bigodot^4(\bbC^2)^*$ and we analyze each (as in \S \ref{S:4d-Lor}).  Let $\{ e_1, e_2\}$ and $\{ \omega_1, \omega_2 \}$ be dual bases in $\fg_{-1}$ and $\fg_1$ respectively.  
  Take 
 $X= \pmat{0 & 1\\ 0 & 0}$,
 $Y= \pmat{0 & 0\\ 1 & 0}$,
 $H= \pmat{1 & 0\\ 0 & -1}$,
 $I= \pmat{1 & 0\\ 0 & 1}$ for a $\fg_0 = \fgl(\fg_{-1})$ basis.
 (Note the grading element $Z = -I$.)  Symmetry bounds for each root type are given in Table \ref{F:G2-P1}.
 Some models, found either by Cartan \cite{Car1910} or Strazzullo \cite[\S 6.10]{Str2009}, are given in Table \ref{F:G2-P1-ex}.  We refer to \cite{Str2009} for the symmetry algebras.
 
 \begin{table}[h] 
 $\begin{array}{|c|c|c|c|c|} \hline
 \mbox{Type} & \mbox{Normal form $\phi$} & \mbox{Basis for } \fann(\phi) & \dim(\fa^\phi) & \fa^\phi \mbox{ filtration-rigid?} \\ \hline
 {}(4) & \omega_1^4 & Y, H-I & 7 & \times\\
 {}(3,1) & \omega_1^3 \omega_2 & 2H - I & 6 & \checkmark\\
 {}(2,2) & \omega_1^2 \omega_2^2 & H & 6 & \times\\
 {}(2,1,1) & \omega_1^2 \omega_2 (\omega_1-\omega_2) & \cdot & 5 & \times\\
 {}(1,1,1,1) & \omega_1 \omega_2 (\omega_1 - \omega_2)(\omega_1-k\omega_2) & \cdot & 5 & \times\\ \hline
 \end{array}$
 \caption{Root types and upper bounds on symmetry dimensions for $G_2 / P_1$ geometries}
  \label{F:G2-P1}
 \end{table}

 \begin{table}[h]
 $\begin{array}{|c|c|c|c|c|} \hline
 \mbox{Type} & F(x,y,p,q,z) & \dim(\cS)& \mbox{Reference in \cite{Str2009}}\\ \hline
 {}(4) & q^3 & 7 & 6.1.1\\
 {}(3,1) & p^3 + q^2 & 4 & 6.1.3\\
 {}(2,2) & y + \ln(q) & 6 &  6.3.4\\
 {}(2,1,1) & pq^2 & 5 & 6.1.6\\
 {}(1,1,1,1) & pq^3 & 5 & 6.1.7\\ \hline
 \end{array}$
 \caption{$(2,3,5)$-distributions realizing each root type}
  \label{F:G2-P1-ex}
 \end{table} 
 
  \begin{thm} \label{T:G2P1-bounds} Among $(2,3,5)$-distributions with constant root type:
 \begin{center}
 \begin{tabular}{|c||c|c|c|c|c|} \hline
 Root type & $(4)$ & $(3,1)$ & $(2,2)$ & $(2,1,1)$ & $(1,1,1,1)$\\ \hline
 max. $\dim(\cS)$  & $7$ & $4$ or $5$ & $6$ & $5$ & $5$ \\ \hline
 \end{tabular}
 \end{center}
 Except for type $(3,1)$, all models realizing these upper bounds are locally homogeneous near a regular point.\footnote{Using Cartan's reduction method \cite{Car1910}, R.L. Bryant (private communication) showed that there are no type $(3,1)$ models with 5 symmetries. Thus, $4$ is the maximal symmetry dimension for root type $(3,1)$.}
 \end{thm}
 
 \begin{proof} We show that in type $(3,1)$, $\fa = \fg_- \op \fa_0$ is filtration-rigid.  Here, $e_0 := 2H-I$ spans $\fa_0$.  Choose root vectors such that
 \[
 e_{-1}^1 = e_{-\alpha_1}, \quad e_{-1}^2 = e_{-\alpha_1 - \alpha_2}, \quad e_{-2} = e_{-2\alpha_1 - \alpha_2}, \quad
 e_{-3}^1 = e_{-3\alpha_1 - \alpha_2}, \quad e_{-3}^2 = e_{-3\alpha_1 - 2\alpha_2}
 \]
 satisfy commutator relations $[e_{-1}^1,e_{-1}^2] = e_{-2}$ and $[e_{-1}^i,e_{-2}] = e_{-3}^i$.
 Then $\ad_{e_0}$ is diagonal with 
 \[
 \cE_0 = \{ 0 \}, \quad \cE_{-1} = \{ 1, -3 \}, \quad \cE_{-2} = \{ -2 \}, \quad \cE_{-3} = \{ -1, -5 \}.
 \]
  Let $\ff$ be a filtered deformation of $\fa$.  The first hypothesis of Proposition \ref{P:f-rigid} is satisfied, as is (i), but (ii) is not.  Similar to the Petrov type III case, there is $\tilde{e}_0 \in \ff^0$ which acts diagonally in some basis $\{ f_i^\alpha \}$ with the same eigenvalues as above, and with the only other non-trivial brackets 
  \[
  [f_{-1}^1,f_{-1}^2] = f_{-2}, \quad [f_{-1}^1, f_{-2}] = f_{-3}^1, \quad [f_{-1}^2, f_{-2}] = f_{-3}^2, \quad
  [f_{-1}^1,f_{-3}^1] = a \tilde{e}_0, \quad [f_{-2},f_{-3}^1] = b f_{-1}^2.
  \]
 Since $0 = \Jac_\ff(f_{-1}^1, f_{-3}^1, f_{-3}^2 ) = -5a f_{-3}^2$ and $0 = \Jac_\ff(f_{-1}^1, f_{-2}, f_{-3}^1 ) = (2a - b) f_{-2}$, then $a=b=0$ and so $\fa$ is filtration-rigid. Aside from type $(3,1)$, when the upper bound is realized, $\fg_- \subset \fs(u)$.
  \end{proof}
  
 Theorem \ref{T:G2P1-bounds} recovers the bounds found by Cartan \cite{Car1910} using his method of equivalence.

 \subsubsection{$(3,6)$-geometry} Bryant studied $(3,6)$-distributions in \cite{Bry1979,Bry2006}, and constructed the associated $B_3 / P_3$ geometry, where $B_3 = \SO_7(\bbC)$ (or $\SO_{3,4}$).  This is a 2-graded geometry with $\fg_0  \cong \fgl_3(\bbC)$, $\fg_{-1} = \bbC^3$, and $\bigwedge^2 \fg_{-1} \cong \fg_{-2} \cong \bbC^3$.  From $H^2_+(\fg_-,\fg) \cong \Bthree{ssx}{2,2,-6} = \bbV_{-w\cdot\lambda_2}$, generated by $w = (32)$, we obtain $\fS = 11$.  On $(x_1,x_2,x_3,y_1,y_2,y_3)$-space, we give three models $D = \{ \theta_1 = \theta_2 = \theta_3 = 0 \}$, with $\theta_1 = dy_1 - x_3 dx_2$, $\theta_2 = dy_2 - x_1 dx_3$, and:
 \begin{enumerate}
 \item[(a)] $\theta_3 = dy_3 - x_2 dx_1$\quad (21 symmetries -- flat model);
 \item[(b)] $\theta_3 = dy_3 - x_2 dx_1 - (x_1 x_3)^2 dx_3$\quad (11 symmetries -- submaximally symmetric model);
 \item[(c)] $\theta_3 = dy_3 - (x_2 + y_3) dx_1$\quad (10 symmetries).
 \end{enumerate}
 The symmetries can explicitly be found in {\tt Maple 17}.  The following code does this for model (b):
 \begin{verbatim}
 with(DifferentialGeometry): with(GroupActions): DGsetup([x1,x2,x3,y1,y2,y3],M):
 dist:=evalDG([dy1-x3*dx2,dy2-x1*dx3,dy3-x2*dx1-(x1*x3)^2*dx3]):
 InfinitesimalSymmetriesOfGeometricObjectFields([dist],output="list"); nops(%);
 \end{verbatim}

 Model (b) was obtained by following the construction in Lemma \ref{L:def-a}.  Take $\bbC^7$ with the standard anti-diagonal symmetric $\bbC$-bilinear form, define $\epsilon_i \in \fh^*$ by $\epsilon_i( \diag(a_1,a_2,a_3,0,-a_3,-a_2,-a_1)) = a_i$.  The simple roots are $\alpha_1 = \epsilon_1 - \epsilon_2$, $\alpha_2 = \epsilon_2 - \epsilon_3$, $\alpha_3 = \epsilon_3$.
 The lowest weight of $H^2_+(\fg_-,\fg)$ is $-w\cdot\lambda_2 = -\alpha_1 + 3\alpha_3 = -\epsilon_1 + \epsilon_2 + 3\epsilon_3$.  Since $\Phi_w = \{ \alpha_3, \alpha_2 + 2\alpha_3 \}$, and $w(-\lambda_2) = -\lambda_1 - \lambda_2 + 2\lambda_3 = -\alpha_1 - \alpha_2$, then $\phi_0 = e_{\alpha_3} \wedge e_{\alpha_2 + 2\alpha_3} \ot e_{-\alpha_1 - \alpha_2}$.
Letting $\fa_0 = \fann(\phi_0)$, then $\fa = \fg_- \op \fa_0$ has general element:
 \[
 \mat{ccc|c|ccc}{ a_{22} + 3a_{33} & 0 & 0 & 0 & 0 & 0 & 0\\
 a_{21} & a_{22} & 0 & 0 & 0 & 0 & 0\\
 a_{31} & a_{32} & a_{33} & 0 & 0 & 0 & 0\\ \hline
 a_{41} & a_{42} & a_{43} & 0 & 0 & 0 & 0\\ \hline
 a_{51} & a_{52} & 0 & -a_{43} & -a_{33} & 0 & 0\\
 a_{61} & 0 & -a_{52} & -a_{42} & -a_{32} & -a_{22} & 0\\
 0 & -a_{61} & -a_{51} & -a_{41} & -a_{31} & -a_{21} & -(a_{22}+3a_{33})}.
 \]
 Let $E_{ij}$ denote the matrix with $a_{ij} = 1$ and zeros elsewhere.  If we take the root vectors
 \[
 e_{-\alpha_3} = E_{43} \in \fg_{-1}; \qquad
 e_{-\alpha_2 - 2\alpha_3} = E_{52} \in \fg_{-2}; \qquad
 e_{-\alpha_1 - \alpha_2} = E_{31} \in \fg_0,
 \]
 then $\phi_0$ yields a filtered deformation $\ff$ of $\fa$ via $[E_{43},E_{52}] = E_{31}$.  (Note that $[E_{43}, E_{52}] = 0$ in $\fa$.)  Letting $\fk = \fa_0$, the distribution is $\fg_{-1} / \fk \subset \ff / \fk$.  Integrating the structure equations gives model (b).

 \subsection{Systems of 2nd order ODE}
 \label{S:2-ODE}
 
 Let $m \geq 2$.  Consider a 2nd order ODE system
 \begin{align} \label{E:2-ODE}
 \ddot{x}^i = f^i(t,x^j,\dot{x}^j), \qquad 1 \leq i \leq m
 \end{align}
 in $m = \rkg - 1$ dependent variables, up to {\em point transformations}.  Given the jet spaces $J^k = J^k(\bbR,\bbR^m)$ with contact systems $\cC^k$,   \eqref{E:2-ODE} defines a submanifold $M \subset J^2$ transverse to $\pi^2_1 : J^2 \to J^1$.  Point transformations are diffeomorphisms of $J^2$ which preserve $\cC^2$; all are prolongations of diffeomorphisms of $J^0$.  On $J^2$, in canonical coordinates $(t,x^j,p^j,q^j)$, $\cC^2$ is spanned by $\{ dx^j - p^j dt, \, dp^j - q^j dt \}$ and $M = \{ q^i = f^i(t,x^j,p^j) \}$ fibers over $J^0$.  This is an ``external'' description of the geometry.  Pulling back $\cC^2$ to $M$ yields an ``internal'' description: a manifold of dimension $n = 2m+1$, with coordinates $(t,x^i,p^i)$, and a distribution $D = L \op \Pi$ consisting of:
 \begin{enumerate}
 \item the line field $L$ spanned by $\frac{d}{dt} := \p_t + p^i \p_{x^i} + f^i \p_{p^i}$ whose integral curves correspond to solutions of \eqref{E:2-ODE}, and
 \item the integrable $m$-dimensional distribution $\Pi$ spanned by $\{ \p_{p^i} \}_{i=1}^m$.  This is the fibre of the submersion $M \to J^0$, which is preserved by all point transformations.
 \end{enumerate}
 The flat model is $\ddot{x}^i = 0$, $1 \leq i \leq m$, with symmetry algebra isomorphic to $\fg = A_{m+1} = \fsl(m+2,\bbR)$.

 The {\em Fels invariants} \cite{Fels1995}, namely the Fels curvature and torsion, are respectively:
 \begin{align} \label{E:Fels-inv}
 S^i_{jkl} = G^i_{jkl} - \frac{3}{m+2} G^r_{r(jk} \delta^i_{l)}, \qquad T^i_j = F^i_j - \frac{1}{m} \delta^i_j F^k_k,
 \end{align}
 where
 \begin{align*}
G^i_{jkl} =  \frac{\p^3 f^i}{\p p^j \p p^k \p p^l}, \qquad  F^i_j = \half \frac{d}{dt} \l( \frac{\p f^i}{\p p^j} \r) - \frac{\p f^i}{\p x^j} - \frac{1}{4} \frac{\p f^i}{\p p^k} \frac{\p f^k}{\p p^j}.
 \end{align*}

 When $m \geq 3$, the structure above (also called {\em path geometry}) underlies an $A_{m+1} / P_{1,2}$ geometry, where $A_{m+1} = \SL(m+2,\bbR)$.  The geometry is 2-graded, with $\fg_0 = \bbR^2 \op \fsl(\Pi_{-1})$, $\fg_{\pm 1} = L_{\pm 1} \op \Pi_{\pm 1}$, $\fg_{\pm 2} = L_{\pm 1} \ot \Pi_{\pm 1}$, and $W^\fp_+(2) = \{ (21), (12) \}$ generate homogeneity $+3$ and $+2$ modules:
 \begin{align} \label{E:H2-ODE}
 H^2_+(\fg_-,\fg) &= \pathwts{qxswws}{0,-4,3,0,0,1} \op \pathwts{xxswws}{-4,1,1,0,0,1}\\
 &= L_1 \ot \l(\bigodot{}^3 \Pi_1 \ot \Pi_{-1}\r)_0 \op L_1^{\ot 2} \ot\l( \Pi_1 \ot \Pi_{-1}\r)_0, \nonumber
 \end{align}
 where $(\Pi_{-1})^* \cong \Pi_1$ via the Killing form, so the ``$0$'' subscripts indicate trace-free with respect to contractions.
 For the tensor descriptions above, we used
 \begin{align*}
 L_{-1} &= \pathwts{xxwwww}{2,-1,0,0,0,0}, \qquad
 L_1 = \pathwts{xxwwww}{-2,1,0,0,0,0}, \\
 \Pi_{-1} &= \pathwts{xxwwww}{-1,1,0,0,0,1}, \qquad
 \Pi_{1} = \pathwts{xxwwww}{1,-2,1,0,0,0}. \nonumber
 \end{align*}
 (e.g. The lowest roots of $\Pi_1$ and $\Pi_{-1}$ are $\alpha_2$ and $-\alpha_2 - ... - \alpha_\rkg$, so by the ``minus lowest weight convention'', we convert their negatives into weight notation using the Cartan matrix.)
 
The tensors $S$ and $T$ correspond to $(\Kh)_{+3}$ and $(\Kh)_{+2}$ respectively.  From \S \ref{S:corr-twistor}, we have:
 \begin{itemize}
 \item $S \equiv 0$ iff $(\Kh)_{+3} \equiv 0$ iff the geometry admits an $A_\rkg / P_1$ description, i.e.\ \eqref{E:2-ODE} are the equations for (unparametrized) geodesics of a projective connection, cf. \S \ref{S:projective}.
 \item $T \equiv 0$ iff $(\Kh)_{+2} \equiv 0$ iff the geometry admits an $A_\rkg / P_2 \cong \tGr(2,\bbR^{\rkg+1})$ description.  This is a 1-graded geometry with underlying structure a manifold $N^{2m}$, endowed with a field $\cV \subset \bbP(TN)$ of type $(2,m)$-Segr\'e varieties.  This is a {\em Segr\'e} (or {\em almost Grassmannian}) {\em structure}.\footnote{When $m=2$, this is equivalent to a signature $(2,2)$ conformal structure.  Note $A_3 / P_2 \cong D_3 / P_1$.}
 \end{itemize}
 In \cite{Gro2000}, Grossman studies {\em torsion-free} path geometries, i.e.  \eqref{E:2-ODE} satisfying $T \equiv 0$.  If the geometry is also ``geodesic'', i.e.\ $S\equiv 0$, then $\Kh \equiv 0$, and so the geometry is flat, cf. \cite[Thm.\ 1]{Gro2000}.

 \begin{remark} \label{RM:A3P12}
 For $m = 2$, i.e.\ $A_3 / P_{1,2}$, the underlying structure is a rank 3 distribution $D$ on a 5-manifold with a splitting $D = L \op \Pi$, but now $\Pi$ (rank 2) may not be integrable; instead $[\Pi,\Pi] \subseteq D$ \cite{CS2009,BEGN2012}.  Indeed, $w = (23) \in W^\fp_+(2)$ generates $\Athree{xxw}{4,-4,0} = \bigwedge^2 \Pi_1 \ot L_{-1}$ (homogeneity +1).   For pairs of 2nd order ODE, the same discussion above holds provided $(\Kh)_{+1} \equiv 0$.  Having nowhere vanishing $(\Kh)_{+1}$ is equivalent to $[\Pi,\Pi] = D$, so $\Pi$ becomes a $(2,3,5)$-distribution (\S \ref{S:G2P1}).  In this case, our results establish 9 as an upper bound on the symmetry dimension.  On a 5-manifold $(x,y,p,q,z)$, this is realized by $\Pi = \{ \p_q, \p_x + p \p_y + q\p_p + q^2 \p_z \}$ (i.e.\ the flat model for $(2,3,5)$-distributions) and $L = \{ \p_p + 2q\p_z \}$.   Abstractly, on $\Lie(G_2) / \fp_1$, $\Pi = \{ e_{-\alpha_1}, e_{-\alpha_1 - \alpha_2} \}$ and $L = \{ e_{-2\alpha_1 - \alpha_2} \}$ (mod $\fp_1$).  Only $\fgl_2(\bbC) \cong \fg_0 \subset \fp_1$ stabilizes both $\Pi$ and $L$.  Hence, $\cS \cong \fp_1^{\opn}$.
 \end{remark}
 
 When $m \geq 2$, $\fS_{(21)} = m^2+5$ and $\fS_{(12)} = m^2 + 4$, e.g. for $w = (21)$, \eqref{E:H2-ODE} yields by Recipe \ref{R:red-geom},
 \begin{align*}
 J_w &= \{ 3, m+1 \}, \quad I_w = \{ 1 \}, \quad \bar\fg / \bar\fp \cong A_1 / P_1 \qRa \dim(\fa_+) = 1 \quad\mbox{[NPR]},\\
  \fp_w^{\opn} &\cong \fp_{1,m-1} \subset A_{m-1} \qRa \dim(\fa_0) = 1 + \dim(\fp_w^{\opn}) = m^2 - 2m + 3,\\
  \fS_w &= \fU_w = \dim(\fa(w)) = 2m + 1 + (m^2 - 2m + 3) + 1 = m^2 + 5.
  \end{align*}
 
 \begin{prop}
 Any submaximally symmetric 2nd order ODE system is torsion-free (not geodesic).  A model, along with its symmetries, is:
 \begin{align} \label{E:2-ODE-sm}
 \begin{array}{c|c}
 \l\{ \begin{array}{l@{\,}c@{\,\,}ll}
 \ddot{x}^1 &=& 0, \\
 &\vdots\\
 \ddot{x}^{m-1} &=& 0, \\
 \ddot{x}^m &=& (\dot{x}^1)^3
 \end{array} \r. & 
 \begin{array}{lll}
  \bT = \p_t, \quad \bX_i = \p_{x^i}, \quad  \bA_j = t\p_{x^j}, \quad
 \bB_j^k = x^k \p_{x^j},\\
 \qquad (1 \leq i \leq m, \quad 2 \leq j \leq m, \quad 1 \leq k < m)\\
 \bC = 2 t\p_{x^1} + 3 (x^1)^2\p_{x^m}, \quad
 \bD_1 = x^1 \p_{x^1} + 3 x^m\p_{x^m}, \\
 \bD_2 = t \p_t - x^m \p_{x^m}, \quad
 \bS = t^2 \p_t + t \sum_{i=1}^m x^i\p_{x^i}+\half (x^1)^3 \p_{x^m}.
 \end{array}\\
 \end{array}
 \end{align}
 \end{prop}

 The symmetries are stated as vector fields $\bV = \xi \p_t + \varphi^i \p_{x^i}$ on $J^0$, but these admit unique prolongations $\bV^{(1)} = \bV + \varphi^i_{(1)} \p_{p^i}$ on $J^1$ and $\bV^{(2)} = \bV^{(1)} + \varphi^i_{(2)} \p_{q^i}$ on $J^2$ via the formula
 \begin{align} \label{E:prolongation}
 \varphi^i_{(1)} = \frac{d\varphi^i}{dt} - p^i \frac{d\xi}{dt}, \qquad
 \varphi^i_{(2)} = \frac{d\varphi^i_{(1)}}{dt} - q^i \frac{d\xi}{dt}.
 \end{align}
 Externally, $\bV^{(2)}$ are everywhere tangent to $M \subset J^2$.  Internally, $\bV^{(1)}$ are identified with vector fields on $M$ which preserve the geometric data described earlier.
 
  \begin{remark}
 The model \eqref{E:2-ODE-sm} was found in the $m=2$ case by Casey et al. \cite[(3.5)]{CDT2013}.
 \end{remark} 

 For the scalar case $\ddot{x} = f(t,x,\dot{x})$, i.e.\ $m=1$ and $A_2 / P_{1,2}$, we have $H^2_+(\fg_-,\fg) = \Atwo{xx}{-5,1} \op \Atwo{xx}{1,-5}$.  The two 1-dimensional irreps correspond to the Tresse relative invariants \cite{Tresse1896}:
 \begin{align} \label{E:Tresse}
 I_1 = \l( \frac{d}{dt} \r)^2 (f_{pp}) - f_p \frac{d}{dt} (f_{pp}) - 4 \frac{d}{dt} (f_{xp}) + 4 f_p f_{xp} - 3 f_x f_{pp} + 6 f_{xx}, \qquad I_2 = f_{pppp},
 \end{align}
 where $p = \dot{x}$, and $\frac{d}{dt} = \p_t + p \p_x + f \p_p$.  From \S \ref{S:exceptions}, $\fS \leq 3$, and $\fS = 3$ \cite{Tresse1896} is established by:
 \begin{align} \label{E:scalar-2-ODE}
 \begin{array}{l|l|c}
 \ddot{x} = \displaystyle\frac{\dot{x} - (\dot{x})^3 + a(1-(\dot{x})^2)^{3/2}}{t}, \quad a \neq 0 &
 \begin{array}{l} \bX_1 = \p_x, \quad \bX_2 =t\p_t+x\p_x, \\ \bX_3 = 2tx \p_t + (t^2 + x^2)\p_x \end{array} & I_1 I_2 \neq 0\\ \hline
 \ddot{x} = \displaystyle\frac{\varepsilon (\dot{x})^3 - \half \dot{x}}{t}, \quad \varepsilon = \pm 1 & \begin{array}{l} \bX_1 = \p_x, \quad \bX_2 = t\p_t + x\p_x,\\ \bX_3 = 2tx\p_t + x^2 \p_x \end{array} & I_2 = 0
  \end{array}
 \end{align}

 
 \subsection{Projective structures}
 \label{S:projective}
 
 A projective connection $[\nabla]$ is an equivalence class of torsion-free affine connections, where $\nabla \sim \widehat\nabla$ iff their unparametrized geodesics are the same.  These structures underlie an $A_\rkg / P_1$ geometry, or equivalently an $A_\rkg / P_{1,2}$ geometry with $S\equiv 0$ (see \S \ref{S:2-ODE}).

  Locally, on an $\rkg$-manifold $N$ with coordinates $(x^a)_{a=0}^{\rkg-1}$, define $\nabla$ via its Christoffel symbols
\begin{align*}
 \nabla_{\p_{x^a}} \p_{x^b} = \Gamma^c_{ab} \p_{x^c}, \qquad \Gamma^c_{ab} = \Gamma^c_{(ab)}, \qquad 0 \leq a,b,c \leq \rkg-1 =: m.
\end{align*}
 A curve $\sigma(t)$ has geodesic equation $(x^a)'' + \Gamma^a_{bc} (x^b)' (x^c)' = 0$.  Locally, we may assume $(x^0)' \neq 0$ and solve for $t = t(x^0)$.  For $i > 0$, regard $x^i$ as functions of $x^0$, so the geodesic equations become
 \begin{align}
 \ddot{x}^i &= \Gamma^0_{ab} \dot{x}^i \dot{x}^a \dot{x}^b - \Gamma^i_{ab} \dot{x}^a \dot{x}^b, \label{E:proj-ODE}\\
 &= \Gamma^0_{jk} \dot{x}^i \dot{x}^j \dot{x}^k + 2\Gamma^0_{0j} \dot{x}^i \dot{x}^j -\Gamma^i_{jk} \dot{x}^j \dot{x}^k + \Gamma^0_{00} \dot{x}^i  - 2\Gamma^i_{0j} \dot{x}^j -\Gamma^i_{00},\qquad 1 \leq i,j,k \leq m. \nonumber
 \end{align}
 Projective equivalence $\nabla \mapsto \widehat\nabla$ is given operationally by
$\Gamma^a_{bc}\mapsto\widehat\Gamma^a_{bc} := \Gamma^a_{bc}+\delta^a_b\Upsilon_c+\delta^a_c\Upsilon_b$, where $\Upsilon$ is an arbitrary 1-form, and \eqref{E:proj-ODE} is invariant under these changes.

 Now $A_\rkg / P_1$ gives $\fg = \fg_{-1} \op \fg_0 \op \fg_1$, $W^\fp_+(2) = \{ (12) \}$, and for $\rkg \geq 3$,
 \begin{align*}
 & H^2_+(\fg_-,\fg) = \projwts{xsswws}{-4,1,1,0,0,1} = \l( \bigwedge{\!}^2 \fg_1 \ot \fsl(\fg_{-1}) \r)_0, \quad J_w = \{ 2, 3, \rkg \}, \quad I_w = \emptyset,\\
 & \fp_w^{\opn} \cong \fp_{1,2,\rkg-1} \subset A_{\rkg-1}, \quad \dim(\fa_0) = \rkg^2 - 3\rkg + 5, \quad \fU_w = \dim(\fg_{-1}) + \dim(\fa_0) =  (\rkg-1)^2 + 4,
  \end{align*}
 which recovers Egorov's result $\fS = (\rkg-1)^2 + 4$ for $\rkg \geq 3$ in \cite{Ego1951}.  Egorov also gave the model:
 \begin{align} \label{E:Egorov-model}
 \Gamma^0_{12} = \Gamma^0_{21} = x^1, \quad
 \Gamma^c_{ab} = 0 \mbox{ otherwise}
 \quad\leftrightarrow\quad \ddot{x}^i = 2 x^1 \dot{x}^1 \dot{x}^2 \dot{x}^i, \quad 1 \leq i \leq m,
 \end{align}
 which has the following $m^2 + 4$ point symmetries (writing $t := x^0$):
 \begin{align} \label{E:Egorov-syms}
\l\{ \begin{array}{ll}
\p_t,\quad \p_{x^2},\ \dots,\ \p_{x^m},\quad x^i\p_t,\quad x^i\p_{x^j}\quad (i\ge1,\ j\ge3),\\
2t\p_t+x^1\p_{x^1},\quad t\p_t+x^2\p_{x^2},\quad x^1x^2\p_t - \p_{x^1},\quad (x^1)^3\p_t - 3x^1\p_{x^2}.
\end{array} \r.
 \end{align}
 On $(t,x^j,p^j)$-space $M$, the vertical subspace for the projection $M \to N$ is spanned by $\{ \p_{p^i} \}_{i=1}^m$, so \eqref{E:Egorov-syms} are also the symmetries of $[\nabla]$ determined by \eqref{E:Egorov-model}.  The Fels torsion $T = (T^i_j)$ of \eqref{E:Egorov-model} has matrix rank 1 with nontrivial components $T^i_1=-p^1p^2 p^i$ and $T^i_2=(p^1)^2 p^i$.

 When $m=1$, i.e. $A_2 / P_1$, we have $\fS \leq 3$ from \S \ref{S:exceptions}.  Writing $(x^0,x^1) = (x,y)$, $\varepsilon = \pm 1$, the model
 \begin{align} \label{E:2d-proj}
 \Gamma^0_{00} = -\frac{1}{2x}, \quad \Gamma^0_{11} = \frac{\varepsilon}{x}, \quad  \Gamma^c_{ab} = 0 \mbox{ otherwise}, \qquad \begin{array}{l} \bX_1 = \p_x, \quad \bX_2 = t\p_t + x\p_x,\\ \bX_3 = 2tx\p_t + x^2 \p_x \end{array}
 \end{align}
 confirms $\fS = 3$ \cite{Tresse1896}.  Indeed, the second model in \eqref{E:scalar-2-ODE} is the geodesic equation of $[\nabla]$.

 
 \subsection{$(2,m)$-Segr\'e structures} A {\em Segr\'e} (or {\em almost Grassmannian}) structure is a $k(\rkg+1-k)$-manifold $N$ with a field $\cV \subset \bbP(TN)$ of $(k,\rkg+1-k)$-Segr\'e varieties.  These arise as underlying structures of $A_\rkg / P_k \cong \tGr(k,\bbR^{\rkg+1})$ geometries (for $k \neq 1,\rkg$).  We focus on the $(k,\rkg) = (2,m+1)$ case. The structure is encoded by a $2 \times m$ matrix of 1-forms $\Theta = \{ \Theta^A_i \}$ on $N^{2m}$: its $2\times 2$ minors 
 \begin{align} \label{E:Segre-eqns}
 \Theta^A_i \Theta^B_j - \Theta^A_j \Theta^B_i = 0, \quad 1 \leq i, j \leq m, \quad 1 \leq A,B \leq 2,
 \end{align}
 cut out the image $\cV|_z \subset \bbP(T_z N)$ of the Segr\'e embedding $\sigma : \bbP^1 \times \bbP^{m-1} \inj \bbP^{2m-1} \cong \bbP(T_z N)$, $[u] \times [v] \mapsto [u\ot v]$.  We have $W^\fp_+(2) = \{ (21), (23) \}$, and for $m \geq 2$,
 \[
 \fS_{(21)} = m^2 + 5, \qquad \fS_{(23)} = \l\{ \begin{array}{cl} m^2 - m + 8, & m \geq 3;\\ 9, & m=2. \end{array} \r.
 \]
 Thus, $\fS = \fS_{(21)}$, but when $m=2$ or 3, we also have $\fS = \fS_{(23)}$.
 
 The $(21)$-branch for 2nd order ODE systems corresponds to the $(21)$-branch for Segr\'e structures.  We exhibit a Segr\'e structure in this branch.  The twistor space $N$ of the ODE \eqref{E:2-ODE-sm} on $M$ is its solution space, i.e. the quotient by the integral curves of $\frac{d}{dt}$.  Any solution $c(t)$ of \eqref{E:2-ODE-sm} is given by
 \begin{align*}
 x^i = a_i t + b_i + \half (a_1)^3 t^2 \delta_{im}, \qquad p^i = a_i + (a_1)^3 t\, \delta_{im},
 \end{align*}
 where $z = (a_1,...,a_m, b_1,..., b_m)$ are parameters (local coordinates on $N$).  The map $\Psi : M \to N$ is
 \begin{align*}
 a_i = p^i - (p^1)^3 t\, \delta_{im}, \qquad b_i = x^i - p^i t + \half (p^1)^3 t^2 \delta_{im}.
 \end{align*}
 This induces a curve $\zeta(t) = \Psi_*(\Pi|_{c(t)})$ in $\tGr(m,T_z N)$ with $\zeta(t)$ spanned by:
 \begin{align} \label{E:Push-Pi}
 \bZ_i(t) = \Psi_*(\p_{p^i}|_{c(t)}) = \p_{a_i} - t \p_{b_i} 
 + \delta_{i1} \left( - 3 (a_1)^2 t \p_{a_m} + \frac{3}{2} (a_1)^2 t^2 \p_{b_m} \right).
 \end{align} 
 This curve satisfies \eqref{E:Segre-eqns}, where $\Theta = \{ \Theta^A_i \}$ is given by
 \begin{align} \label{E:Segre-sm}
 \Theta^1_i = da_i - \frac{3}{2} (a_1)^2 \delta_{im} db_1, \qquad \Theta^2_i = db_i.
 \end{align}
 Thus, $\zeta(t) \subset \cV|_z$ as $t$ varies (and moreover fills it when regarding $t$ as a projective parameter).
 
 Given any symmetry $\bV$ of \eqref{E:2-ODE-sm}, $\widetilde\bV = \Psi_*(\bV^{(1)})$ is a symmetry of the Segr\'e structure \eqref{E:Segre-sm}, i.e.\ the ideal generated by \eqref{E:Segre-eqns} is preserved under $\cL_{\widetilde\bV}$.  We obtain:
 \begin{align} \label{E:Segre-syms}
 \left\{ \begin{array}{l}
 \widetilde\bT = -(a_1)^3 \p_{a_m} - \sum_{i=1}^m a_i \p_{b_i}, \quad \widetilde\bX_i = \p_{b_i}, \quad
 \widetilde\bA_j = \p_{a_j}, \quad \widetilde\bB_j^k = a_k \p_{a_j} + b_k \p_{b_j}, \\
  \qquad\qquad\qquad (1 \leq i \leq m, \quad 2 \leq j \leq m, \quad 1 \leq k < m) \\
 \widetilde\bC = 2\p_{a_1} + 6 a_1 b_1 \p_{a_m} + 3(b_1)^2 \p_{b_m}, \quad
 \widetilde\bD_1 = a_1 \p_{a_1} + 3 a_m \p_{a_m} + b_1 \p_{b_1} + 3 b_m \p_{b_m},\\
 \widetilde\bD_2 = -\sum_{k=1}^{m-1} a_k \p_{a_k} - 2a_m \p_{a_m} - b_m \p_{b_m},\quad
 \widetilde\bS = \sum_{i=1}^m b_i \p_{a_i} + \frac{3}{2} a_1 (b_1)^2 \p_{a_m} + \half (b_1)^3 \p_{b_m}.
 \end{array} \right.
 \end{align}
 When $m=2$, $\metric := \Theta^1_1 \Theta^2_2 - \Theta^1_2 \Theta^2_1$ is (after changing coordinates) the $(2,2)$ pp-wave \eqref{E:pp-2,2}.  Thus,
 
 \begin{prop}
 If $m \neq 3$, any submaximally symmetric $(2,m)$-Segr\'e structure is necessarily of ODE type, i.e.\ in the $(21)$-branch (possibly after using duality in the $m=2$ case).  For $m \geq 2$, a submaximally symmetric model $\Theta = \{ \Theta^A_i \}$ is given by \eqref{E:Segre-sm} with symmetries \eqref{E:Segre-syms}.
 \end{prop}

 Any Segr\'e variety admits two {\em rulings}: writing $\sigma_u(v) = \sigma(u,v) = \sigma^v(u)$, we have for \eqref{E:Segre-sm}:
 \begin{itemize}
 \item $m$-ruling: $\sigma_u : \bbP^{m-1} \to \bbP^{2m-1}$ for $[u] \in \bbP^1$.  Each $m$-plane in $\cR|_z \subset \tGr(m,T_z N)$ has basis
 \begin{align*}
 u_1\p_{a_i} + u_2 \left( \p_{b_i} + \delta_{i1} \frac{3}{2} (a_1)^2 \p_{a_m} \right), \quad 1 \leq i \leq m.
 \end{align*}
 \item $2$-ruling: $\sigma^v : \bbP^1 \to \bbP^{2m-1}$ for $[v] \in \bbP^{m-1}$. Each $2$-plane in $\cP|_z \subset \tGr(2,T_z N)$ has basis
 \begin{align*}
 v_1 \p_{a_1} + ... + v_m \p_{a_m}, \quad v_1 \l( \frac{3}{2} (a_1)^2 \p_{a_m} + \p_{b_1}\r) + v_2 \p_{b_2} ... + v_m \p_{b_m}.
 \end{align*}
 \end{itemize}
 This gives submanifolds $\cR \subset \tGr(m,TN)$ and $\cP \subset \tGr(2,TN)$.  From \eqref{E:Push-Pi}, we see that $\Psi_*(\Pi) \subset \cR$.  The image under $\Psi$ of each fibre of $M \to J^0$ is an $m$-dimensional submanifold $\Sigma \subset N$, with $T_z \Sigma \in \cR|_z \subset T_z N$ for any $z \in \Sigma$.  Conversely, every element of $\cR$ is tangent to such a submanifold, so $\cR$ is {\em integrable}.  In fact, Grossman \cite{Gro2000} shows that integrability of the $m$-ruling $\cR$ is a general feature of Segr\'e structures arising from 2nd order ODE systems by establishing an isomorphism of the bundles $M \to N$ and the $\bbP^1$-bundle $P_\Delta \to N$ whose fibres parametrize the $m$-ruling.   The $2$-ruling $\cP$ is never integrable for such structures, except for the flat model.  For Segr\'e structures in the $(23)$-branch, $\cR$ is non-integrable while $\cP$ is integrable.
 
 \newpage
 \appendix
  
 \section{Yamaguchi's prolongation and rigidity theorems}
 \label{App:Yam}

 Let $\fg$ be a complex simple Lie algebra and $\fp$ a parabolic subalgebra.  In \cite{Yam1993}, Yamaguchi proved:\footnote{The rigidity list in \cite{Yam1993} contains some minor errors, which were corrected in \cite{Yam1997}.}
 
  \begin{appthm}[Prolongation theorem] \label{T:Y-pr-thm} We have $\fg \cong \prn(\fg_-)$ except for:
 \begin{enumerate}[(a)]
 \item 1-gradings: \quad $A_\rkg / P_k, \quad B_\rkg / P_1, \quad C_\rkg / P_\rkg, \quad 
  D_\rkg / P_1, \quad D_\rkg / P_\rkg,\quad E_6 / P_6, \quad E_7 / P_7$.
 \item contact gradings, i.e.\ 2-gradings with $\dim(\fg_{-2}) = 1$, and non-degenerate bracket $\bigwedge^2 \fg_{-1} \to \fg_{-2}:$
 \[
 A_\rkg / P_{1,\rkg}, \quad B_\rkg / P_2, \quad C_\rkg / P_1, \quad D_\rkg / P_2, \quad G_2 / P_2, \quad F_4 / P_1, \quad E_6 / P_2, \quad E_7 / P_1, \quad E_8 / P_8.
 \]
 \item $(\fg,\fp) \cong (A_\rkg, P_{1, i})$ with $\rkg \geq 3$ and $i \neq 1, \rkg$, or $(C_\rkg, P_{1,\rkg})$ with $\rkg \geq 2$.
 \end{enumerate}
 Moreover, we always have $\fg \cong \prn(\fg_-,\fg_0)$ except when $(\fg,\fp) \cong (A_\rkg,P_1)$ or $(C_\rkg,P_1)$.
 \end{appthm}

 \begin{appthm}[Rigidity theorem] \label{T:Y-rigid} We have $H^2_+(\fg_-,\fg) \neq 0$ if and only if $(\fg,\fp)$ is: (i) 1-graded, (ii) a contact gradation, or (iii) listed in Table \ref{table:non-rigid}.
 \end{appthm}
 
  \begin{table}[h]
 $\begin{array}{|l|c|c|c|c|c|c|} \hline
 G & \mbox{Range} & \mbox{2-graded} & \mbox{3-graded} & \mbox{4-graded} & \mbox{5-graded} & \mbox{6-graded} \\ \hline\hline
 A_\rkg & \rkg \geq 4 & P_{1,s}, P_{2,s}, P_{s,s+1} & P_{1,2,s}, P_{1,s,\rkg} & P_{1,2,s,t} & - & -\\
 & \rkg = 3 & P_{1,2} & P_{1,2,3} & - & - & -\\ \hline
 B_\rkg & \rkg \geq 4 & P_3, P_\rkg & P_{1,2} & P_{2,3} & - & -\\ 
	& \rkg = 3 & P_3 & P_{1,2}, P_{1,3} & P_{2,3} & P_{1,2,3} & -\\
	& \rkg = 2 & - & P_{1,2} & - & - & - \\ \hline
 C_\rkg & \rkg \geq 4 & P_2, P_{\rkg-1} & P_{1,\rkg}, P_{2,\rkg}, P_{\rkg-1,\rkg} & P_{1,2} & P_{1,2,\rkg} & P_{1,2,s}\, (s < \rkg)\\
 	& \rkg = 3 & P_2 & P_{1,3}, P_{2,3} & P_{1,2} & P_{1,2,3} & -\\ \hline
 D_\rkg & \rkg \geq 5 & P_3, P_{1,\rkg} & P_{1,2} & P_{2,3}, P_{1,2,\rkg} & - & -\\
	& \rkg = 4 & P_{1,4} & P_{1,2} & P_{1,2,4} & - & - \\ \hline
 G_2 & -& - & P_1 & - & P_{1,2} & -\\ \hline
 \end{array}$
 \caption{Yamaguchi-nonrigid geometries, excluding 1-graded and parabolic contact geometries}
 \label{table:non-rigid}
 \end{table}


 \section{Dynkin diagram recipes}

 \label{App:Dynkin}
 
 We use the notations of \S \ref{S:rep-th}. The following Dynkin diagram recipes are well-known \cite{BE1989,CS2009}.
 
 \begin{recipe} \label{R:dim} 
 Deleting the $I_\fp$ nodes from $\cD(\fg,\fp)$ yields $\cD(\fg_0^{ss})$, and $\dim(\fz(\fg_0)) = |I_\fp|$.  Also, 
 \[
 \dim(\fg_-) = \half(\dim(\fg) - \dim(\fg_0)), \qquad \dim(\fp) = \half(\dim(\fg) + \dim(\fg_0)).
 \]
 \end{recipe}
 
 \begin{recipe} \label{R:reflection} Let $\lambda \in \fh^*$ be the weight with coefficient $r_i = \langle \lambda, \alpha_j^\vee \rangle$ inscribed on the $i$-th node of $\cD(\fg)$.  To compute $\sr_j(\lambda)$, add $r_j$ to adjacent coefficients, with multiplicity if there is a multiple edge directed towards the adjacent node; then replace $r_j$ by $-r_j$.
 \end{recipe}

 \begin{recipe} \label{R:Hasse} $w = (jk) \in W^\fp(2)$ iff $j \in I_\fp$ and $j \neq k \in I_\fp \cup \cN(j)$, where $\cN(j) = \{ i \mid c_{ij} \leq -1 \}$.
 \end{recipe}
 
 By the ``minus lowest weight'' convention, the $\fg_0$-irrep $\bbV_\mu$ with lowest weight $\mu$ is denoted by the Dynkin diagram notation for $-\mu$.  Below is Kostant's theorem for $H^2(\fg_-,\bbU)$.

  \begin{recipe} \label{R:Kostant} Let $\bbU$ be a $\fg$-irrep with minus lowest weight $\lambda$.  Then
 $H^2(\fg_-,\bbU) \cong \bop_{w \in W^\fp(2)} \bbV_{-w\cdot \lambda}$
 as $\fg_0$-modules.
 Let $w = (jk) \in W^\fp(2)$, so $\Phi_w = \{ \alpha_j, \sr_j(\alpha_k) \}$.  Via the isomorphism $H^2(\fg_-,\bbU) \cong \ker(\Box) \subset \bigwedge^2 \fg_-^* \ot \bbU$, the $\fg_0$-module $\bbV_{-w\cdot \lambda}$ has the unique (up to scale) lowest weight vector
 \begin{align} \label{E:phi-0}
 \phi_0 := e_{\alpha_j} \wedge e_{\sr_j(\alpha_k)} \ot v,
 \end{align}
 where $e_\gamma \in \fg_\gamma$ are root vectors, and $v \in \bbU$ is a weight vector with weight $w(-\lambda)$.
 \end{recipe}
 
 Let $-\mu$ be $\fp$-dominant, and define $J_\mu := \{ j \not\in I_\fp \mid \langle \mu, \alpha_j^\vee \rangle \neq 0 \}$.
 
 \begin{recipe} \label{R:a0} Let $\phi_0 \in \bbV_\mu$ a lowest weight vector, $\fa_0 = \fann(\phi_0)$, $\fp_\mu^{\opn} := (\fg_0^{ss})_{\leq 0}$ be the (opposite) parabolic in $\fg_0^{ss}$ in the $Z_{J_\mu}$ grading.
   Then $ \dim( \fa_0) = \max\{ \dim(\fann(\phi) ) \mid 0 \neq \phi \in \bbV_{\mu} \}$, with:
 \begin{align}
 \fa_0 & = \{ H \in \fh \mid \mu(H) = 0 \} \op \bop_{\gamma \in \Delta(\fg_{0,\leq0})} \fg_\gamma, \label{E:a0}\\
 \dim( \fa_0) &= \dim(\fp_\mu^{\opn}) + |I_\fp| - 1. \label{E:dim-a0}
 \end{align}
 \end{recipe}
 
  The notations and recipes below are new (see \S \ref{S:PR-analysis}). Define $I_\mu := \{ j \in I_\fp \mid \langle \mu, \alpha_j^\vee \rangle = 0 \}$.
 \begin{itemize}
 \item  Denote by $\cD(\fg,\fp,\mu)$ the Dynkin diagram notation for $-\mu$ marked with an {\em asterisk} on uncrossed nodes with a nonzero coefficient, and a {\em square} on crossed nodes with a zero coefficient.
 \item Let $\fa(\mu) := \prn^\fg(\fg_-,\fann(\phi_0))$, where $\phi_0$ is a lowest weight vector in $\bbV_\mu$.
 \item If $\fg$ is simple, $w \in W^\fp(2)$, and $\mu = -w\cdot \lambda_\fg$, write $J_w := J_\mu$, $I_w := I_\mu$, $\fa(w) = \fa(\mu)$, etc.
 \end{itemize}
 
 \begin{recipe} \label{R:squares} $(\fg,\fp,\mu)$ is PR iff $I_\mu = \emptyset$ iff $\cD(\fg,\fp,\mu)$ has no squares.
 \end{recipe}

 \begin{recipe} \label{R:red-geom} If $I_\mu \neq \emptyset$, then on $\cD(\fg,\fp,\mu)$, remove all nodes corresponding to $I_\fp \backslash I_\mu$ and $J_\mu$, and any edges connected to these nodes.  In the resulting diagram, remove any connected components which do not contain an $I_\mu$ node.  This is $\cD(\overline\fg,\overline\fp)$ for the reduced geometry $(\overline\fg, \overline\fp)$, where $\overline\fp$ corresponds to crosses on the $I_\mu$ nodes, and $\dim(\fa_k(\mu)) = \dim(\overline\fg_k)$, $\forall k > 0$.  Combine with Recipes \ref{R:dim} and \ref{R:a0} to compute $\fU_\mu = \dim(\fa(\mu)) = \dim(\fg_-) + \dim(\fa_0) + \dim(\fa_+)$.
 \end{recipe}

 \begin{ex}[$G_2 / P_1$] \label{EX:G2P1} The highest weight of $\fg = \Lie(G_2)$ is $\lambda_\fg = \lambda_2 = 3\alpha_1 + 2\alpha_2 = \Gdd{ww}{0,1}$.  For $G_2 / P_1$, i.e.\ $\Gdd{xw}{}$, the grading element is $Z = Z_1$, $\fg$ is 3-graded, $\fz(\fg_0) = \tspan\{ Z_1 \}$, $\fg_0^{ss} \cong \fsl_2(\bbC)$, and $W^\fp(2) = \{ (12) \}$.  Given $w = (12)$, we compute $w \cdot \lambda_\fg = w(\lambda_\fg + \rho) - \rho$ using Recipe \ref{R:reflection}:
 \[
  w \cdot \lambda_\fg = (12) (\Gdd{xw}{1,2}) - \rho = (1) (\Gdd{xw}{7,-2}) - \rho= \Gdd{xw}{-8,4} = -8\lambda_1 + 4\lambda_2 = -4\alpha_1,
 \]
 so by the minus lowest weight convention, $\bbV_{-w\cdot\lambda_\fg} = \Gdd{xw}{-8,4}$ has homogeneity $+4$, and $W^\fp(2) = W^\fp_+(2)$.  Since the lowest root of $\fg_1$ is $\alpha_1 = 2\lambda_1 - \lambda_2$, then as a $\fg_0 \cong \fgl_2(\bbC)$ module,
 \[
 H^2(\fg_-,\fg) = \bbV_{-w\cdot \lambda_\fg} = \Gdd{xw}{-8,4} = \bigodot{}^4( \Gdd{xw}{-2,1} ) = \bigodot{}^4(\fg_1) = \bigodot{}^4(\fg_{-1})^* = \bigodot{}^4 (\bbC^2)^*.
 \]
 This recovers Cartan's result \cite{Car1910} that the fundamental (harmonic) curvature tensor for $(2,3,5)$-geometries is a binary quartic field defined on the 2-plane distribution.  Furthermore,
 \[
 \Phi_w = \{ \alpha_1, 3\alpha_1 + \alpha_2 \}, \qquad w(-\lambda_\fg) = 3\lambda_1 - 2\lambda_2 = -\alpha_2.
 \]
 Thus, $H^2(\fg_-,\fg) \cong \ker(\Box) \subset \bigwedge^2 \fg_-^* \ot \fg$ has {\em lowest} weight vector $\phi_0 = e_{\alpha_1} \wedge e_{3\alpha_1 + \alpha_2} \ot e_{-\alpha_2}$.

For $\Gdd{xs}{-8,4}$, $J_w = \{ 2 \}$, $I_w = \emptyset$, i.e.\ no squares.  Thus,  $G_2 / P_1$ is PR, so $\fa(w) = \fg_- \op \fa_0$.  Hence, $\fp_w^{\opn} \cong \fp_1 \subset A_1$, so $\dim(\fa_0) = \dim(\fp_w^{\opn}) = 2$.  Thus, $\fS = \dim(\fa(w)) = \dim(\fg_-) + \dim(\fa_0) = 5 + 2 = 7$.

 \end{ex}
  
  See the examples in \S \ref{S:Dynkin} for applications of Recipe \ref{R:red-geom}.
  
 
 \section{Submaximal symmetry dimensions}
 \label{App:Submax}

  The tables to follow summarize data associated to the gap problem for all {\em regular, normal} parabolic geometries of {\em complex} or {\em split-real} type $(G,P)$ with $G$ {\em simple}.  Let $\fS$ be the submaximal symmetry dimension, and define $\fS_w$ similarly but require $\im(\Kh) \subset \bbV_{-w\cdot\lambda_\fg}$.  Then $\fS = \max_{w \in W^\fp_+(2)} \fS_w$.
For each $w \in W^\fp_+(2)$, compute $\fU_w = \dim(\fa(w))$ by Recipe \ref{R:red-geom}.  Other than the exceptional $A_2 / P$ and $B_2 / P$ case (Table \ref{F:A2B2}), we always have $\fS_w = \fU_w$, c.f. Theorem \ref{T:main-thm}.  Also, by Remark \ref{E:coverings}, we may simplify the calculation of $\fS_w$ by always passing to the minimal twistor space.
 
 Note that $\fa_1 = 0$ in Tables \ref{F:A2B2}, \ref{F:1-graded}, \ref{F:contact}.

 \begin{tiny}
 \begin{table}[h]
 $\begin{array}{|c|c|c|c|c|c|c|c|c|c|c|} \hline
 G & P & \dim(G/P) & w & J_w & \begin{tabular}{c}Twistor\\space type\end{tabular} & \fS_w & \fU_w\\ \hline\hline
 A_2 & P_1 & 2 & (12) & \{ 2 \} & - & 3 & 4\\
 	& P_{1,2} & 3 & (12) & \emptyset & A_2 / P_1 & 3 & 4\\
		&&& (21) & \emptyset & A_2 / P_2 & 3 & 4\\ \hline\hline
 B_2 & P_1 & 3& (12) & \{ 2 \} & - & 4 & 5\\
 	& P_2 & 3 & (21) & \{ 1 \} & - & 5 & 5\\
	& P_{1,2} & 4 & (12) & \emptyset & B_2 / P_1 & 4 & 5\\
		&&& (21) & \emptyset & B_2 / P_2 & 5 & 5\\
  \hline
 \end{array}$
 \caption{Submaximal symmetry dimensions for geometries of type $A_2/P$ and $B_2 / P$}
 \label{F:A2B2}
 \end{table}
 
 \begin{table}[h]
 $\begin{array}{|c|c|c|c|c|c|c|}\hline
 G & P &\mbox{Range} & \dim(G/P) & w & J_w & \fS_w = \fU_w\\ \hline\hline
 A_\rkg & P_1 & \rkg \geq 3 & \rkg & (12) & \{ 2,3,\rkg \} & (\rkg-1)^2 + 4\\
 	& P_2 & \rkg \geq 3 & 2(\rkg -1)  & (21) & \{ 3, \rkg \} & (\rkg-1)^2 + 5 \\
 	&&&&(23) & \mycase{\{ 1, 4, \rkg \}, & \rkg \geq 4;\\ \{ 1 \}, & \rkg = 3} & \mycase{\rkg^2 - 3\rkg + 10, & \rkg \geq 4;\\ 9, & \rkg = 3}\\ 
		& P_k & 3 \leq k \leq \lceil \frac{\rkg}{2}\rceil & k(\rkg+1-k)  & (k,k+1) & \{ 1, k-1, k+2, \rkg \} & \rkg(\rkg-1) - k(\rkg-k) + 6\\
 	&&&& (k,k-1) & \{ 1, k-2, k+1, \rkg \} & \rkg(\rkg - k) + (k-1)^2+6\\ \hline
 B_\rkg & P_1 & \rkg \geq 3 & 2\rkg-1  & (12) & \{ 3 \} & 2\rkg^2 - 5\rkg + 9\\ \hline
 C_\rkg & P_\rkg & \rkg \geq 3 & \binom{\rkg+1}{2} & (\rkg,\rkg-1) & \{ 1, \rkg - 2, \rkg - 1 \} & \frac{\rkg(3\rkg - 5)}{2} + 5\\[0.02in] \hline
 D_\rkg & P_1 & \rkg \geq 4 & 2\rkg-2 & (12) & \mycase{\{ 3 \}, & \rkg \geq 5; \\ \{ 3, 4 \}, & \rkg = 4} & 2\rkg^2 - 7\rkg + 12 \\
		& P_\rkg & \rkg \geq 5 & \binom{\rkg}{2} & (\rkg,\rkg-2) & \{ 2, \rkg-3, \rkg-1 \} & \frac{\rkg(3 \rkg - 11)}{2} + 16\\[0.02in] \hline
 E_6 & P_6 & - & 16 & (65) & \{ 2, 4 \} & 45\\
 E_7 & P_7 & - & 27 & (76) & \{ 1, 5 \} & 76\\ \hline
 \end{array}$
 \caption{Submaximal symmetry dimensions for 1-graded geometries}
 \label{F:1-graded}
 \end{table}

 \begin{table}[h]
 $\begin{array}{|c|c|c|c|c|c|c|c|}\hline
 G & P &\mbox{Range} & \dim(G/P) & w & J_w & \fS_w = \fU_w\\ \hline\hline
 A_\rkg & P_{1,\rkg} & \rkg \geq 3 & 2\rkg-1 & (1,\rkg) & \{ 2, \rkg-1 \} & (\rkg-1)^2 + 4\\
 &&&& (12) & \mycase{ \{ 2, 3 \}, & \rkg \geq 4; \\ \{ 2 \}, & \rkg = 3 } & (\rkg-1)^2 + 4\\
 &&&& (\rkg, \rkg-1) & \mycase{ \{ \rkg-2, \rkg-1 \}, & \rkg \geq 4; \\ \{ 2 \}, & \rkg = 3 } & (\rkg-1)^2 + 4\\ \hline
 B_\rkg & P_2 & \rkg \geq 3 & 4\rkg - 5 &(21)& \{ 1, 3 \} & 2\rkg^2 - 5\rkg + 8\\ 
 &&&&(23) & \mycase{\{ 1, 3, 4 \}, & \rkg \geq 4;\\ \{ 1, 3 \}, & \rkg = 3} & \mycase{2\rkg^2 - 7\rkg + 15, & \rkg \geq 4;\\ 11, & \rkg = 3}\\ \hline
 C_\rkg & P_1 & \rkg \geq 2 & 2\rkg-1 & (12) & \mycase{\{ 2, 3\}, & \rkg \geq 3;\\ \{ 2\}, & \rkg = 2} & \mycase{2\rkg^2 - 5\rkg + 8, & \rkg \geq 3;\\ 5, & \rkg = 2}\\ \hline
 D_\rkg & P_2 & \rkg \geq 4 & 4\rkg-7 &(21)& \mycase{\{ 1, 3 \}, & \rkg \geq 5;\\ \{ 1, 3, 4 \}, & \rkg = 4} & 2\rkg^2 - 7\rkg + 11\\
 &&&&(23) & \mycase{\{ 1, 3, 4 \}, & \rkg \neq 5;\\ \{ 1, 3, 4, 5\}, & \rkg = 5} & 2\rkg^2 - 9\rkg + 19\\ 
 &&\rkg = 4&&(24) & \{ 1, 3, 4 \} & 15\\ \hline
 G_2 & P_2 & - & 5 & (21) & \{ 1 \} & 7\\
 F_4 & P_1 & - & 15 & (12) & \{ 2 , 3 \} & 28\\
 E_6 & P_2 & - & 21 & (24) & \{ 3, 4, 5 \} & 43\\
 E_7 & P_1 & - & 33 & (13) & \{ 3, 4 \} & 76\\
 E_8 & P_8 & - & 57 & (87) & \{ 6, 7 \} & 147\\ \hline
 \end{array}$
 \caption{Submaximal symmetry dimensions for parabolic contact geometries}
 \label{F:contact}
 \end{table}
 \end{tiny}
 
 \begin{tiny}
 \begin{landscape}
 \begin{table}[h]
 $\begin{array}{|c|c|c|c|c|c|c|c|c|c|@{}c@{}|c|} \hline
 G & P & \mbox{Range} &\dim(G/P) & w & J_w & I_w & (\overline\fg,\overline\fp) & \fa_1 & \fa_2 & \begin{tabular}{c}Twistor\\space type\end{tabular} & \fS_w = \fU_w\\ \hline\hline
  A_\rkg & P_{1,2} & \rkg \geq 3 & 2\rkg - 1 & (21) & \{ 3, \rkg \} & \{ 1 \} & A_1 / P_1 & \checkmark & - & A_\rkg / P_2 & (\rkg-1)^2 + 5\\ 
   &&&&(12) & \{ 3, \rkg \} & \emptyset & - & - & - & A_\rkg / P_1 & (\rkg-1)^2 + 4\\
 &&\rkg = 3&& (23) & \emptyset & \emptyset & - & - & - & A_3 / P_2 & 9\\
 		& P_{1,3} & \rkg \geq 4 & 3\rkg - 4 & (31) & \{ 2, 4, \rkg \} & \emptyset & - & - & - & - & \rkg^2 - 3\rkg + 9\\
 	&&&& (12) & \{ 2, \rkg \} & \emptyset & - & - & - & A_\rkg / P_1 & (\rkg-1)^2 + 4\\
		&& \rkg = 4 && (34) & \{ 2 \} & \emptyset & - & - & - & A_4 / P_3 & 14\\
         & P_{1,s}& 4 \leq s \leq \rkg-1 & \rkg s - (s-1)^2 & (12) & \{ 2,3, \rkg\} & \{ s \} & A_{\rkg - 4} / P_{s-3} & \checkmark & - & A_\rkg / P_1 & (\rkg-1)^2 + 4\\
         &&&&(1,s) & \{ 2, s-1, s+1, \rkg \} & \emptyset & - &- & - & - & (\rkg-2)s + (\rkg-s)^2 + 6\\
         && 4 \leq s = \rkg-1 & & (\rkg-1,\rkg) & \{ \rkg-2 \} & \emptyset & - & - & - & A_\rkg / P_{\rkg-1} & (\rkg-1)^2 + 5\\
         & P_{2,s} & 4 \leq s \leq \rkg-1 & s(\rkg+3-s) - 4 & (21) & \{ 3, \rkg \} & \{ s \} & A_{\rkg - 4} / P_{s-3} & \checkmark & - & A_\rkg / P_2 & (\rkg-1)^2 + 5 \\
         & & 4 \leq s = \rkg-1 & & (\rkg-1,\rkg) & \{ 1, \rkg-2 \} & \{ 2 \} & A_{\rkg-4} / P_1 & \checkmark & - & A_\rkg / P_{\rkg-1} & (\rkg-1)^2 + 5\\
         & P_{s,s+1} & 2 \leq s \leq \rkg-2 & (\rkg-s)(s+1) + s & (s,s+1) & \{ 1, s-1,s+2,\rkg \} & \{ s+1 \} & A_1 / P_1 & \checkmark & - & A_\rkg / P_s & \rkg(\rkg-1) - s(\rkg-s) + 6\\
         &&&& (s+1,s) & \{ 1, s-1,s+2,\rkg \} & \{ s \} & A_1 / P_1 & \checkmark & - & A_\rkg / P_{s+1} & \rkg(\rkg-1) - s(\rkg-s) + 6\\
         && s= 2 \leq \rkg - 2 & & (21) & \{ \rkg \} & \emptyset & - & - & - & A_\rkg / P_2 & (\rkg-1)^2 + 5\\
         & P_{1,s,\rkg} & 3 \leq s \leq \rkg-2 & \rkg - 1 + (\rkg+1-s) s & (1,\rkg) & \{ 2, \rkg-1 \} & \{ s \} & A_{\rkg - 4} / P_{s-2} &  \checkmark & - & A_\rkg / P_{1,\rkg} & (\rkg-1)^2 + 4\\
         & P_{1,2,3} & \rkg \geq 3 & 3(\rkg - 1) & (21) & \{ \rkg \} \backslash \{ 3 \} & \{ 1 \} & A_1 / P_1 & \checkmark & - & A_\rkg / P_2 & (\rkg-1)^2 + 5\\
         &&&& (12) & \{ \rkg \} \backslash \{ 3 \} & \emptyset & - & - & - & A_\rkg / P_1 & (\rkg-1)^2 + 4\\
	& & \rkg = 3 &6 & (23) & \emptyset & \{ 3 \} & A_1 / P_1 & \checkmark & - & A_3 / P_2 & 9\\
	&&&& (32) & \emptyset & \emptyset & - & - & - & A_3 / P_3 & 8\\
	&&&& (13) & \emptyset & \emptyset & - & - & - & A_3 / P_{1,3} & 8\\
         & P_{1,2,s} & 4 \leq s < \rkg & s(\rkg+3-s)-3 & (21) & \{ 3, \rkg \} & \{ 1, s \} & A_1 / P_1 \times A_{\rkg - 4} / P_{s-3} & \checkmark & - & A_\rkg / P_2 & (\rkg-1)^2 + 5\\
         &&&& (12) & \{ 3,\rkg \}& \{ s \} & A_{\rkg - 4} / P_{s-3} & \checkmark & - & A_\rkg / P_1 & (\rkg-1)^2 + 4\\ 
         & P_{1,2,\rkg} & \rkg \geq 4 & 3(\rkg - 1) & (21) & \{ 3 \}& \{ 1 \} & A_1 / P_1 & \checkmark & - & A_\rkg / P_2 & (\rkg-1)^2 + 5\\
         &&&& (12) & \{ 3 \}& \emptyset & - & - & - & A_\rkg / P_1 & (\rkg-1)^2 + 4\\ 
         &&&& (1,\rkg) & \{ \rkg-1\} & \emptyset & - & - & - & A_\rkg / P_{1,\rkg} & (\rkg-1)^2 + 4\\
         & P_{1,2,3,s} & 4 \leq s < \rkg & 3(\rkg-1) + (\rkg+1-s)(s-3) & (21) & \{ \rkg \} & \{ 1,s \} & A_1 / P_1 \times A_{\rkg - 4} / P_{s-3} & \checkmark & - & A_\rkg / P_2 & (\rkg-1)^2 + 5\\ 
         & P_{1,2,3,\rkg} & \rkg \geq 4 & 4\rkg - 6 & (21) & \emptyset & \{ 1 \} & A_1 / P_1 & \checkmark & - & A_\rkg / P_2 & (\rkg-1)^2 + 5\\ 
         && \rkg = 4 & & (34) & \emptyset & \{ 4 \} & A_1 / P_1 & \checkmark & - & A_\rkg / P_3 & 14\\
         & P_{1,2,s,t} & 4 \leq s < t < \rkg & \begin{array}{c} 2\rkg-1 + (\rkg+1-s)(s-2) \\ + (\rkg+1 - t)(t-s) \end{array}& (21) & \{ 3, \rkg \}& \{ 1,s,t \} & A_1 / P_1 \times A_{\rkg - 4} / P_{s-3,t-3} & \checkmark & \checkmark & A_\rkg / P_2 & (\rkg-1)^2 + 5 \\
         & P_{1,2,s,\rkg} & 4 \leq s < \rkg & 3(\rkg-1) + (\rkg-s)(s-2) & (21) & \{3\} & \{ 1, s \} & A_1/P_1 \times A_{\rkg - 4} / P_{s-3}  & \checkmark & - & A_\rkg / P_2 & (\rkg-1)^2 + 5\\
         & & 4\leq s=\rkg-1 & & (\rkg-1,\rkg) & \{ \rkg-2 \} & \{ 2, \rkg \} & A_{\rkg -4} / P_1 \times A_1 / P_1 & \checkmark & - & A_\rkg / P_{\rkg-1} & (\rkg-1)^2 + 5\\ \hline\hline
 B_\rkg & P_3 & \rkg \geq 4 & 6\rkg - 12 & (32) & \{ 1, 4 \} & \emptyset & - & - & - & - & 2\rkg^2 - 7\rkg + 16\\
 	&& \rkg = 3 && (32) & \{ 1, 2 \} & \emptyset & - & - & - & - & 11 \\
	& P_\rkg & \rkg \geq 4 & \binom{\rkg + 1}{2} & (\rkg, \rkg - 1) & \{ 2, \rkg - 2, \rkg -1 \} & \emptyset & - & - & - & - & \frac{\rkg(3\rkg - 7)}{2} + 10\\[0.02in]
 	& P_{1,2} & \rkg \geq 3 & 4\rkg - 4 & (12) & \{ 3 \} & \{ 2 \} & A_1 / P_1 & \checkmark & - & B_\rkg / P_1 & 2\rkg^2 - 5\rkg + 9\\
	&&&& (21) & \{ 3 \} & \emptyset & - & - & - & B_\rkg / P_2 & 2\rkg^2 - 5\rkg + 8\\
 	& P_{2,3} & \rkg \geq 4 & 6\rkg - 10 & (32) & \{ 1,4 \} & \{ 2 \} & A_1 / P_1 & \checkmark & - & B_\rkg / P_3 & 2\rkg^2 - 7\rkg + 16\\
	&& \rkg = 3 & & (32) & \{ 1 \} & \emptyset &-&-& -& B_3 / P_3 & 11 \\
 	& P_{1,3} & \rkg = 3 & 8 & (32) & \{ 2 \} & \emptyset &-&-& -& B_3 / P_3 & 11 \\
	& P_{1,2,3}& \rkg = 3 & 9 & (32) & \emptyset & \emptyset &-& -&-& B_3 / P_3 & 11 \\ \hline
\end{array}$
 \caption{Submaximal symmetry dimensions for Yamaguchi-nonrigid geometries in type $A$ and $B$, excluding 1-graded \& parabolic contact geometries}
 \label{F:AB}
\end{table}
\end{landscape}
\end{tiny}
\begin{tiny}
\begin{landscape}
 \begin{table}[h]
 $\begin{array}{|c|c|c|c|c|c|c|c|c|c|@{}c@{}|c|} \hline
 G & P & \mbox{Range} & \dim(G/P) & w & J_w & I_w & (\overline\fg,\overline\fp) & \fa_1 & \fa_2 & \begin{tabular}{c}Twistor\\space type\end{tabular} & \fS_w = \fU_w\\ \hline\hline
 C_\rkg & P_2 & \rkg \geq 3 & 4\rkg - 5 & (21) & \{ 3 \} & \emptyset & - & - & - & - & 2\rkg^2 - 5\rkg + 9\\
 	&&\rkg = 3 &  & (23) & \{ 1, 3 \} & \emptyset & - & - & - & - & 11\\
	& P_{\rkg - 1} & \rkg \geq 4 & \frac{(\rkg+4)(\rkg-1)}{2} & (\rkg - 1, \rkg) & \{ 1, \rkg -2, \rkg\} & \emptyset & - & - & - & - & \frac{\rkg(3 \rkg - 5)}{2}+5\\[0.02in]
	& P_{1,2} & \rkg \geq 3 & 4\rkg - 4 & (21) & \{ 3 \} & \{ 1 \} & A_1 / P_1 & \checkmark & - & C_\rkg / P_2 & 2\rkg^2-5\rkg+9\\
	&&&& (12) & \{ 3 \} & \emptyset & - & - & - & C_\rkg / P_1 & 2\rkg^2-5\rkg+8\\
 	& P_{1,\rkg} & \rkg \geq 3 &\frac{\rkg^2 + 3\rkg - 2}{2} & (1,\rkg) & \{ 2, \rkg - 1 \} & \emptyset & - & - & - & -& \frac{\rkg(3\rkg-5)}{2} + 5\\[0.02in]
	&& \rkg \geq 4 && (12) & \{ 2, 3\} & \{ \rkg \} & C_{\rkg - 3} / P_{\rkg - 3} &\checkmark& - & C_\rkg / P_1& 2\rkg^2 - 5\rkg + 8\\ 
	&& \rkg = 3 && (12) & \{ 2 \} & \emptyset & - & - & - & C_3 / P_1 & 11 \\
 	& P_{2,\rkg} & \rkg \geq 4 & \frac{\rkg^2 + 5\rkg - 8}{2} & (21) & \{ 3 \} & \{ \rkg \} & C_{\rkg - 3} / P_{\rkg - 3} &\checkmark & - & C_\rkg / P_2 & 2\rkg^2-5\rkg+9\\
	&& \rkg = 3 && (21) & \emptyset & \emptyset & - & - & - & C_3 / P_2 & 12\\
	&& \rkg = 3 && (23) & \{ 1 \} & \emptyset & - & - & - & C_3 / P_2 & 11 \\
        & P_{\rkg - 1, \rkg} & \rkg \geq 4 & \frac{\rkg^2 + 3\rkg - 2}{2} & ( \rkg -1, \rkg ) & \{ 1, \rkg - 2 \} & \emptyset & - & - & - & C_\rkg / P_{\rkg - 1} & \frac{\rkg(3\rkg - 5)}{2} + 5\\[0.02in]
         & P_{1,2,3} & \rkg \geq 3 & 6\rkg - 9 & (21) & \emptyset & \{ 1 \} & A_1 / P_1 & \checkmark & - & C_\rkg / P_2 & 2\rkg^2 - 5\rkg + 9\\
         & P_{1,2,s} & 4\leq s<\rkg & \frac{-6 - 3s^2 + 5s + 4s \rkg}{2} & (21) & \{ 3 \} & \{ 1, s\} & A_1 / P_1 \times C_{\rkg - 3} / P_{s - 3} & \checkmark & \checkmark & C_\rkg / P_2 & 2\rkg^2-5\rkg+9\\
         & P_{1,2,\rkg} & \rkg \geq 4 &\frac{\rkg^2 + 5\rkg -6}{2} & (21) & \{ 3 \} & \{ 1, \rkg \} & A_1 / P_1 \times C_{\rkg - 3} / P_{\rkg - 3} & \checkmark & - & C_\rkg / P_2 & 2\rkg^2-5\rkg+9\\ \hline\hline
 D_\rkg & P_3 & \rkg \geq 5 & 6\rkg - 15 & (32) & \mycase{\{ 1, 4 \}, & \rkg \geq 6; \\ \{ 1, 4, 5 \}, & \rkg = 5} & \emptyset & - & - & - & - & 2\rkg^2 - 9\rkg + 20\\	
 	& P_{1,2} & \rkg \geq 4 & 4\rkg - 6 & (12) & \mycase{\{ 3 \}, & \rkg \geq 5;\\ \{ 3, 4 \}, & \rkg = 4} & \{ 2 \} & A_1 / P_1 & \checkmark & - & D_\rkg / P_1 & 2\rkg^2-7\rkg + 12\\
	&&&& (21) &  \mycase{\{ 3 \}, & \rkg \geq 5;\\ \{ 3, 4 \}, & \rkg = 4} & \emptyset & - & - & - & D_\rkg / P_2 & 2\rkg^2-7\rkg + 11\\
         & P_{1,\rkg} & \rkg \geq 5 & \frac{(\rkg+2)(\rkg - 1)}{2} & (12) & \{ 3 \} & \{ \rkg \} & D_{\rkg-3} / P_{\rkg-3} & \checkmark & - & D_\rkg / P_1 & 2\rkg^2-7\rkg + 12\\
	&& \rkg = 4 && (12) & \{ 3 \} & \emptyset & - & - & - & D_4 / P_1 & 16\\
	&&&& (42) & \{ 3 \} & \emptyset & - & - & - & D_4 / P_4 & 16\\
  	& P_{2,3} & \rkg \geq 5 & 6\rkg - 13 & (32) & \mycase{\{ 1, 4 \}, & \rkg \geq 6; \\ \{ 1, 4, 5 \}, & \rkg = 5} & \{ 2 \} & A_1 / P_1 & \checkmark & - & D_\rkg / P_3 & 2\rkg^2 - 9\rkg + 20\\
	& P_{1,2,\rkg} & \rkg \geq 5 &  \frac{\rkg^2 + 3\rkg - 6}{2} & (12) & \{ 3 \} & \{ 2, \rkg \} & A_1 / P_1 \times D_{\rkg - 3} / P_{\rkg - 3} & \checkmark & - & D_\rkg / P_1 & 2\rkg^2-7\rkg + 12\\
 	&& \rkg = 4 && (12) & \{ 3 \} & \{ 2 \} & A_1 / P_1 & \checkmark & - & D_4 / P_1 & 16\\ 
 	&&&& (42) & \{ 3 \} & \{ 2 \} & A_1 / P_1 & \checkmark & - & D_4 / P_4 & 16\\ 
	\hline\hline
  G_2 & P_1 & - & 5 & (12) & \{ 2 \} & \emptyset &-&-& - & -& 7\\
	& P_{1,2} & - & 6 & (12) & \emptyset & \emptyset &-&-& -& G_2 / P_1 & 7\\ \hline
 \end{array}$
 \caption{Submaximal symmetry dimensions for Yamaguchi-nonrigid geometries in type $C$, $D$, $G$, excluding 1-graded \& parabolic contact geometries}
 \label{F:CDG}
 \end{table}
 \end{landscape}
 \end{tiny}
 
 
 \section{Direct proof of prolongation-rigidity for conformal geometry}
 
 Consider $\bbC^{n+2}$ with an anti-diagonal symmetric $\bbC$-bilinear form $g$, i.e. $g_{ij} = \delta_{i, j'}$, where $j' := n-j+1$.  Then 
 \[
 \fso(n+2,\bbC) = \l\{ \mat{c|c|c}{a & \omega & 0\\ \hline v & A & -\omega'\\ \hline 0 & -v' & -a} : a \in \bbC, A \in \fso(n,\bbC), \omega \in (\bbC^n)^*, v \in \bbC^n \r\}
 \]
 consists of matrices skew with respect to the anti-diagonal.  We have $\fg_0 = \fz(\fg_0) \op \fg_0^{ss} = \bbC \op \fso(n,\bbC)$. Symbolically, the brackets are:
 \begin{enumerate}
 \item $\fg_0 \times \fg_0 \to \fg_0$: $[(a,A), (b,B)] = (0,[A,B])$.
 \item $\fg_0 \times \fg_1 \to \fg_1$: $[(a,A),\omega] = \omega(a-A)$.
 \item $\fg_0 \times \fg_{-1} \to \fg_{-1}$: $[(a,A),v] = (A-a)v$.
 \item $\fg_1 \times \fg_{-1} \to \fg_0$: $[\omega,v] = (\omega v, \omega'v' - v\omega)$.
 \end{enumerate}

 The $\fg_0$-module decomposition of $H^2(\fg_-,\fg)$ is given by
 \[
  H^2(\fg_-,\fg) \cong
   \l\{ \begin{array}{lll} 
   \l(\bigwedge^2 \fg_1 \ot \fg_0^{ss} \r)_0, & n \geq 4 & \mbox{(Weyl tensors)};\\ 
    \yng(2,1)(\fg_1) & n = 3 & \mbox{(Cotton tensors)}. \end{array} \r.
 \]
 If $n \neq 4$, $H^2(\fg_-,\fg)$ is $\fg_0$-irreducible; if $n=4$, it splits into self-dual and anti-self-dual submodules.  
 
 Let $\{ e_i \}$ and $\{ \omega^j \}$ be the standard (dual) bases of $\bbC^n$ and $(\bbC^n)^*$ respectively.  Define $E_i{}^j = e_i \ot \omega^j$ and $R_\ell{}^k = E_\ell{}^k - E_{k'}{}^{\ell'} \in \fg_0^{ss}$.  Note $R_\ell{}^k = -R_{k'}{}^{\ell'}$, and we require $k \neq \ell'$.  
 Let $\omega^{ijk}{}_\ell = \omega^i \wedge \omega^j \ot R_\ell{}^k$.
 An element $W = W_{ijk}{}^\ell \omega^{ijk}{}_\ell \in \bigwedge^2 \fg_1 \ot \fg_0^{ss}$ has the index symmetries $ W_{ijk}{}^\ell = -W_{jik}{}^\ell$, $W_{ijk}{}^\ell = -W_{ij\ell'}{}^{k'}$.

 Letting $A = \sum_{i < j'} A^i_j R_i{}^j$ and using the bracket relations above, the $\fg_0$-action on $\bigwedge^2 \fg_1 \ot \fg_0^{ss}$ is
 \begin{align*}
 \wW &:= (a,A) \cdot W 
 = \l(2a W_{ijk}{}^\ell - W_{rjk}{}^\ell A^r_i - W_{irk}{}^\ell A^r_j - W_{ijr}{}^\ell A^r_k + W_{ijk}{}^r A^\ell_r \r) \omega^{ijk}{}_\ell\\
 \wW_{ijk}{}^\ell 
 &= 2a W_{ijk}{}^\ell - W_{rjk}{}^\ell A^r_i - W_{irk}{}^\ell A^r_j + \half\l( - W_{ijr}{}^\ell A^r_k + W_{ijr}{}^{k'} A^r_{\ell'} + W_{ijk}{}^r A^\ell_r - W_{ij\ell'}{}^r A^{k'}_r\r),
 \end{align*}
 where we have accounted for ``anti-diagonal skew-symmetry'' in the last line above.

 Similarly when $n=3$, given $C = C_{ijk} \omega^i \wedge \omega^j \ot \omega^k$, and $\wC = (a,A) \cdot C$, we have
 \[
 \wC_{ijk} = 3aC_{ijk} - C_{rjk} A^r_i - C_{irk} A^r_j - C_{ijr} A^r_k.
 \]
 \begin{lemma} Conformal geometry in dimension $\geq 3$ is prolongation-rigid.
 \end{lemma}
 
 \begin{proof} Let $\omega \in \fg_1$ and $\fk = [\omega,\fg_{-1}]$.  Use the $G_0 = CO(\fg_{-1})$ action to normalize $\omega$ (over $\bbC$) to one of two values, namely a null vector or a non-null vector.  Suppose first that $n \geq 4$.
 
 {\em Case 1:} $\omega$ is null.  Take $\omega = \omega^1$.  Then $\fk = \{ (v^1, -v^i R_i{}^1) : v \in \fg_{-1} \}$ and $\dim(\fk) = n-1$.  (Note $R_n{}^1 = 0$.)  Letting $(a,A) = (1,E_n^n - E_1^1)$, we obtain
 \begin{align*}
 0 &= \wW_{ijk}{}^\ell = W_{ijk}{}^\ell \l(2  - A^i_i - A^j_j + \half\l( - A^k_k + A^\ell_\ell - A^{\ell'}_{\ell'}  +  A^{k'}_{k'}\r)\r).
 \end{align*}
 Since $A^n_n - A^1_1 = 2$, then the only possible nonzero components are $W_{njn}{}^\ell = -W_{nj\ell'}{}^1 = -W_{jnn}{}^\ell$ for $1 < j,\ell < n$, i.e. $ W = W_{njn}{}^\ell \omega^{njn}{}_\ell$.  Now let $(a,A) = (0, R_c{}^1)$, $1 < c < n$.  Note $A^c_1 = 1 = -A^n_{c'}$.  Then for $1 < j, \ell < n$ with $j \neq c'$, we have $ 0 = \tilde{W}_{c'jn}{}^\ell = - W_{rjn}{}^\ell A^r_{c'} = W_{njn}{}^\ell$.  Thus, $W = 0$.
 
  {\em Case 2:} $\omega$ is non-null.  Take $\omega = \omega^1 + \omega^n$.  Then $\fk = \{ (v^1+v^n, v^{i'} (R_1{}^i + R_n{}^i)) : v \in \fg_{-1} \}$ and $\dim(\fk) = n$.  In particular, note that $(1,0_{n \times n}) \in \fk$, and $0 = \wW = (1,0_{n\times n}) \cdot W = 2W$.

 Now consider $n=3$.  If $\omega$ is null, then letting $(a,A) = (1,E_n{}^n - E_1{}^1)$ yields $0 = \wC_{ijk} = C_{ijk}( 3 - A^i_i - A^j_j - A^k_k )$, so $C_{ijk} = 0$ since $i \neq j$.  If $\omega$ is not null, then as above $(1,0_{n \times n}) \in \fk$, so again $C=0$.
 \end{proof}
 
 
 \section{NPR geometries}
 \label{App:NPR}
 
 The classification of NPR geometries in Table \ref{F:NPR} follows immediately from the results in Appendix \ref{App:Submax}, but as we do not show the extensive calculations there, we sketch here an independent proof.
 
 By Proposition \ref{P:PR} and Corollary \ref{C:g-PR}, if $G$ is of exceptional type or $|I_\fp| = 1$, then $G/P$ is PR.  Thus, we henceforth assume that $G/P$ is a classical complex flag variety, with $|I_\fp| \geq 2$ (and $G$ simple).   Regularity $Z(-w\cdot\lambda_\fg) \geq 1$ is equivalent to \eqref{E:reg} which simplifies in each case of \eqref{E:F1} to:
  \begin{align}
 \l\{\begin{array}{rll} 
 {\rm (a):} & Z(\lambda_\fg) \leq r_j - (r_k +1) c_{kj};\\
 {\rm (b):} & Z(\lambda_\fg) \leq r_j + (r_k +1)(1 - c_{kj});\\
 {\rm (c):} & Z(\lambda_\fg) \leq 2(r_i +1).
 \end{array} \r. \label{E:F2}
 \end{align}
 The proof is simply an analysis of \eqref{E:F1} and \eqref{E:F2}.  We leave the $A_\rkg$ and $C_\rkg$ cases to the reader.
 
 \framebox{{\bf $B_\rkg$ or $D_\rkg$ case}}:  Here, $c_{kj} \in \{ 0,-1\}$ iff $\fg = D_\rkg$.  Also, $0 \leq r_j + r_k \leq 1$ (since $j \neq k$) using:
 \begin{align} \label{E:BD-lambda}
 \lambda_\fg= \lambda_2 = \l\{ \begin{array}{lll} 
 \alpha_1 + 2\alpha_2 + ... + 2\alpha_\rkg, & \fg = B_\rkg & (\rkg \geq 3);\\
 \alpha_1 + 2\alpha_2 + ... + 2\alpha_{\rkg-2} + \alpha_{\rkg-1} + \alpha_\rkg, & \fg = D_\rkg & (\rkg \geq 4).
\end{array} \r.
 \end{align}
 
 \begin{enumerate}[(a)]
 \item Assume $c_{kj} = -2$, so $\fg = B_\rkg$ and $(k,j) = (\rkg-1,\rkg)$.  From \eqref{E:F1}, $c_{ki} = 0$, so $\rkg \geq 4$, $r_j = r_k = 0$.  From \eqref{E:F2}, $|I_\fp| = Z(\lambda_\fg) = 2$, which contradicts \eqref{E:BD-lambda}.  Thus, $c_{kj} \in \{ 0, -1 \}$.  Then $r_j - (r_k+1) c_{kj} \leq 2$, so from \eqref{E:F2}, $|I_\fp| = Z(\lambda_\fg) = 2$ and exactly one of $j,k$ is $2$.  If $\fg = B_\rkg$ or $j=2$, then $Z(\lambda_\fg) \geq 3$, so $\fg = D_\rkg$ and $k=2 \not\in I_\fp$.  Hence, \framebox{$D_\rkg / P_{1,\rkg}$, $\rkg \geq 5$; $w = (12)$; $i=\rkg$}
 
 \item Assume $\fg = B_\rkg$.  From \eqref{E:F2}, $|I_\fp| \geq 3$, so $Z(\lambda_\fg) \geq 5$.  But since $0 \leq r_j + r_k \leq 1$ and $c_{kj} \leq -1$, then $r_j + (r_k+1)(1-c_{kj}) \geq 5$ only if $r_j = 0$, $r_k = 1$, $c_{kj} = -2$, and hence $(k,j) = (\rkg-1,\rkg)$.  From \eqref{E:F1}, $c_{ki} = 0$, so $\rkg \geq 4$, $r_j = r_k = 0$.  From \eqref{E:F2}, $Z(\lambda_\fg) = 3$, a contradiction.
  
 Thus, $\fg = D_\rkg$.  Assuming $r_j = 1$, then $r_k = 0$ and \eqref{E:F2} implies $|I_\fp|=Z(\lambda_\fg) = 3$.  But $j = 2\in I_\fp$, so by \eqref{E:BD-lambda}, $Z(\lambda_\fg) \geq 4$, a contradiction.  Hence, $r_j = 0$, and \eqref{E:F2} implies $r_k = 1$, $c_{kj} = -1$, and $3 \leq |I_\fp| \leq Z(\lambda_\fg) \leq 4$.  From \eqref{E:BD-lambda}, $I_\fp = \{ 1,2, \rkg \}$, so: \framebox{$D_\rkg / P_{1,2,\rkg}$, $\rkg \geq 5$; $w = (12)$; $i=\rkg$}
 
 \item Assume $r_i = 0$, so from \eqref{E:F2}, $|I_\fp| = Z(\lambda_\fg) = 2$.  Using \eqref{E:BD-lambda}, this is impossible given $c_{ij} = -1$ by \eqref{E:F1}.  Thus, $r_i = 1$, and $2 \leq |I_\fp| \leq Z(\lambda_\fg) \leq 4$.  Keeping in mind $c_{ij} = -1$, we have:
  \begin{itemize}
 \item \framebox{$B_\rkg / P_{1,2}$, $\rkg \geq 3$; $w = (12)$; $i=2$} and \framebox{$D_\rkg / P_{1,2}$, $\rkg \geq 4$; $w = (12)$; $i=2$}
 \item \framebox{$B_\rkg / P_{2,3}$, $\rkg \geq 4$; $w = (32)$; $i=2$} and \framebox{$D_\rkg / P_{2,3}$, $\rkg \geq 4$; $w = (32)$; $i=2$}
 \item \framebox{$D_\rkg / P_{1,2,\rkg}$, $\rkg \geq 5$; $w = (12)$; $i=2$}.  Also, \framebox{$D_4 / P_{1,2,4}$; $w = (12)$ or $(42)$; $i=2$}
 \end{itemize}
 \end{enumerate}

\end{document}